\documentclass[10pt,oneside,english]{amsart}
\usepackage[T1]{fontenc}
\usepackage[latin9]{luainputenc}
\usepackage{geometry}
\usepackage{color}
\usepackage{babel}
\usepackage{mathrsfs}
\usepackage{amsbsy}
\usepackage{amstext}
\usepackage{amsthm}
\usepackage{amssymb}
\usepackage[unicode=true,pdfusetitle,
 bookmarks=true,bookmarksnumbered=false,bookmarksopen=false,
 breaklinks=false,pdfborder={0 0 1},backref=false,colorlinks=false]
 {hyperref}

\makeatletter

\providecommand{\tabularnewline}{\\}

\theoremstyle{plain}
\newtheorem{thm}{\protect\theoremname}
\theoremstyle{remark}
\newtheorem{rem}[thm]{\protect\remarkname}
\theoremstyle{plain}
\newtheorem{conjecture}[thm]{\protect\conjecturename}
\theoremstyle{remark}
\newtheorem*{acknowledgement*}{\protect\acknowledgementname}
\theoremstyle{plain}
\newtheorem{lem}[thm]{\protect\lemmaname}
\theoremstyle{plain}
\newtheorem{fact}[thm]{\protect\factname}
\theoremstyle{definition}
\newtheorem{defn}[thm]{\protect\definitionname}
\theoremstyle{definition}
\newtheorem{example}[thm]{\protect\examplename}
\theoremstyle{plain}
\newtheorem{cor}[thm]{\protect\corollaryname}
\theoremstyle{remark}
\newtheorem*{rem*}{\protect\remarkname}

\@ifundefined{date}{}{\date{}}
\usepackage{tikz}
\usepackage{pgfplots}
\usetikzlibrary{matrix,arrows,decorations.pathmorphing}
\usepackage{verbatim}
\usepackage{pgf}
\usepackage{color}
\definecolor{BLACK}{RGB}{0, 0, 0}
\definecolor{green}{RGB}{0, 0, 0}
\definecolor{cyan}{RGB}{0,100, 0}
\definecolor{yellow}{RGB}{0,100,0}
\definecolor{red}{RGB}{0,100,0}
\definecolor{BLUE}{RGB}{0,100,0}

\usetikzlibrary[patterns]
\definecolor{blue}{RGB}{0,100,0}
\colorlet{linkequation}{BLACK}

\usepackage{scalerel,stackengine}
\stackMath
\newcommand\widecheck[1]{%
\savestack{\tmpbox}{\stretchto{%
  \scaleto{%
    \scalerel*[\widthof{\ensuremath{#1}}]{\kern-.6pt\bigwedge\kern-.6pt}%
    {\rule[-\textheight/2]{1ex}{\textheight}}
  }{\textheight}%
}{0.5ex}}%
\stackon[1pt]{#1}{\scalebox{-1}{\tmpbox}}%
}

\makeatother

\providecommand{\acknowledgementname}{Acknowledgement}
\providecommand{\conjecturename}{Conjecture}
\providecommand{\corollaryname}{Corollary}
\providecommand{\definitionname}{Definition}
\providecommand{\examplename}{Example}
\providecommand{\factname}{Fact}
\providecommand{\lemmaname}{Lemma}
\providecommand{\remarkname}{Remark}
\providecommand{\theoremname}{Theorem}

\begin{document}
\title{The SO(3) Vortex Equations over Orbifold Riemann Surfaces}
\author{\textcolor{black}{Mariano Echeverria}}
\begin{abstract}
We study the general properties of the moduli spaces of $SO(3)$ vortices
over orbifold Riemann surfaces and use these to characterize the solutions
of the $SO(3)$ monopole equations on Seifert manifolds following
in the footsteps of Mrowka, Ozsváth and Yu.

We also study the solutions to the $SO(3)$ monopole equations on
$S^{1}\times\varSigma$ in order to motivate the construction of a
version of monopole Floer homology, which we call framed monopole
Floer homology, in analogy with the construction given by Kronheimer
and Mrowka for the case of instanton Floer homology. 

Finally, the $SO(3)$ vortex moduli spaces provide a nice toy model
for recent work due to Feehan and Leness regarding the study of a
natural Morse-Bott function on the moduli space of $SO(3)$ monopoles
over Kahler manifolds. In particular, we compute the Morse-Bott indices
of this function.

\end{abstract}

\maketitle

\section{Introduction}

\subsection{$SO(3)$ vortices and Morse-Bott Theory}

\ 

The study of vortices on Riemann surfaces has been so fruitful since
the work of Jaffe and Taubes \cite{MR614447} that it would be entirely
reasonable to assume that on this topic there is nothing new under
the Sun. Indeed, the vortex equations have been studied and generalized
in multiple directions (see for example \cite{MR1396775,MR1124279,MR1085139,MR1246476,MR1273268,MR1354000,MR1423193,MR3454234,MR3513572,MR1086749}),
and they have proved incredibly useful in 3- and 4-manifold topology
due to its interaction with the Seiberg-Witten equations \cite{MR1306021,MR1306023,MR1324704,MR2350473}.
However, by scraping the bottom of the barrel one can still find some
new things as we now explain. 

First of all,\textbf{ }an \textbf{$SO(3)$ }vortex\textbf{ }consists
of a pair $(C,\varUpsilon)$, where $C$ is a unitary connection on
a $U(2)$ bundle $E$ over a Riemann surface $\varSigma$, and $\varUpsilon$
is a section of $E$. The connection $C$ must induce a predetermined
connection $C^{\det}$ on $\det E$, which is why we are really studying
the $SO(3)$ vortex equations instead of the $U(2)$ vortex equations.
The equations $(C,\varUpsilon)$ must satisfy are the equations (\ref{eq:SO(3) vortex equations}),
that is, 
\begin{align*}
*F_{C}^{0}-i\left[\varUpsilon\varUpsilon^{*}-\frac{1}{2}|\varUpsilon|^{2}I_{E}\right]=0\\
\bar{\partial}_{C}\varUpsilon=0
\end{align*}
 The meaning of these equations is explained in Section (\ref{sec:the--Vortex SO(3)}),
but they probably already look familiar to most readers. By studying
the solutions to the $SO(3)$ vortex equations modulo gauge transformations
one obtains the $SO(3)$ vortex moduli space $\mathcal{M}(\varSigma,E)$. 

The study of these equations has a long history. For example, in his
PhD thesis (see also \cite[Section 5]{MR1265143}), García-Prada studied
the moduli space of non-abelian vortices from a gauge theoretic and
algebraic geometric point of view. From the latter perspective, these
moduli spaces can be identified with the moduli spaces of stable pairs,
and García-Prada interpreted stable pairs of arbitrary rank on any
compact Kahler manifold $X$ as stable bundles on $X\times\mathbb{CP}^{1}$,
which allowed him to encode the stability parameter of the moduli
space of stable pairs in terms of the stability condition in the sense
of Mumford of the corresponding stable bundle.

Around the same time, Bradlow and Daskalopoulous \cite{MR1124279,MR1250254}
studied the moduli spaces of stable pairs for the case of a Riemann
surface from a more analytic point of view, and a bit later the moduli
spaces for the rank two case were studied by Thaddeus \cite{MR1273268}
and Bertram \cite{MR1297851} from the algebraic geometric point of
view. 

The case of the $U(2)$ vortex equations was also studied extensively
from a gauge-theoretic perspective by Bradlow and Daskalopoulous \cite{MR1124279,MR1250254},
as well as Bradlow, Daskalopoulous and Wentworth \cite{MR1124279}.
More importantly, the moduli space of $SO(3)$ vortices for Riemann
surfaces without marked points is a particular case of the general
framework developed in the paper by García-Prada, Gothen and Mundet
i Riera \cite{Prada-Higgs}.

For our purposes it was important to work with the more general setup
of orbifold Riemann surfaces, using the ideas of Furuta and Steer
\cite{MR1185787} for the case of the Yang-Mills equations (which
is the orbifold version of Atiyah and Bott's seminal work \cite{MR702806}).
We use gauge-theoretic techniques instead of geometric invariant theory
(GIT), which is also used to study these moduli spaces from the algebraic
geometric point of view. In fact, Bertram describes briefly the construction
of the moduli space of stable parabolic pairs \cite[Section 3]{MR1297851}
(using GIT techniques).

One reason for studying the moduli space of $SO(3)$ vortices on Riemann
surfaces is that they provide useful toy models complementing recent
work of Feehan and Leness, who studied the $SO(3)$ monopole moduli
spaces over Kahler surfaces from a Morse-theoretic point of view.
Following the work of Hitchin \cite{MR887284}, Feehan and Leness
\cite{Feehan-Leness[Virtual]} studied the $L^{2}$ norm of the spinor
of an $SO(3)$ monopole as a Morse-Bott function on the moduli space
of $SO(3)$ monopoles over Kahler manifolds, which in particular means
they had to compute the Morse-Bott indices of the critical sets of
this function. We do the analogous computations for $SO(3)$ vortices
over a Riemann surfaces. Since we are also finding these formulas
for the case of orbifolds, the work \cite{MR1375314} by Nasatyr and
Steer also served as an important model. We now explain how the computation
of these indices is done.

As mentioned earlier, $\mathcal{M}(\varSigma,E)$ is obtained by studying
the solutions to the $SO(3)$ vortex equations modulo gauge. In this
case, the gauge group used is the determinant-one gauge group $\mathcal{G}^{\det}(E)$,
which is not the gauge group of all unitary automorphisms of $E$.
The upshot of doing this is that there is a residual symmetry on the
moduli space of $SO(3)$ vortices modulo $\mathcal{G}^{\det}(E)$. 

More precisely, there is a circle action on the moduli space obtained
by rescaling the section $\varUpsilon$, in other words, $\varUpsilon\rightarrow e^{i\theta}\varUpsilon$.
Associated to this circle action, there is a moment map $\mu$ which
is essentially the $L^{2}$ norm of the section, that is, $\mu(C,\varUpsilon)=\frac{1}{2}\|\varUpsilon\|_{L^{2}(\varSigma)}^{2}$.
Following Bradlow, Daskalopoulous and Wentworth \cite{MR1124279},
the idea is to study $\mu$ as a Morse-Bott function on the $SO(3)$
vortex moduli space (they studied this for the $U(2)$ vortex moduli
space for the case of a smooth Riemann surface, i.e, without marked
points). 

The critical sets of $\mu$ can be identified with the fixed points
of the circle action, which consists of the moduli space of projectively
flat connections on $\varSigma$, and certain moduli spaces of abelian
$U(1)$ vortices on $\varSigma$. The abelian vortices satisfy equations
(\ref{abelian vortex equation}), which read
\begin{align*}
*F_{C_{L}}-\frac{i}{2}|\alpha|^{2}=*\frac{1}{2}F_{C^{\det}}\\
\bar{\partial}_{C_{L}}\alpha=0
\end{align*}
Here $E$ has reduced as $E=L\oplus(L^{*}\otimes\det E)$, and $C_{L}$
is a connection on $L$, while $\alpha$ is a section of $L$. Via
an easy Chern-Weil argument (still valid in the orbifold case, see
lemma (\ref{abelian vortex equation})) we must have $c_{1}(L)\leq\frac{1}{2}c_{1}(E)$,
where in the orbifold case these Chern classes are in general rational
cohomology classes, and Poincaré duality is used to regard these as
rational numbers. Any $SO(3)$ vortex which is not a projectively
flat connection nor an abelian vortex will be called an \emph{irreducible}
\emph{$SO(3)$ vortex}, following the standard terminology in gauge
theory.

In any case, a theorem due to Frankel \cite{MR131883} shows that
the indices of these critical sets equals the dimension of the subspace
on which the circle action acts with negative weight. Feehan and Leness
recently \cite{Feehan-Leness[Virtual]} computed these indices for
the analogue of $\mu$ in the case of $SO(3)$ monopoles over a Kahler
surface. In section Theorem (\ref{thm:computation indices}) we compute
these indices for the case of an orbifold Riemann surface. 

Before giving the general formula for the index, we remark that if
our Riemann surface is regarded as an orbifold Riemann surface we
then write $\check{\varSigma}=(\varSigma,(p_{1},a_{1}),\cdots,(p_{n},a_{n}))$,
where the marked points $p_{1},\cdots,p_{n}$ have been assigned positive
integers $a_{1},\cdots,a_{n}$. Moreover, an orbifold $U(2)$ bundle
$\check{E}$ over $\check{\varSigma}$ carries some \emph{isotropy
data, }which consists of integers $b_{i}^{\pm}$ at each marked points
satisfying certain conditions which we recall in the next section. 
\begin{thm}
\textbf{General Properties of the $SO(3)$ vortex moduli spaces and
indices of the Morse-Bott function $\mu$: }

a) Suppose that $\check{E}$ is an orbifold $U(2)$ bundle over $\check{\varSigma}=(\varSigma,(p_{1},a_{1}),\cdots,(p_{n},a_{n}))$
with isotropy data $(b_{1}^{\pm},\cdots,b_{n}^{\pm})$. Then if $[C,\varUpsilon]\in\mathcal{M}(\check{\varSigma},\check{E})$
is an irreducible $SO(3)$ vortex, the moduli space is smooth at this
point, and of dimension
\[
\dim_{\mathbb{R}}\mathcal{M}(\check{\varSigma},\check{E})=2\left(g-1+c_{1}(\det\check{E})+n-n_{0}-\sum_{i=1}^{n}\frac{b_{i}^{-}+b_{i}^{+}}{a_{i}}\right)
\]
where $n_{0}=\#\{i\mid b_{i}^{-}=b_{i}^{+}\}$.

b) There is an $S^{1}$ action on $\mathcal{M}(\check{\varSigma},\check{E})$
whose fixed point set consists of the moduli space of projectively
flat connections on $\text{ad}\check{E}$ and the moduli spaces of
abelian orbifold vortices associated to the reductions $\check{E}=\check{L}\oplus(\check{L}^{*}\otimes\det\check{E})$,
with $c_{1}(\check{L})\leq\frac{1}{2}c_{1}(\check{E})$.

If one assumes that $E$ is of odd degree in the smooth case or $\check{E}$
is an odd power of the fundamental orbifold line bundle $\check{L}_{0}$
in the case where the $a_{i}$ are mutually coprime, and also that
\[
c_{1}(\check{E})>2c_{1}(K_{\check{\varSigma}})=2\left(2g-2+n-\sum_{i=1}^{n}\frac{1}{a_{i}}\right)
\]
 where $K_{\check{\varSigma}}$ represents the canonical orbifold
line bundle, then $\mathcal{M}(\check{\varSigma},\check{E})$ is also
smooth at the reducible $SO(3)$ vortices. In fact, $\mathcal{M}(\check{\varSigma},\check{E})$
is a smooth closed Kahler manifold.

In particular, the function $\mu$ is a Morse-Bott function on $\mathcal{M}(\check{\varSigma},\check{E})$
and if the orbifold line bundle $\check{L}$ has isotropy $b_{i}$,
the index of abelian vortex associated to the reduction $\check{E}=\check{L}\oplus(\check{L}^{*}\otimes\det\check{E})$
is 
\[
\text{ind}(\check{E},\check{L})=2\left[g-1+c_{1}(\det\check{E})-2c_{1}(\check{L})+\sum_{i\mid\epsilon_{i}=1}\frac{b_{i}^{+}-b_{i}^{-}}{a_{i}}+n_{-}+\sum_{i\mid\epsilon_{i}=-1}\frac{b_{i}^{-}-b_{i}^{+}}{a_{i}}\right]
\]
where $\epsilon_{i}=1$ if $b_{i}=b_{i}^{+}$, $\epsilon_{i}=-1$
if $b_{i}=b_{i}^{-}$. Here $n_{-}=\#\{i\mid b_{i}=b_{i}^{-}<b_{i}^{+}\}$.
\end{thm}

\begin{rem}
a) The proof of the previous theorem is the main content of Sections
3, 4 and 5. 

b) We point out that in general some assumption on the degree of $\check{E}$
seems necessary. For example, the cokernel of the linearized $SO(3)$
vortex map at a projectively flat connection can basically be identified
with $\text{coker}\bar{\partial}_{C}\simeq H^{1}(\check{E})$ and
the only way to guarantee the vanishing of $H^{1}(\check{E})$ for
a semistable bundle $\check{E}$ is by taking $c_{1}(\check{E})$
sufficiently large.

c) Also, $c_{1}(\check{L})\leq\frac{1}{2}c_{1}(\check{E})$ is a necessary
but not sufficient condition for the appearance of a moduli space
of abelian vortices. One also needs for $\check{L}$ to have the correct
isotropy data and for the background degree of $\check{L}$ to be
non-negative. We explain this terminology in Section (\ref{sec:A-CrashCourse-on})
and illustrate this with the examples we do at the end of the paper. 

d) An easy way to guarantee that the moduli space $\mathcal{M}^{*}(\check{\varSigma},\check{E})$
of irreducible $SO(3)$ vortices is empty is by knowing that there
are no orbifold line bundles $\check{L}$ compatible with the isotropy
data of $\check{E}$ (this is more common that one might expect as
the examples in Section (\ref{sec:-Seifert}) illustrate). The reason
for this is that the function $\mu$ must achieve an absolute maximum
on $\mathcal{M}(\check{\varSigma},\check{E})$, but if $\mathcal{M}^{*}(\check{\varSigma},\check{E})\subset\mathcal{M}(\check{\varSigma},\check{E})$
is non-empty then the maximum must occur at an abelian vortex.

e) Finally, $c_{1}(\check{E})>2c_{1}(K_{\check{\varSigma}})$ does
guarantee in the smooth case that the dimension of the moduli space
of $SO(3)$ vortices is positive {[}except when $g=0${]}. Likewise,
we will see in Section (\ref{sec:-Seifert}) examples where this condition
is not enough to guarantee positive dimensional moduli spaces.
\end{rem}

In the smooth case (no marked points on the Riemann surface), the
formula is quite clean and simply reads 
\[
\text{ind}(E,L)=2(g-1+\deg E-2\deg L)
\]

As an example of what one can do with the formula for the index, choosing
$\deg L=0$ we see that the index is $2(g-1+\deg E)$, which is the
dimension formula for the moduli space of $SO(3)$ vortices in the
smooth case. This means that this moduli space of abelian vortices
(which in fact ends up being a single point), corresponds to the maximum
for the function $\mu$.

Strictly speaking, these indices can be computed using the orbifold
version of Riemann-Roch even in the case where we are not imposing
conditions on $\check{E}$ which guarantee that the moduli spaces
are smooth at the reducible points, so they are best interpreted (in
the terminology of Feehan and Leness) as \emph{virtual }Morse-Bott
indices. 

\subsection{Framed Monopole Floer Homology and $SO(3)$ monopoles on $S^{1}\times\varSigma$}

\ 

A second motivation for studying the moduli space of $SO(3)$ vortices
over Riemann surfaces which could be of more interest to low-dimensional
topologists is the conjecture due to Kronheimer and Mrowka relating
the framed instanton homology $HI^{\#}(Y)$ of a 3 manifold and the
tilde version $\widetilde{HM}(Y)$ of monopole Floer homology (equivalently,
the hat version of Heegaard Floer homology $\widehat{HF}(Y)$). More
precisely, a special case of \cite[Conjecture 7.24]{MR2652464} states
the following:
\begin{conjecture}
\cite[Conjecture 7.24]{MR2652464} Let $Y$ denote a closed oriented
3-manifold. Then 
\[
HI^{\#}(Y)\simeq\widehat{HF}(Y)\otimes\mathbb{C}\simeq\widetilde{HM}(Y)\otimes\mathbb{C}
\]
\end{conjecture}

A definition of $\widetilde{HM}(Y)=\bigoplus_{\mathfrak{s}\in\text{Spin}^{c}(Y)}\widetilde{HM}(Y,\mathfrak{s})$
can be found in \cite{MR2764887,MR2746727} and by the work of Kutluhan,
Lee and Taubes \cite{MR4194309}, it is isomorphic to $\widehat{HF}(Y)=\bigoplus_{\mathfrak{s}\in\text{Spin}^{c}(Y)}\widehat{HF}(Y,\mathfrak{s})$.
The Euler characteristic of these groups are 
\[
\chi(\widetilde{HM}(Y))=\begin{cases}
|H_{1}(Y,\mathbb{Z})| & b_{1}(Y)=0\\
0 & b_{1}(Y)\neq0
\end{cases}
\]
That $HI^{\#}(Y)$ has the same Euler characteristic as $\widetilde{HM}(Y)$
was shown by Scaduto in \cite[Corollary 1.4]{MR3394316}, and the
isomorphism between the groups has been verified for certain cases
in \cite{MR3394316,BaldwinSivek[2019],MR3890778}.

The equality up to a sign between the Euler characteristics in the
more general case of the sutured versions of monopole, Heegaard and
instanton Floer homologies has been shown recently by Zhenkun Li and
Fan Ye in their paper \cite[Theorem 1.2]{Li-Ye[SuturedInstanton]}.

At this point it is useful to recall how $HI^{\#}(Y)$ is defined.
One first considers the ordinary instanton Floer homology $HI(Y\#T^{3},\gamma)$,
where $\gamma$ represents the Poincaré dual for an admissible $U(2)$
bundle $E$ over $Y\#T^{3}$ (more precisely, $\gamma$ is a loop
which represents one of the circle factors of $T^{3}$). 

The group $HI(Y\#T^{3},\gamma)$ is $\mathbb{Z}/8\mathbb{Z}$ graded,
and $HI^{\#}(Y)$ is isomorphic to four consecutive summands of $HI(Y\#T^{3},\gamma)$.
In the special case of $Y=S^{3}$, it then follows that $HI^{\#}(S^{3})$
corresponds to four summands of $HI(T^{3},\gamma)$, which is two
dimensional since it has two generators and no differential, so $HI^{\#}(S^{3})$
ends up being one-dimensional \cite[Section 4.1]{MR2860345}.

Notice that the Euler characteristic of these groups satisfies
\begin{equation}
\chi(HI^{\#}(Y))=\frac{1}{2}\chi(HI(Y\#T^{3},\gamma))=\frac{1}{2}\chi(CI(Y\#T^{3},\gamma))\label{Euler characteristics instantons}
\end{equation}
where $CI(Y\#T^{3},\gamma)$ denotes the chain complex used to define
the Floer groups. 

On the other hand, $\widetilde{HM}(Y)$ can be defined in terms of
a mapping cone construction which uses the monopole Floer chain complexes
on $Y$ \cite{MR2764887,MR4194309}. For us, the important fact is
that $\widetilde{HM}(Y)$ is defined in terms of data on $Y$, without
using additional 3-manifolds like $T^{3}$ as in the case of $HI^{\#}(Y)$.

Thus, in order to try to compare $HI^{\#}(Y)$ and $\widetilde{HM}(Y)$
it may be convenient to use an alternative construction of $\widetilde{HM}(Y)$
so that it is also defined on $Y\#T^{3}$. As communicated to the
author by Tom Mrowka, this alternative construction of $\widetilde{HM}(Y)$
was known to him and Peter Kronheimer, and probably to the experts
at large as well.

We motivate the construction in Section (\ref{sec:-Framed Monopole Homology}),
but basically we consider $HM(Y\#T^{3},\omega,\varGamma)$, that is,
the monopole Floer homology on $Y\#T^{3}$ associated to a suitable
non-balanced perturbation and a local coefficient system $\varGamma$,
which is needed in order to obtain a well defined differential. We
call this version of monopole Floer homology \emph{framed monopole
Floer homology}, and denote it as $HM^{\#}(Y)$ in analogy to the
notation for the case of instantons. 

It may not be clear why study $HM(Y\#T^{3},\omega,\varGamma)$ in
a paper that is about $SO(3)$ vortices over Riemann surfaces, but
the main idea is that if we think of $T^{3}$ as $S^{1}\times T^{2}$,
then the perturbation term $\omega$ appears naturally from how the
abelian vortices embed in the moduli space of $SO(3)$ vortices on
$T^{2}$. 

To be more specific, we first need to know how the solutions to the
$SO(3)$ monopole equations on $S^{1}\times\varSigma$ compare to
the $SO(3)$ vortices on $\varSigma$. In Section (\ref{sec:-Framed Monopole Homology})
we prove (for the Seiberg-Witten case see \cite{MR1438191,MR2141962}):
\begin{thm}
Suppose that $(B,\varPsi)$ is an irreducible $SO(3)$ monopole for
the spin-u structure 
\[
V=(\mathbb{C}\oplus K_{\varSigma}^{-1})\otimes E
\]
where $E$ is the pullback of a $U(2)$ bundle on $\varSigma$ under
the obvious projection map $S^{1}\times\varSigma\rightarrow\varSigma$.
Write 
\[
\varPsi=\alpha\oplus\beta\in\varGamma(E)\oplus\varGamma(K_{\varSigma}^{-1}\otimes E)
\]

Then either $\beta$ vanishes identically or $\alpha$ vanishes identically. 

a) The solutions with $\beta\equiv0$ can be identified with the irreducible
$SO(3)$ vortices for $(\varSigma,E)$.

b) The solutions with $\alpha\equiv0$ can be identified (via Serre
duality) with the irreducible $SO(3)$ vortices for $(\varSigma,K_{\varSigma}^{-1}\otimes E)$.

Moreover, if we assume that $c_{1}(E)>2c_{1}(K_{\varSigma})$, then
only solutions of type a) can occur.
\end{thm}

Returning to our discussion of framed monopole homology, choose the
$U(2)$ bundle $E$ over $T^{2}$ with $c_{1}(E)=1$. The expected
dimension of the moduli space of $SO(3)$ vortices is two, so after
dividing by the circle action we get a manifold which is topologically
an interval, the maximum of this interval corresponding to the unique
abelian vortex inside this moduli space, and the minimum of this interval
corresponding to the unique projectively flat connection on $\text{ad}E$
\footnote{Recall that irreducible projectively flat connections on Riemann surfaces
correspond to stable bundles, and in the case of an elliptic curve
these were classified by Atiyah \cite{MR0131423}. For the point of
view of algebraic geometry, we are looking at the moduli space of
stable bundles with a fixed determinant, which is why we get a single
element.}.

This forces $HM(T^{3},\omega)$ to be one dimensional (since Seiberg-Witten
monopoles on $S^{1}\times\varSigma$ are in bijection with abelian
vortices on $\varSigma$). Therefore, we can think of the moduli space
of $SO(3)$ vortices on $T^{2}$ as providing a \emph{canonical }cobordism
between the unique abelian vortex (which generates $HM^{\#}(S^{3})=HM(T^{3},\omega)$),
and the unique projectively flat connection on $T^{2}$ (which can
be taken as a generator for $HI^{\#}(S^{3})$). 

Thus, from this point of view, the isomorphism between $HM^{\#}(S^{3})$
and $HI^{\#}(S^{3})$ can be understood as a consequence of the fact
that there is a \emph{canonical }cobordism between the unique generator
for $HM^{\#}(S^{3})$ and the unique generator for $HI^{\#}(S^{3})$.
We should also point out that here we haven't perturbed the $SO(3)$
monopole equations on $S^{1}\times\varSigma$, since otherwise the
moduli space would be zero dimensional, instead of an interval as
we have in our situation. 

We should also note that the chain complex $CM^{\#}(Y)=CM(Y\#T^{3},\omega)$
which yields $HM(Y\#T^{3},\omega,\varGamma)$ is finitely generated,
thus the Euler characteristic of $HM^{\#}(Y)$ is well defined with
respect to the Novikov field implicit in our choice of $\varGamma$.
In particular,
\begin{equation}
\chi(HM^{\#}(Y))=\chi(CM^{\#}(Y))\label{eq: Euler characteristic monopoles}
\end{equation}
This immediately begs the questions:

1) How is $\chi(HM^{\#}(Y))$ related to $\chi(\widetilde{HM}(Y))$? 

2) How is $\chi(HM^{\#}(Y))$ related to $\chi(HI^{\#}(Y))$?

Regarding the first question, a Künneth formula proven by Kutluhan,
Lee and Taubes in \cite{MR4194309} relates $CM(Y\#T^{3},\omega,\varGamma)$
to a mapping cone $S_{U}(CM(Y\sqcup T^{3}),\omega,\varGamma))$. Since
the right hand side still involves a local coefficient system, this
is still not quite isomorphic to $\widetilde{HM}(Y)$ (which is defined
over integer coefficients for example), but rather a twisted version
$\widetilde{HM}(Y,\varGamma)$ (see the discussion preceding our theorem
(\ref{theo iso framed group}) for more details). But the Künneth
formula implies that 
\begin{equation}
\chi(HM^{\#}(Y))=\chi(\widetilde{HM}(Y))\label{eq: euler characteristic}
\end{equation}

Regarding the second question, thanks to work in progress by Aleksander
Doan and Chris Gerig \cite{Doan-Gerig[KM]}, which is based on ideas
of Kronheimer and Mrowka, after appropriate choices of orientations
\[
\chi(CM(Y\#T^{3},\omega))=\frac{1}{2}\chi(CI(Y\#T^{3},\gamma))
\]
so combining (\ref{Euler characteristics instantons}), (\ref{eq: Euler characteristic monopoles})
and (\ref{eq: euler characteristic}) we obtain 
\[
\chi(\widetilde{HM}(Y))=\chi(HM^{\#}(Y))=\chi(HI^{\#}(Y))
\]
 The interesting feature of the idea of Kronheimer and Mrowka used
by Doan and Gerig is that the equality between both Euler characteristics
is obtained by studying the $SO(3)$ monopole equations on the $3$-manifold,
so in a sense is a low dimensional version of Witten's conjecture.
One could then try to see if this can also be used for relating the
Floer homologies $HM^{\#}(Y)$ and $HI^{\#}(Y)$, but that is outside
the scope of this work. 

Alternatively, in order to relate the Euler characteristics one can
use the fact that for 3-manifolds with first Betti number the Euler
characteristic of the instanton Floer homology groups is (up to a
sign) twice the Casson-Walker-Lescop invariant $\lambda_{CWL}(Y)$
\cite[Theorem 1.1]{MR3403212}, and again up to a sign $\lambda_{CWL}(Y)$
equals the sum of the Turaev torsion functions on the 3-manifold \cite{MR2002617},
which are known to agree with the Seiberg-Witten invariant \cite{MR1418579}.

\subsection{$SO(3)$ monopoles on Seifert manifolds}

\ 

Finally, the fact that we studied the $SO(3)$ vortices over orbifold
Riemann surfaces allows us to describe the solutions to the $SO(3)$
monopole equations on Seifert manifolds in terms of the $SO(3)$ vortex
moduli spaces, which is the analogue of the computations done by Mrowka,
Ozsváth and Yu \cite{MR1611061} for the case of the Seiberg-Witten
equations. 

Just as in their paper, one must work with connections which are spinorial
with respect to a metric connection compatible with the Seifert structure
(which is not the Levi-Civita connection). This connection can be
thought of as an adiabatic connection \cite{MR1651420}. In Section
(\ref{sec:-Seifert}) we find: 
\begin{thm}
Suppose that $\pi:Y=S(\check{L})\rightarrow\check{\varSigma}$ is
a Seifert manifold arising as the unit circle bundle of an orbifold
line bundle $\check{L}\rightarrow\check{\varSigma}$. Let $(B,\varPsi)$
denote an irreducible $SO(3)$ monopole for the spin-u structure 
\[
V=(\mathbb{C}\oplus\pi^{*}(K_{\check{\varSigma}}^{-1}))\otimes\pi^{*}(\check{E})
\]
where $\check{E}$ is an $U(2)$ bundle over $\varSigma$. Write 
\[
\varPsi=\alpha\oplus\beta\in\varGamma(\pi^{*}(\check{E}))\oplus\varGamma(\pi^{*}(K_{\check{\varSigma}}^{-1})\otimes\pi^{*}(\check{E}))
\]

Then either $\beta$ vanishes identically or $\alpha$ vanishes identically. 

a) The solutions with $\beta\equiv0$ can be identified with the irreducible
$SO(3)$ vortices for $(\check{\varSigma},\check{E}')$. Here $\check{E}'$
is any $U(2)$ bundle over $\check{\varSigma}$ satisfying $\det\check{E}'\simeq\det\check{E}$. 

b) The solutions with $\alpha\equiv0$ can be identified (via Serre
duality) with the irreducible $SO(3)$ vortices for $(\varSigma,K_{\varSigma}^{-1}\otimes\check{E}')$,
where $\check{E}'$ has the same meaning as in part a).

Moreover, if we assume that $c_{1}(\check{E})>2c_{1}(K_{\check{\varSigma}})$,
then only solutions of type a) can occur.
\end{thm}

\begin{rem}
a) We should point out that if one were interested in relating the
flat $SU(2)$ or $SO(3)$ connections on $Y=S(\check{L})$ with solutions
to the Seiberg-Witten equations on $Y$, then the right $U(2)$ bundles
to look at are those which satisfy $\det\check{E}\simeq\check{L}^{k}$
for some integer $k$, since these are the ones Furuta and Steer used
to relate the projectively flat connections on $Y$ with solutions
to the Yang-Mills equations on $\check{\varSigma}$ \cite[Theorem 3.7]{MR1185787}

b) At the end of Section (\ref{sec:-Seifert}) we work out the example
of the Poincaré homology sphere. In this case one can choose $\det\check{E}\simeq\check{L}_{0}$
where $\check{L}_{0}$ satisfies $c_{1}(\check{L}_{0})=\frac{1}{2\cdot3\cdot5}=\frac{1}{30}$
and generates the topological Picard group for orbifold lines bundles
over $S^{2}(2,3,5)$. 

There are six $U(2)$ bundles $\check{E}'$ such that $\det\check{E}'\simeq\check{L}_{0}$,
but only two end up being interesting: call these $\check{E}_{1}$
and $\check{E}_{2}$. Each of these contain one projectively flat
connection.

The expected dimension of the moduli space of $SO(3)$ vortices for
$\check{E}_{1}$ and $\check{E}_{2}$ is $2$ and $0$ respectively,
so after taking the quotient by the $S^{1}$ action, it is clear that
the second one must be empty (in the sense that there are no irreducible
$SO(3)$ vortices), while the first one becomes one dimensional, thus
providing a cobordism between the unique abelian vortex and the unique
projectively flat connection on the bundle $\check{E}_{1}$. 

So on the Poincaré homology sphere, the moduli space of $SO(3)$ vortices
consists of:

i) The trivial $SU(2)$ connection, 

ii) Two irreducible $SU(2)$ flat connections

iii) One Seiberg-Witten monopole, 

iv) A one-dimensional moduli space of irreducible $SO(3)$ monopoles
which serves as a cobordism between the Seiberg-Witten monopole and
one of the irreducible flat connections.

It is interesting to compare the case of the Poincaré homology sphere
with that of $T^{3}$, since in both situations we have a one dimensional
moduli of $SO(3)$ vortices connecting one of the flat connections
which the unique Seiberg-Witten monopole. 

In fact, thanks to the work of Taubes \cite{MR1037415}, it is known
that the Casson invariant $\lambda_{C}(Y)$ \cite{MR1030042} is morally
$(1/2)$ the count of irreducible flat $SU(2)$ connections on $Y$. 

On the other hand, as Lim showed \cite{MR1739221} (verifying a conjecture
due to Kronheimer), on an integer homology sphere $Y$, the count
of irreducible Seiberg-Witten monopoles, while not a topological invariant,
can be made to agree with $\lambda_{C}(Y)$ after adding a suitable
correction term.

So morally there are twice as many flat connections as there are Seiberg-Witten
monopoles, and we find it is amusing to give examples where this ``principle''
holds on the nose.

This also suggests studying the monopole and instanton Floer homologies
of the connected sum $Y\#\varSigma(2,3,5)$, in analogy to studying
the monopole and instanton Floer homologies of $Y\#T^{3}$.

We also describe the case of the Brieskorn homology sphere $\varSigma(2,3,7)$
near the end of the paper. The computations are slightly more complicated,
but one finds a similar picture to that of $\varSigma(2,3,5)$.
\end{rem}

We finish this introduction by mentioning that in follow-up work we
plan to study the flowlines for the $SO(3)$ Chern-Simons-Dirac functional
on Seifert manifolds and $S^{1}\times\varSigma$ in terms of the geometry
of the moduli space of stable pairs on natural (orbifold) ruled surfaces
associated to these 3-manifolds \cite{MR1400761,MR1611061,MR1871404}.

\ 

\ 

\textbf{\emph{Outline of the paper: }}

In Section 2 we give a basic review of gauge theory on orbifolds.
In many ways this material is already common knowledge, so this section
is basically intended to fix some notation. If one is willing to invest
some time in the algebraic geometry literature, it is not difficult
to find very general formulas for the Hirzebruch-Grothendieck-Riemann-Roch
index theorem on orbifolds (or even stacks!). However, usually these
formulas are not fleshed out explicitly for the situations we have
in mind, so we also work out some of them in this section. 

Section 3 introduces the $SO(3)$ vortex equations over an orbifold
Riemann surface. The discussion is pretty standard, and include things
like the deformation theory of the moduli spaces, the conditions under
which abelian vortices appear in the moduli space, and the Kahler
structure on the moduli spaces. The fact that we are working over
an orbifold only means we need to keep track of more topological data,
but besides that it is almost indistinguishable from the smooth case.
We are also quite explicit when describing the different structures
on the moduli spaces, since many simplifications occur from the fact
that we are in dimension two. 

Section 4 discusses the well-known correspondence between $SO(3)$
vortices and stable pairs. Our use of the stable pairs point of view
is very limited in this work, and we mostly wrote it to clarify certain
things in case the reader is more acquainted with the $U(2)$ vortex
equations. 

In Section 5 we study the Morse-Bott function $\mu$ on the moduli
space of $SO(3)$ vortices. This is essentially the version for Riemann
surfaces of the work of Feehan and Leness \cite{Feehan-Leness[Virtual]},
or the classical paper by Hitchin but now for the case of $SO(3)$
vortices \cite{MR887284}. 

Section 6 describes some basic facts about the $SO(3)$ monopole equations
on 3-manifolds. This material is an adaptation of \cite{MR1855754},
but we found it useful to work it out explicitly because the reader
may not be familiarized with certain features of the $SO(3)$ monopole
equations. We also describe briefly the $U(2)$ monopole equations
on 3-manifolds, which could be of interest for technical reasons which
we discuss in this section.

In Section 7 we describe the $SO(3)$ monopole equations on 3-manifolds
of the form $S^{1}\times\varSigma$, where $\varSigma$ is a Riemann
surface. Our main interest is the case when $\varSigma$ is the 2-torus
$T^{2}$, since this example is what motivated our definition of the
\emph{framed monopole Floer homology }of $Y$. We give its definition
in this section, and discuss some of the properties mentioned in the
introduction.

Finally, in Section 8 we describe the behavior of the $SO(3)$ monopole
equations for Seifert manifolds and discuss the examples of $\varSigma(2,3,5)$
and $\varSigma(2,3,7)$.
\begin{acknowledgement*}
The idea to study the $SO(3)$ vortex equations on Riemann surfaces
arose from many helpful conversations with Paul Feehan and Tom Leness
surrounding their project \cite{Feehan-Leness[Virtual]} so the author
would like to thank them for their encouragement and many valuable
suggestions. The author would also like to thank Yi-Jen Lee for pointing
out the Künneth formula in her paper \cite{MR4194309} with Kutluhan
and Taubes, as well as Chris Woodward, Andrei Teleman, Oscar García-Prada,
Michael Thaddeus, Dinesh Valluri, Aleksander Doan, Tom Mrowka and
Matthew Stoffregen for useful conversations and correspondence.
\end{acknowledgement*}

\section{\label{sec:A-CrashCourse-on}A CrashCourse on Orbifolds }

We now give a brief summary of gauge theory on orbifolds, which in
the modern literature (e.g \cite{MR2493583,MR2359514}) falls into
the theory of analytic Deligne-Mumford stacks.

For our purposes, the main arena will be that of a Riemann surface
$\varSigma$ with finitely many marked points $p_{1},\cdots,p_{n}$.
Since this material is standard, we refer to \cite{MR1151325,MR1611061,MR1185787,MR1362026,MR1375314,MR1703606,MR1718086,MR1863850,MR1241873,MR1308489,MR1651420,MR1846125,Nicolaescu[1991],MR1388002,MR1359140,MR2450211,MR2493583,MR3376575,MR527023}
for more details and proofs of the main results we need.

We will follow the conventions of Kronheimer and Mrowka \cite{MR2860345},
and use a $\check{}$ whenever we want to emphasize the orbifold structure
of a geometric object. For example, $\check{\varSigma}$ will denote
an orbifold Riemann surface, while $\varSigma$ will denote the underlying
topological space (i.e, we forget the marked points). So we can think
of $\check{\varSigma}$ as a shorthand notation for the data $(\varSigma,p_{1},\cdots,p_{n})$,
with the isotropy at each of the marked points (which will be recalled
soon), being implicit.\textbf{}

In general, a \textbf{complex orbifold} $\check{X}$ consists of a
connected paracompact complex space such that every point $p$ has
an open neighborhood $U_{p}$ which is of the form $U_{p}\simeq V_{p}/G_{p}$,
where $V_{p}$ is a suitable complex manifold and $G_{p}$ a finite
group acting biholomorphically on $V_{p}$. 

The germ $(\check{X},p)$ of an orbifold at $p$ is called the \textbf{quotient
germ }\cite[Definition 1.1]{MR1388002}. If $\dim_{\mathbb{C}}\check{X}=m$,
we can assume that $G_{p}$ is a  finite subgroup of $GL_{m}(\mathbb{C})$,
unique up to conjugation, and $V_{p}$ is an open neighborhood of
$0\in\mathbb{C}^{m}$ such that $g(V_{p})=V_{p}$ for all $g\in G_{p}$.
We can think of the quotient germ $(\check{X},p)$ as $(\check{X},p)=(\mathbb{C}^{m},0)/G_{p}$
and the map $\pi:(\mathbb{\mathbb{C}}^{m},0)\rightarrow(\check{X},p)$
is called the \textbf{local smoothing covering }of $\check{X}$ at
$p$.\textbf{}

We will assume that there are finitely many points $p_{1},\cdots,p_{n}$
of $\check{X}$ where $G_{p}$ is non-trivial. In fact, we can assume
that $G_{p_{i}}\simeq\mathbb{Z}/a_{i}\mathbb{Z}$, where the action
is $z\ni\mathbb{C}^{m}\rightarrow e^{2\pi i/a_{i}}z$. In this case
the natural orbifold metric $\check{g}$ on $\check{X}$ is one with
a conical singularity at each $p_{i}$ of cone angle $2\pi/a_{i}$.
Integration on orbifolds can also be defined using a partition of
unity, and the orbifold version of the de Rham complex yields cohomology
groups isomorphic to those of the underlying topological space \cite[Section 2.1]{MR3376575}.
In other words, 
\[
H^{k}(\check{X};\mathbb{R})\simeq H^{k}(X;\mathbb{R})\simeq H_{DR}^{k}(X;\mathbb{R}),\;\;\;k=0,1,\cdots,\dim_{\mathbb{R}}X
\]
with a similar statement for $\mathbb{C}$ instead of $\mathbb{R}$.
Moreover, the cup product of orbifold de Rham classes is induced by
the wedge product. By the same token, any element of $H^{k}(\check{X};\mathbb{R})$
can be represented by a unique orbifold harmonic $k$ form.

\textbf{Orbi-bundles} can be defined in an analogous way to how we
defined orbifolds. Namely, we have the data $\pi:\check{E}\rightarrow\check{X}$
, where both $\check{E},\check{X}$ are orbifolds and $\pi$ is a
holomorphic surjective map. If $F$ is a complex manifold, $\pi$
has fibre $F$ if for every $p\in\check{X}$ there is a local smoothing
covering $U_{p}=V_{p}/G_{p}$ such that there is a $G_{p}$ action
on $V_{p}\times G_{p}$ with $\check{E}\mid_{U_{p}}\simeq(V_{p}\times F)/G_{p}$
over $U_{p}$.

In general the fibres of $\pi$ are orbifolds, namely, $\pi^{-1}(p)=F/G_{p}$
for some action of $G_{p}$ on $F$. In particular, notice that the
projection $\pi$ is not required to be locally trivial, as opposed
to the usual vector bundle situation. Regardless, it still makes sense
to talk about connections on vector orbi-bundles \cite[Section 2.2]{MR3376575}. 

For the particular case of an orbifold line bundle $\check{L}\rightarrow\check{X}$,
its tensor powers $\check{L}^{\otimes q}$ will continue to be orbifold
line bundles. In fact, one can take $q$ sufficiently large so that
$\check{L}^{\otimes q}$ becomes locally trivial, in which case one
can define the \textbf{orbifold Chern class }of $\check{L}$ as 
\[
c_{1}(\check{L})\equiv\frac{1}{q}c_{1}(\check{L}^{\otimes q})\in H^{2}(\check{X};\mathbb{Q})
\]
Notice that in general this will be a class in the second cohomology
of the orbifold with \emph{rational coefficients. }Alternatively,
one can also use Chern-Weil theory in order to define $c_{1}(\check{L})$
as 
\[
c_{1}(\check{L})=\left[\frac{i}{2\pi}F_{\check{\nabla}}\right]
\]
where $[\cdot]$ denotes the orbifold de Rham class defined by the
curvature of an orbifold connection $\check{\nabla}$ on $\check{L}$. 

For higher rank orbi-bundles $\check{E}$, the Chern classes $c_{i}(\check{E})$
can be defined in a similar way \cite[Section 2]{MR1388002}. Moreover,
there is a natural notion of an \textbf{orbifold Euler characteristic
}of $\check{X}$, given by the formula 
\[
e(\check{X})=e(X)-\sum_{i=1}^{n}\left(1-\frac{1}{a_{i}}\right)
\]
 In particular, this definition guarantees that the relation 
\[
e(\check{X})=(-1)^{m}c_{m}(\varOmega_{\check{X}}^{1})
\]
continues to hold, where $c_{m}(\varOmega_{\check{X}}^{1})$ is the
orbifold Chern number of the orbifold holomorphic cotangent bundle. 

When $\check{\mathscr{F}}$ is a locally free orbifold sheaf of rank
$r$ on $\check{X}$, the version of \textbf{Hirzebruch-Grothendieck-Riemann-Roch}
reads in this context \cite[Section 3]{MR1388002}: 
\begin{equation}
\chi(\check{X},\check{E})=\int_{\check{X}}\text{ch(\ensuremath{\check{E}}})\text{td}(\check{X})+\sum_{i=1}^{n}\frac{1}{\#G_{p_{i}}}\left(\sum_{g\in G_{p_{i}}-\{1_{m}\}}\frac{\text{tr}(\rho(g))}{\det(1_{m}-g)}\right)\label{eq:Riemann Roch}
\end{equation}
where on the left hand side we mean the alternating sum of the dimensions
of the cohomology groups $H^{i}(\check{X},\check{E})$.

 Here we are regarding $g\in G_{p}\simeq\mathbb{Z}/a_{i}\mathbb{Z}$
as an $m\times m$ matrix, and $\rho(g)$ denotes the action of $g$
on the orbifold bundle $\check{E}$. 

In order to compare this formula with the one which appears in the
paper by Furuta and Steer \cite[Theorem 1.5]{MR1185787}, we take
$\check{X}=\check{\varSigma}$ and assume $\check{E}=\check{L}$ is
a line bundle. In this case, 
\[
\begin{cases}
\text{ch}(\check{L})=\text{rk}(\check{L})+c_{1}(\check{L})=1+c_{1}(\check{L})\\
\text{td}(\check{L})=1+\frac{c_{1}(\check{L})}{2}
\end{cases}
\]
Therefore we have 
\begin{align*}
 & \int_{\check{\varSigma}}\text{ch(\ensuremath{\check{L}}})\text{td}(\check{\varSigma})\\
= & \int_{\check{\varSigma}}\left(1+c_{1}(\check{L})\right)\left(1+\frac{c_{1}(\check{\varSigma})}{2}\right)\\
= & \int_{\check{\varSigma}}\left(\frac{1}{2}\right)c_{1}(\check{\varSigma})+c_{1}(\check{L})\\
= & \frac{1}{2}\left(2-2g-\sum_{i=1}^{n}\left(1-\frac{1}{a_{i}}\right)\right)+\int_{\check{\varSigma}}c_{1}(\check{L})\\
= & (1-g)-\sum_{i=1}^{n}\left(\frac{a_{i}-1}{2a_{i}}\right)+\int_{\check{\varSigma}}c_{1}(\check{L})
\end{align*}
Now we want to analyze the term $\sum_{g\in G_{p_{i}}-\{1\}}\frac{\text{tr}(\rho(g))}{\det(1-g)}$.To
understand how to compute $\rho(g)$, it is useful to work with orbifold
connections in a slightly more concrete way \cite[section 1.(ii)]{MR1241873}. 

For the case of an orbifold line bundle $\check{L}$, near each marked
point $p$ we can choose polar coordinates $(r,\theta)$ on a disk
neighborhood $D_{p}$ of $p$ so that we have the connection form
\[
A_{\lambda}=i\lambda d\theta
\]
 where $\lambda$ is a constant known the \emph{holonomy parameter}.
The reason for this name is that holonomy of a connection $A$ on
$D_{p}\backslash\{p\}$ whose matrix connection coincides with $A_{\lambda}$
near $p$ on the positively-oriented small circles of constant $r$
will have a limiting holonomy approximately equal to 
\[
e^{-2\pi i\lambda}
\]
Strictly speaking, we want to study the previous situation \emph{up
to conjugacy }\cite{MR1152376}, but since $U(1)$ is abelian this
is not important, so it suffices to take $0\leq\lambda<1$ in order
to guarantee that $e^{-2\pi i\lambda}$ exhausts all possible elements
in $U(1)$. We are interested in connections that differ from $A_{\lambda}$
by a smooth one form on $\varSigma$, and for the orbifold interpretation
to hold, we need to take $\lambda\in\mathbb{Q}\cap[0,1)$. In fact,
we will write $\lambda$ as $\lambda=\frac{b}{a}$ for $0<b<a$. The
parameter $b$ can then be identified with what Furuta and Steer call
the \emph{isotropy data }of the line bundle \cite{MR1185787}. It
is now clear that we can write 
\[
\sum_{g\in G_{p_{i}}-\{1\}}\frac{\text{tr}(\rho(g))}{\det(1-g)}=\sum_{k=1}^{a_{i}-1}\frac{e^{-2\pi ikb_{i}/a_{i}}}{1-e^{2\pi ik/a_{i}}}
\]

Now we need to find a more concrete version of the previous sum. A
similar expression appears in \cite[p.7]{DhillonValluri[2020]}, and
the author would like to thank Dinesh Valluri for suggesting how to
prove the following lemma.
\begin{lem}
Let $\zeta$ denote an a-th root of unity, which for simplicity we
take as $\zeta=e^{2\pi i/a}$. Then for any $0<b<a$,
\[
\sum_{k=1}^{a-1}\frac{\zeta^{kb}}{1-\zeta^{k}}=\sum_{k=1}^{a-1}\frac{e^{2\pi ikb/a}}{1-e^{2\pi ik/a}}=b-\frac{a+1}{2}
\]
\end{lem}

\begin{proof}
\textbf{Step $1$: }

Since $1=\zeta^{aj}=(\zeta^{j})^{a}$ we have that 
\[
0=(\zeta^{j}-1)(1+\zeta^{j}+\cdots+(\zeta^{j})^{a-1})
\]
 So for $0<j<a$, we have 
\begin{equation}
\sum_{k=1}^{a-1}\zeta^{kj}=-1\label{IDENTITY 1}
\end{equation}

\textbf{Step 2: 
\begin{equation}
\sum_{k=1}^{a-1}\frac{1}{\zeta^{k}-1}=-\left(\frac{a-1}{2}\right)\label{identity 2}
\end{equation}
}Observe that the left hand side of (\ref{identity 2}) equals 
\begin{align*}
 & \sum_{k=1}^{a-1}\frac{1}{\zeta^{k}-1}\\
= & \sum_{k=1}^{a-1}\frac{1-\zeta^{k}+\zeta^{k}}{\zeta^{k}-1}\\
= & \sum_{k=1}^{a-1}\left[-1+\left(\frac{\zeta^{k}}{\zeta^{k}-1}\right)\right]\\
= & -(a-1)+\sum_{k=1}^{a-1}\frac{\zeta^{k}}{\zeta^{k}-1}\\
= & -(a-1)+\sum_{k=1}^{a-1}\frac{1}{1-\zeta^{-k}}
\end{align*}
Now, notice that
\begin{align*}
 & \sum_{k=1}^{a-1}\frac{1}{1-\zeta^{-k}}\\
= & \sum_{k=1}^{a-1}\frac{1}{1-\zeta^{a-k}}\\
= & \frac{1}{1-\zeta^{a-1}}+\frac{1}{1-\zeta^{a-2}}+\cdots+\frac{1}{1-\zeta}\\
= & \sum_{k=1}^{a-1}\frac{1}{1-\zeta^{k}}
\end{align*}
 Therefore, the proof of (\ref{identity 2}) becomes 
\begin{align*}
 & \sum_{k=1}^{a-1}\frac{1}{\zeta^{k}-1}=-(a-1)+\sum_{k=1}^{a-1}\frac{1}{1-\zeta^{k}}\\
\\
\implies & \sum_{k=1}^{a-1}\frac{1}{\zeta^{k}-1}=-\frac{(a-1)}{2}
\end{align*}

\textbf{Step 3: 
\begin{equation}
\frac{\zeta^{kb}-1}{\zeta^{k}-1}=1+\zeta^{k}+\cdots+(\zeta^{k})^{b-1}\label{IDENTITY 3}
\end{equation}
}This is obtained by factoring the numerator $\zeta^{kb}-1=(\zeta^{k})^{b}-1$.

\textbf{Proof of Lemma: 
\begin{align*}
 & \sum_{k=1}^{a-1}\frac{\zeta^{kb}-1}{\zeta^{k}-1}\\
= & \sum_{k=1}^{a-1}\left(\sum_{j=0}^{b-1}\zeta^{kj}\right)\\
= & \sum_{j=0}^{b-1}\sum_{k=1}^{a-1}\zeta^{kj}\\
=^{\text{Step}(1)} & (a-1)+\sum_{j=1}^{b-1}\underbrace{\sum_{k=1}^{a-1}\zeta^{kj}}_{-1}\\
= & a-1-(b-1)\\
= & a-b
\end{align*}
}Therefore, 
\begin{equation}
\sum_{k=1}^{a-1}\frac{\zeta^{kb}}{\zeta^{k}-1}-\sum_{k=1}^{a-1}\frac{1}{\zeta^{k}-1}=a-b\label{identity 4}
\end{equation}
Using step (2) we obtain 
\[
\sum_{k=1}^{a-1}\frac{\zeta^{kb}}{\zeta^{k}-1}=-\frac{(a-1)}{2}+a-b=\frac{a+1}{2}-b
\]
Therefore 
\[
\sum_{k=1}^{a-1}\frac{\zeta^{kb}}{1-\zeta^{k}}=b-\frac{a+1}{2}
\]
\end{proof}

Therefore 
\begin{align*}
 & \sum_{i=1}^{n}\frac{1}{\#G_{p_{i}}}\left(\sum_{g\in G_{p_{i}}-\{1\}}\frac{\text{tr}(\rho(g))}{\det(1-g)}\right)\\
= & \sum_{i=1}^{n}\frac{1}{a_{i}}\left(\sum_{k=1}^{a_{i}-1}\frac{e^{-2\pi ikb_{i}/a_{i}}}{1-e^{2\pi ik/a_{i}}}\right)\\
= & \sum_{i=1}^{n}\frac{1}{a_{i}}\left(\sum_{k=1}^{a_{i}-1}\frac{e^{2\pi ik(a_{i}-b_{i})/a_{i}}}{1-e^{2\pi ik/a_{i}}}\right)\\
= & \sum_{i=1}^{n}\frac{1}{a_{i}}\left(a_{i}-b_{i}-\frac{a+1}{2}\right)\\
= & -\sum_{i=1}^{n}\frac{b_{i}}{a_{i}}+\sum_{i=1}^{n}\frac{a_{i}-1}{2a_{i}}
\end{align*}

So the orbifold version of Riemann Roch for line bundles is 
\begin{align*}
 & \dim_{\mathbb{C}}H^{0}(\check{\varSigma},\check{L})-\dim_{\mathbb{C}}H^{1}(\check{\varSigma},\check{L})\\
= & (1-g)-\sum_{i=1}^{n}\left(\frac{a_{i}-1}{2a_{i}}\right)+\int_{\check{\varSigma}}c_{1}(\check{L})-\sum_{i=1}^{n}\frac{b_{i}}{a_{i}}+\sum_{i=1}^{n}\frac{a_{i}-1}{2a_{i}}\\
= & (1-g)+\int_{\check{\varSigma}}c_{1}(\check{L})-\sum_{i=1}^{n}\frac{b_{i}}{a_{i}}\\
= & (1-g)+c_{1}(\check{L})-\sum_{i=1}^{n}\frac{b_{i}}{a_{i}}
\end{align*}
which agrees with \cite[Theorem 1.5]{MR1185787}, under the usual
abuse of notation that identifies $c_{1}(\check{L})$ with a number.
\begin{rem}
The quantity $c_{1}(\check{L})-\sum_{i=1}^{n}\frac{b_{i}}{a_{i}}$
is denoted the \textbf{background degree }in \cite[Definition 2.7]{MR1611061}.
As explained in \cite[p.44]{MR1185787}, if one thinks of $H^{*}(\check{\varSigma},\check{L})$
as $H^{*}(\check{\varSigma},\mathcal{O}(\check{L}))$, where $\mathcal{O}(\check{L})$
is the sheaf of holomorphic sections, then $\mathcal{O}(\check{L})$
is a locally free sheaf with degree $c_{1}(\check{L})-\sum_{i=1}^{n}\frac{b_{i}}{a_{i}}$.
We will write 
\begin{equation}
\deg_{B}(\check{L})=c_{1}(\check{L})-\sum_{i=1}^{n}\frac{b_{i}}{a_{i}}\label{eq:background degree}
\end{equation}
 for this formula of the background degree.

One can also relate orbifold line bundles with divisors \cite[Section 1B]{MR1375314}.
Namely, a divisor $\check{D}$ on $\check{\varSigma}$ is a finite
linear combination 
\[
\check{D}=\sum_{p\in\varSigma}\frac{n_{p}}{a_{p}}\cdot p
\]
where $n_{p}\in\mathbb{Z}$, and $a_{p}=1$ if $p$ is not one of
the marked points on $\check{\varSigma}$. The degree of a divisor
is then 
\[
\deg\check{D}=\sum_{p}\frac{n_{p}}{a_{p}}
\]
Proposition 1.3 in \cite{MR1375314} then describes a bijection between
equivalence classes of divisors and holomorphic prbifold line bundles.
If $\check{L}_{\check{D}}$ is the line bundle corresponding to $\check{D}$,
then $c_{1}(\check{L}_{\check{D}})=\deg\check{D}$.
\end{rem}

For our purposes it is also necessary to have a formula of Riemann-Roch
for $U(2)$ orbifold bundles $\check{E}$. One option would be to
apply formula (\ref{eq:Riemann Roch}) one more time, but we find
it more elegant to bootstrap from Riemann-Roch for orbifold line bundles
$\check{L}$. This requires a small detour into the classification
of $U(2)$ orbifold bundles, which can be found in \cite[Proposition 1.8]{MR1185787},
as well as \cite[Section 1A]{MR1375314}. 

The isotropy data for a rank two $U(2)$ orbifold bundle $\check{E}$
will be specified by the pairs of integers 
\[
((b_{1}^{-},b_{1}^{+}),\cdots,(b_{n}^{-},b_{n}^{+}))
\]
where $0\leq b_{i}^{-},b_{i}^{+}<a_{i}$, for $i=1,\cdots,n$. In
particular, we can find orbifold line bundles $\check{L}^{-},\check{L}^{+}$
with isotropy data $(b_{i}^{-})$ and $(b_{i}^{+})$ respectively.
Then 
\[
\check{E}=\check{L}^{-}\oplus\check{L}^{+}
\]
is an $U(2)$ orbifold bundle with isotropy data $((b_{1}^{-},b_{1}^{+}),\cdots,(b_{n}^{-},b_{n}^{+}))$.
Notice that the notation seems to suggest that $b_{i}^{-}\leq b_{i}^{+}$,
in fact, this is assumed in \cite[p.599]{MR1375314}. This convention
is related to which group of automorphisms we are allowed to use. 

Up to conjugacy, any matrix in $U(2)$ takes the form 
\[
\left(\begin{array}{cc}
e^{-2\pi i\lambda_{1}} & 0\\
0 & e^{-2\pi i\lambda_{2}}
\end{array}\right)
\]
 where we can assume  $0\leq\lambda_{1}\leq\lambda_{2}<1$. However,
this classification requires conjugating by arbitrary matrices in
$U(2)$. In our case we are interested in using the \emph{determinant
one gauge group}, which means that the bundle $\check{L}^{-}\oplus\check{L}^{+}$
needs to be distinguished from the bundle $\check{L}^{+}\oplus\check{L}^{-}$,
since an automorphism which permutes both factors would locally take
the form $\left(\begin{array}{cc}
0 & 1\\
1 & 0
\end{array}\right)$, and this matrix, while an element of $U(2)$, is not an element
of $SU(2)$, because it has determinant $-1$. 

Therefore, with respect to the determinant one gauge group, we need
to consider the isotropy data as ordered pairs, without any assumption
regarding whether $b_{i}^{-}\leq b_{i}^{+}$ or not. This point is
briefly made on \cite[p. 602]{MR1338483} . However, ultimately we
are interested in studying the moduli spaces after taking the quotient
by the residual $S^{1}$ action mentioned in the introduction. Since
$S^{1}\times_{\{\pm I_{2}\}}SU(2)$ can be identified with $U(2)$
via the map $(e^{i\theta},A)\rightarrow e^{i\theta}A$, this  point
doesn't end up making a big difference, so it is fine to assume $b_{i}^{-}\leq b_{i}^{+}$
while reading this paper.

\textbf{}

In any case, the important point is that the orbifold version of Hirzebruch-Grothendieck-Riemann-Roch
states that the Euler characteristic $\chi(\check{\varSigma},\check{E})$
is topological, so it suffices to find the formula in the case of
$\chi(\check{\varSigma},\check{L}^{-}\oplus\check{L}^{+})$. By this
we mean the following: the spaces $H^{0}(\check{\varSigma},\check{E})$
and $H^{1}(\check{\varSigma},\check{E})$ can be identified with the
dimensions of the kernel and cokernel of an orbifold Dolbeault operator
$\bar{\partial}_{\check{C}}:\varOmega^{0}(\check{\varSigma},\check{E})\rightarrow\varOmega^{0,1}(\check{\varSigma},\check{E})$,
where $\check{C}$ is an arbitrary $U(2)$ orbifold connection on
$\check{E}$. 

In particular, we can assume that $\check{C}$ is a $U(2)$ connection
compatible with the topological splitting $\check{E}=\check{L}^{-}\oplus\check{L}^{+}$,
which we write as $\check{C}=\check{C}^{-}\oplus\check{C}^{+}$. In
this situation it is clear that $\chi(\check{\varSigma},\check{E})=\chi(\check{\varSigma},\check{L}^{-})+\chi(\check{\varSigma},\check{L}^{+})$
and thus we find that 
\begin{align*}
 & \dim_{\mathbb{C}}H^{0}(\check{\varSigma},\check{E})-\dim_{\mathbb{C}}H^{1}(\check{\varSigma},\check{E})\\
= & \chi(\check{\varSigma},\check{L}^{-})+\chi(\check{\varSigma},\check{L}^{+})\\
= & (1-g)+c_{1}(\check{L}^{-})-\sum_{i=1}^{n}\frac{b_{i}^{-}}{a_{i}}+(1-g)+c_{1}(\check{L}^{+})-\sum_{i=1}^{n}\frac{b_{i}^{+}}{a_{i}}\\
= & 2(1-g)+c_{1}(\det\check{E})-\sum_{i=1}^{n}\frac{b_{i}^{-}+b_{i}^{+}}{a_{i}}
\end{align*}
In the special case of a $SU(2)$ orbifold bundle $\check{E}$, we
can assume that $0\leq b_{i}^{-}\leq[a_{i}/2]$, and also take $b_{i}^{+}=a_{i}-b_{i}^{-}$.
In particular, the previous formula reduces to 
\[
2(1-g)-n
\]
Summarizing our findings, we have found 

\ 

\ 

\ 

\ 

\begin{thm}
\textbf{Riemann-Roch for orbifold Riemann surfaces:}

a) If $\check{L}$ is an orbifold line bundle with isotropy data $(b_{i})_{i=1,\cdots,n}$,
then 
\begin{equation}
\dim_{\mathbb{C}}H^{0}(\check{\varSigma},\check{L})-\dim_{\mathbb{C}}H^{1}(\check{\varSigma},\check{L})=(1-g)+c_{1}(\check{L})-\sum_{i=1}^{n}\frac{b_{i}}{a_{i}}\label{Riemann Roch Line Bundles}
\end{equation}

b) If $\check{E}$ is an orbifold $SU(2)$ bundle with isotropy data
$((b_{i}^{-},b_{i}^{+}=a_{i}-b_{i}^{-}))_{i=1}^{n}$, then 
\[
\dim_{\mathbb{C}}H^{0}(\check{\varSigma},\check{E})-\dim_{\mathbb{C}}H^{1}(\check{\varSigma},\check{E})=2(1-g)-n
\]

c) If $\check{E}$ is an orbifold $U(2)$ bundle with isotropy data
$((b_{i}^{-},b_{i}^{+}))_{i=1}^{n}$, and determinant line bundle
$\det\check{E}$, then 
\begin{equation}
\dim_{\mathbb{C}}H^{0}(\check{\varSigma},\check{E})-\dim_{\mathbb{C}}H^{1}(\check{\varSigma},\check{E})=2(1-g)+c_{1}(\det\check{E})-\sum_{i=1}^{n}\frac{b_{i}^{-}+b_{i}^{+}}{a_{i}}\label{Riemann Roch U(2)}
\end{equation}
\end{thm}

Riemann-Roch is specially powerful when combined with Serre duality.
Recall that \textbf{Serre duality }states that 
\[
H^{i}(\check{\varSigma},\check{E})\simeq H^{1-i}(\check{\varSigma},K_{\check{\varSigma}}\otimes\check{E}^{*})
\]
Here $K_{\check{\varSigma}}$ is the canonical bundle, which as an
orbifold line bundle has Seifert invariants 
\begin{equation}
(2g-2,a_{1}-1,\cdots,a_{n}-1)\label{eq:Seifert Canonical}
\end{equation}
in the sense that $\deg_{B}K_{\check{\varSigma}}=2g-2$ and 
\[
c_{1}(K_{\check{\varSigma}})=2g-2+\sum_{i=1}^{n}\frac{a_{i}-1}{a_{i}}=2g-2+n-\sum_{i=1}^{n}\frac{1}{a_{i}}
\]
Thus, suppose we want to check which conditions on $\check{L}$ guarantee
that $H^{1}(\check{\varSigma},\check{L})$ vanishes. First of all,
by Serre duality, 
\[
\dim H^{1}(\check{\varSigma},\check{L})=\dim H^{0}(\check{\varSigma},K_{\check{\varSigma}}\otimes\check{L}^{*})
\]
If $\check{L}$ has isotropy $b_{i}$ at the point $p_{i}$ we define
\[
\epsilon_{i}=\begin{cases}
1 & 0<b_{i}\\
0 & b_{i}=0
\end{cases}
\]
then the isotropy of $\check{L}^{*}$ at the point $p_{i}$ is $\epsilon_{i}(a_{i}-b_{i})$,
and the isotropy of $K_{\check{\varSigma}}\otimes\check{L}^{*}$ is
\[
\begin{cases}
a_{i}-1 & \text{if }b_{i}=0\\
a_{i}-1-b_{i} & \text{if }a_{i}-1>b_{i}\\
0 & \text{if }b_{i}=a_{i}-1
\end{cases}
\]
Notice that the first case ends up being contained in the second one.

Clearly if $\deg_{B}(K_{\check{\varSigma}}\otimes\check{L}^{*})<0$
, then $H^{1}(\check{\varSigma},\check{L})$ vanishes. Looking at
the formula for the background degree, this requires 
\[
c_{1}(K_{\check{\varSigma}}\otimes\check{L}^{*})-\sum_{0\leq b_{i}<a_{i}-1}\frac{a_{i}-1-b_{i}}{a_{i}}<0
\]
This is the same as $2g-2+n-\sum_{i=1}^{n}\frac{1}{a_{i}}-\sum_{0\leq b_{i}<a_{i}-1}\frac{a_{i}-1-b_{i}}{a_{i}}<c_{1}(\check{L})$,
or more succinctly 
\[
2g-2+\sum_{b_{i}=a_{1}-1}\frac{a_{i}-1}{a_{i}}+\sum_{0\leq b_{i}<a_{i}-1}\frac{b_{i}}{a_{i}}<c_{1}(\check{L})
\]
But this ends up being equal to 
\[
2g-2<\deg_{B}(\check{L})
\]
Therefore, we conclude that 
\begin{lem}
Suppose that $\check{L}$ is an orbifold line bundle such that $2g-2<\deg_{B}(\check{L})$.
Then $H^{1}(\check{\varSigma},\check{L})$ vanishes.
\end{lem}

\begin{rem}
Observe that $c_{1}(K_{\check{\varSigma}}\otimes\check{L}^{*})=c_{1}(K_{\check{\varSigma}})-c_{1}(\check{L})$,
so by \cite[Corollary 1.4]{MR1375314} $H^{1}(\check{\varSigma},\check{L})$
will also vanish whenever $c_{1}(K_{\check{\varSigma}})<c_{1}(\check{L})$. 
\end{rem}

It is also important to mention the conditions an orbifold line bundle
$\check{L}$ must satisfy in order to appear in a reduction of an
orbifold $U(2)$ bundle $\check{E}$ \cite{MR1375314,MR1185787}.
First of all, we must have the topological splitting 
\[
\check{E}=\check{L}\oplus\left(\check{L}^{*}\otimes\det\check{E}\right)
\]

The isotropy data of $\check{L}$ is determined by that of $\check{E}$
in the sense that $b_{i}\in\{b_{i}^{-},b_{i}^{+}\}$ for all $i$.
\textbf{} Define the vector $\epsilon=(\epsilon_{1},\cdots,\epsilon_{n})$
by the conditions
\[
\epsilon_{i}=\begin{cases}
0 & b_{i}^{-}=b_{i}^{+}\\
-1 & b_{i}=b_{i}^{-}\\
+1 & b_{i}=b_{i}^{+}
\end{cases}
\]
and the integers
\[
\begin{cases}
n_{0}=\#\{i\mid\epsilon_{i}=0\}\\
n_{\pm}=\#\{i\mid\epsilon_{i}=\pm1\}
\end{cases}
\]
Then $c_{1}(\check{L})$ satisfies $c_{1}(\check{L})-\sum_{i=1}^{n}\frac{\epsilon_{i}(b_{i}^{+}-b_{i}^{-})+(b_{i}^{+}+b_{i}^{-})}{2a_{i}}\in\mathbb{Z}$.

The topological isomorphism classes of orbifold line bundles form
a group under tensor product denoted $\text{Pic}_{V}^{t}(\check{\varSigma})$.
According to \cite[Corollary 1.7]{MR1185787}, when $a_{1},\cdots,a_{n}$
are mutually coprime, there is a unique orbifold line bundle $\check{L}_{0}$
such that $c_{1}(\check{L}_{0})=\frac{1}{\prod_{i=1}^{n}a_{i}}$ ,
and any other orbifold line bundle $\check{L}$ has the form $\check{L}_{0}^{k}$,
for some integer $k$. 

Moreover, as explained in \cite[p. 47]{MR1185787}, for $SO(3)$ orbifold
bundles there is a commutative diagram 
\[
\begin{array}{ccc}
\{U(2)-\text{orbifold bundles}\} & \rightarrow^{\det} & \text{Pic}_{V}^{t}(\check{\varSigma})\\
\downarrow_{Ad} &  & \downarrow\\
\{SO(3)-\text{orbifold bundles}\} & \rightarrow^{w} & \text{Pic}_{V}^{t}/(\text{Pic}_{V}^{t})^{2}
\end{array}
\]
where the map $w$ is essentially the second Stiefel-Whitney class. 

For the examples we will analyze at the end of this paper, it is important
to know when an orbifold $U(2)$ bundle $\check{E}$ admits an irreducible
projectively flat connection. 

The proofs can be found in \cite{MR1185787}, but we found the summary
from \cite{MR1375314} especially useful. In particular, we need the
following facts:
\begin{fact}
\label{FACTS}(See \cite{MR1185787}, \cite[Lemmas 2.2 and 4.2]{MR1375314})

a) Suppose $\check{E}$ is an orbifold $U(2)$ bundle with isotropy
data $(b_{1}^{\pm},b_{2}^{\pm},\cdots,b_{n}^{\pm})$ and let $n_{0}=\#\{i\mid b_{i}^{-}=b_{i}^{+}\}$.
If $g=0$ and $n-n_{0}\leq2$, then $\check{E}$ admits no irreducible
projectively flat connections.

b) The space of irreducible projectively flat connections is connected
and simply connected regardless of the genus of $\varSigma$.

c) The space of irreducible projectively flat connections is empty
if and only if there exists a vector $(\epsilon_{i})$, with $\epsilon_{i}=\pm1$,
such that $n_{+}+\deg_{B}(\det\check{E})\equiv1\mod2$ and $n_{+}-\sum_{i=1}^{n}\frac{\epsilon_{i}(b_{i}^{+}-b_{i}^{-})}{a_{i}}<1-g$.
Here $n_{\pm}=\#\{i\mid\epsilon_{i}=\pm1\}$. 
\end{fact}

\textbf{}

\section{\label{sec:the--Vortex SO(3)}the $SO(3)$ Vortex Moduli Space}

Let $\check{\varSigma}=(\varSigma,p_{1},\cdots,p_{n})$ denote an
orbifold Riemann surface and choose an orbifold $U(2)$ bundle $\check{E}\rightarrow\check{\varSigma}$.
Our policy will be to not make any particular assumptions on $\check{E}$
at this point. For example, in the smooth case it is typical to assume
$\deg E>4g-4$ \cite[Assumption 2]{MR1124279}, but we will explain
why this choice of degree is made in this section and the next.

Even if the bundle is an orbifold $SU(2)$ bundle, we will regard
it as an orbifold $U(2)$ bundle, in the sense that we will make a
choice of reference connection in $\det\check{E}$, though in the
$SU(2)$ case we can just take it to be the trivial connection. 

 For notational ease, we will use $\check{}$ to remind the reader
that our bundles and other geometric data should be interpreted in
the orbifold sense, but for an orbi-section of $\check{E}$, we will
write $\varUpsilon$ instead of $\check{\varUpsilon}$. Similar remarks
apply to our notation for orbifold connections. 

It is a good moment to justify our unusual choice of notation at this
point. Since we are eventually interested in using these moduli spaces
for 3- and 4-manifolds, we want a notation that does not interfere
with the one used in these contexts. In particular, we are trying
to follow the notation of Kronheimer and Mrowka \cite{MR2388043},
so typically for us
\begin{itemize}
\item $X$ denotes a 4-manifold, $A$ a connection of some bundle defined
on $X$, and $\varPhi$ a section of some spinor-type bundle.
\item $Y$ denotes a 3-manifold, $B$ a connection of some bundle defined
on $Y$, and $\varPsi$ a section of some spinor-type bundle.
\item $\varSigma$ denotes a 2-manifold, $C$ a connection of some bundle
defined on $\varSigma$, and $\varUpsilon$ a section of some spinor-type
bundle.
\end{itemize}
Further remarks regarding our notation will be explained as it becomes
necessary.
\begin{defn}
Let $\mathcal{A}^{\det}(\check{E})$ denote the space of $U(2)$ connections
on $\check{E}$ which induce a fixed reference $U(1)$ connection
$C^{\det}$ on $\det\check{E}$. The\textbf{ $SO(3)$ vortex equations}
are equations for a pair $(C,\varUpsilon)\in\mathcal{A}^{\det}(\check{E})\times\varGamma(\check{E})$
which satisfy 
\begin{equation}
SO(3)\text{ vortex equations }\begin{cases}
*F_{C}^{0}-i\left[\varUpsilon\varUpsilon^{*}-\frac{1}{2}|\varUpsilon|^{2}I_{E}\right]=0\\
\bar{\partial}_{C}\varUpsilon=0
\end{cases}\label{eq:SO(3) vortex equations}
\end{equation}
where $F_{C}^{0}$ denotes the traceless part of the curvature of
$C$.
\end{defn}

\begin{rem}
Observe that the first equation makes sense since $\varUpsilon\varUpsilon^{*}-\frac{1}{2}|\varUpsilon|^{2}I_{E}$
is a traceless hermitian endomorphism of $E$, so $i\left[\varUpsilon\varUpsilon^{*}-\frac{1}{2}|\varUpsilon|^{2}I_{E}\right]$
is skew-hermitian, and thus can be regarded as an element of $\mathfrak{su}(2)$.
\end{rem}

If we define $\mathcal{C}(\check{\varSigma},\check{E})=\mathcal{A}^{\det}(\check{E})\times\varGamma(\check{E})$
as the configuration space and $\mathcal{G}^{\det}(\check{E})$ the
gauge group (those automorphisms of $\check{E}$ with determinant
one), we want to study the solutions to the $SO(3)$ vortex equations
modulo gauge, i.e, the moduli space $\mathcal{M}(\check{\varSigma},\check{E})\subset\mathcal{C}(\check{\varSigma},\check{E})/\mathcal{G}^{\det}(\check{E})$.
Recall that the gauge group $u\in\mathcal{G}^{\det}(E)$ acts on $(C,\varUpsilon)$
as 
\[
u\cdot(C,\varUpsilon)=(C-(d_{C}u)u^{-1},u\varUpsilon)
\]

The $SO(3)$ \textbf{vortex map }is the map\textbf{
\begin{align}
\mathfrak{F}_{SO(3)}: & \mathcal{A}^{\det}(\check{E})\times\varGamma(\check{E})\rightarrow\varGamma(\check{\varSigma};\mathfrak{su}(E))\oplus\varGamma(K_{\check{\varSigma}}^{-1}\otimes E)\label{SO(3) VORTEX MAP}\\
(C,\varUpsilon) & \rightarrow\left(*F_{C}^{0}-i\left[\varUpsilon\varUpsilon^{*}-\frac{1}{2}|\varUpsilon|^{2}I_{E}\right],\bar{\partial}_{C}\varUpsilon\right)\nonumber 
\end{align}
}An easy computation shows that:\textbf{ }
\begin{lem}
\label{Lem Linearization vortex map}The linearization of the $SO(3)$
vortex map is
\[
\mathcal{D}\mathfrak{F}_{SO(3),(C,\varUpsilon)}(\dot{c},\dot{\varUpsilon})=\left(*d_{C}\dot{c}-i\left[\varUpsilon\dot{\varUpsilon}^{*}+\dot{\varUpsilon}\varUpsilon^{*}-\text{Re}\left\langle \varUpsilon,\dot{\varUpsilon}\right\rangle I_{E}\right],\bar{\partial}_{C}\dot{\varUpsilon}+\dot{c}''\otimes\varUpsilon\right)
\]
where we have decomposed $\dot{c}\in\varOmega^{1}(\check{\varSigma},\mathfrak{su}(E))\subset\varOmega^{1}(\check{\varSigma},\mathfrak{su}(E)\otimes\mathbb{C})$
as $\dot{c}=\dot{c}'+\dot{c}''$.
\end{lem}

\begin{rem}
Recall that on a Riemann surface the Hodge star operator acting on
one-forms satisfies $*^{2}=-\text{Id}$, so it can be used to define
an almost complex structure on $\varOmega^{1}(\mathfrak{g}_{E})$.
In particular, whenever we write $c=c'+c''$, then $*c=-ic'+ic''$.
\end{rem}

\begin{proof}
It follows easily after observing that for $t\in\mathbb{R}$,
\[
\begin{cases}
|\varUpsilon+t\dot{\varUpsilon}|^{2}=|\varUpsilon|^{2}+t^{2}|\dot{\varUpsilon}|^{2}+2\text{Re}\left\langle \varUpsilon,\dot{\varUpsilon}\right\rangle \\
(\varUpsilon+t\dot{\varUpsilon})(\varUpsilon+t\dot{\varUpsilon})^{*}=\varUpsilon\varUpsilon^{*}+t^{2}\dot{\varUpsilon}\dot{\varUpsilon}^{*}+t\varUpsilon\dot{\varUpsilon}^{*}+t\dot{\varUpsilon}\varUpsilon^{*}
\end{cases}
\]
\end{proof}
\begin{lem}
For $\xi\in\varOmega^{0}(\check{\varSigma};\mathfrak{su}(\check{E}))$,
define 
\[
d_{(C,\varUpsilon)}^{0}(\xi)=(-d_{C}\xi,\xi\varUpsilon)
\]

Then at each solution $(C,\varUpsilon)$ to the $SO(3)$ vortex equations
(\ref{SO(3) VORTEX MAP}), there is a complex 
\[
\varOmega^{0}(\check{\varSigma};\mathfrak{su}(\check{E}))\rightarrow^{d_{(C,\varUpsilon)}^{0}}\varOmega^{1}(\check{\varSigma};\mathfrak{su}(\check{E}))\oplus\varGamma(\check{E})\rightarrow^{D\mathfrak{F}_{SO(3),(C,\varUpsilon)}}\varOmega^{0}(\check{\varSigma};\mathfrak{su}(\check{E}))\oplus\varGamma(K_{\check{\varSigma}}^{-1}\otimes\check{E})
\]
\end{lem}

\begin{proof}
Write $\mathfrak{F}_{SO(3)}=\left(\mathfrak{F}_{SO(3)}^{1},\mathfrak{F}_{SO(3)}^{2}\right)$.
Then just as in \cite[p.312]{MR1664908}, we have that 
\[
\mathcal{D}\mathfrak{F}_{SO(3),(C,\varUpsilon)}\circ d_{(C,\varUpsilon)}^{0}(\xi)=\left([\xi,\mathfrak{F}_{SO(3)}^{1}(C,\varUpsilon)],\xi\mathfrak{F}_{SO(3)}^{2}(C,\varUpsilon)\right)
\]
so the result follows since at a solution to the vortex equations
we have $\mathfrak{F}_{SO(3)}(C,\varUpsilon)=(0,0)$.
\end{proof}
\begin{defn}
Define the $L^{2}$ \textbf{\emph{real}} inner product \emph{\cite[Lemma 9.3.3]{MR2388043}
\begin{equation}
\left\langle (\dot{c}_{1},\dot{\varUpsilon}_{1}),(\dot{c}_{2},\dot{\varUpsilon}_{2})\right\rangle _{L^{2}}\equiv\int_{\check{\varSigma}}\left\langle \dot{c}_{1},\dot{c}_{2}\right\rangle +\text{Re}\left\langle \dot{\varUpsilon}_{1},\dot{\varUpsilon}_{2}\right\rangle \label{Inner product}
\end{equation}
where the point-wise inner product on $\mathfrak{su}(\check{E})$
is given by 
\[
\left\langle \dot{c}_{1},\dot{c}_{2}\right\rangle =\frac{1}{2}\text{Tr}(\dot{c}_{1}^{\dagger}\dot{c}_{2})
\]
Moreover, define 
\begin{equation}
d_{(C,\varUpsilon)}^{0,*}(\dot{c},\dot{\varUpsilon})=-d_{C}^{*}\dot{c}+[\dot{\varUpsilon}\varUpsilon^{*}-\varUpsilon\dot{\varUpsilon}^{*}-i\text{Re}<i\varUpsilon,\dot{\varUpsilon}>I_{E}]\label{Coulomb condition}
\end{equation}
}
\end{defn}

\begin{lem}
If $d_{(C,\varUpsilon)}^{0,*}(c,\dot{\varUpsilon})=0$, then $(\dot{c},\dot{\varUpsilon})$
is $L^{2}$ orthogonal to $\text{Im}d_{(C,\varUpsilon)}^{0}$ .
\end{lem}

\begin{proof}
Observe that 
\begin{align*}
 & \left\langle (-d_{C}\xi,\xi\varUpsilon),(\dot{c},\dot{\varUpsilon})\right\rangle _{L^{2}}\\
= & \left\langle -d_{C}\xi,\dot{c}\right\rangle _{L^{2}}+\left\langle \xi\varUpsilon,\dot{\varUpsilon}\right\rangle _{L^{2}}\\
= & \left\langle \xi,-d_{C}^{*}\dot{c}\right\rangle _{L^{2}}+\left\langle \xi\varUpsilon,\dot{\varUpsilon}\right\rangle _{L^{2}}\\
= & \int_{\varSigma}\left\langle \xi,-d_{C}^{*}\dot{c}\right\rangle +\text{Re}\left\langle \xi\varUpsilon,\dot{\varUpsilon}\right\rangle 
\end{align*}
 So now we need to verify a pointwise identity. First of all, notice
that for two matrices $M_{1},M_{2}$ of $\mathfrak{su}(E)$, their
hermitian inner product is real in the sense that 
\[
\left\langle M_{1},M_{2}\right\rangle =\frac{1}{2}\text{tr}(M_{1}^{\dagger}M_{2})=-\frac{1}{2}\text{tr}(M_{1}M_{2})=-\frac{1}{2}\text{tr}(M_{2}M_{1})=\frac{1}{2}\text{tr}(M_{2}^{\dagger}M_{1})=\left\langle M_{2},M_{1}\right\rangle 
\]
which justifies why there is no need to put $\text{Re}$ in front
of $\left\langle \xi,-d_{C}^{*}\dot{c}\right\rangle $. On the other
hand, we can re-write $\text{Re}\left\langle \xi\varUpsilon,\dot{\varUpsilon}\right\rangle $
as 
\[
\frac{1}{2}\left\langle \xi\varUpsilon,\dot{\varUpsilon}\right\rangle +\frac{1}{2}\left\langle \dot{\varUpsilon},\xi\varUpsilon\right\rangle 
\]
Therefore, the integrand we are considering becomes {[}recall $d_{(C,\varUpsilon)}^{0,*}(c,\dot{\varUpsilon})=0$
so we can substitute $d_{C}^{*}\dot{c}${]}
\begin{align*}
 & \left\langle \xi,-[\dot{\varUpsilon}\varUpsilon^{*}-\varUpsilon\dot{\varUpsilon}^{*}-i\text{Re}<i\varUpsilon,\dot{\varUpsilon}>I_{E}]\right\rangle +\frac{1}{2}\left\langle \xi\varUpsilon,\dot{\varUpsilon}\right\rangle +\frac{1}{2}\left\langle \dot{\varUpsilon},\xi\varUpsilon\right\rangle \\
= & \left\langle \xi,\varUpsilon\dot{\varUpsilon}^{*}\right\rangle -\left\langle \xi,\dot{\varUpsilon}\varUpsilon^{*}\right\rangle +\frac{1}{2}\left\langle \xi\varUpsilon,\dot{\varUpsilon}\right\rangle +\frac{1}{2}\left\langle \dot{\varUpsilon},\xi\varUpsilon\right\rangle \\
= & \frac{1}{2}\left\langle \xi\varUpsilon,\dot{\varUpsilon}\right\rangle -\left\langle \xi,\dot{\varUpsilon}\varUpsilon^{*}\right\rangle +\left\langle \xi,\varUpsilon\dot{\varUpsilon}^{*}\right\rangle -\frac{1}{2}\left\langle \xi\dot{\varUpsilon},\varUpsilon\right\rangle 
\end{align*}
 So really we need to compare the previous quantities. For example,
we want to know whether\textbf{
\begin{equation}
\frac{1}{2}\left\langle \xi\varUpsilon,\dot{\varUpsilon}\right\rangle =^{?}\left\langle \xi,\dot{\varUpsilon}\varUpsilon^{*}\right\rangle \label{eq:QUESTION}
\end{equation}
}This would guarantee the vanishing of the first two terms, and the
other two are essentially the same. We compute this locally. Observe
that we can write 
\begin{align*}
\xi= & \left(\begin{array}{cc}
it & -\bar{z}\\
z & -it
\end{array}\right) & \varUpsilon= & \left(\begin{array}{c}
\varPhi_{1}\\
\varPhi_{2}
\end{array}\right) & \dot{\varUpsilon}= & \left(\begin{array}{c}
\phi_{1}\\
\phi_{2}
\end{array}\right)
\end{align*}
Start with the right hand side 
\begin{alignat*}{1}
 & \left\langle \xi,\dot{\varUpsilon}\varUpsilon^{*}\right\rangle \\
= & -\frac{1}{2}\text{tr}\left(\xi\dot{\varUpsilon}\varUpsilon^{*}\right)\\
= & -\frac{1}{2}\text{tr}\left(\left(\begin{array}{cc}
it & -\bar{z}\\
z & -it
\end{array}\right)\left(\begin{array}{c}
\phi_{1}\\
\phi_{2}
\end{array}\right)\left(\begin{array}{cc}
\bar{\varPhi}_{1} & \bar{\varPhi}_{2}\end{array}\right)\right)\\
= & -\frac{1}{2}\text{tr}\left(\left(\begin{array}{c}
it\phi_{1}-\bar{z}\phi_{2}\\
z\phi_{1}-it\phi_{2}
\end{array}\right)\left(\begin{array}{cc}
\bar{\varPhi}_{1} & \bar{\varPhi}_{2}\end{array}\right)\right)\\
= & -\frac{1}{2}\text{tr}\left(\begin{array}{cc}
it\phi_{1}\bar{\varPhi}_{1} & \text{*}\\
* & -it\phi_{2}\bar{\varPhi}_{2}
\end{array}\right)
\end{alignat*}
where $*$ are things we don't care about since we are just computing
a trace! On the other hand, the left hand side is\textbf{ }
\begin{align*}
 & \frac{1}{2}\left\langle \xi\varUpsilon,\dot{\varUpsilon}\right\rangle \\
= & \frac{1}{2}\varUpsilon^{*}\xi^{*}\dot{\varUpsilon}\\
= & -\frac{1}{2}\left(\begin{array}{cc}
\bar{\varPhi}_{1} & \bar{\varPhi}_{2}\end{array}\right)\left(\begin{array}{cc}
it & -\bar{z}\\
z & -it
\end{array}\right)\left(\begin{array}{c}
\phi_{1}\\
\phi_{2}
\end{array}\right)\\
= & -\frac{1}{2}\left(\begin{array}{cc}
\bar{\varPhi}_{1} & \bar{\varPhi}_{2}\end{array}\right)\left(\begin{array}{c}
it\phi_{1}-\bar{z}\phi_{2}\\
z\phi_{1}-it\phi_{2}
\end{array}\right)\\
= & -\frac{1}{2}\text{tr}\left(\begin{array}{cc}
\bar{\varPhi}_{1}it\phi_{1} & **\\
** & -it\bar{\varPhi}_{2}\phi_{2}
\end{array}\right)
\end{align*}
So the traces of both matrices will be the same.
\end{proof}
\begin{defn}
Suppose that $(C,\varUpsilon)$ is a solution to the $SO(3)$ vortex
equations. Define the cohomology groups 
\[
\begin{cases}
\mathbf{H}_{(C,\varUpsilon)}^{0}\equiv\ker d_{(C,\varUpsilon)}^{0}\\
\mathbf{H}_{(C,\varUpsilon)}^{1}\equiv\ker(d_{(C,\varUpsilon)}^{0,*}\oplus\mathcal{D}\mathfrak{F}_{SO(3),(C,\varUpsilon)})\\
\mathbf{H}_{(C,\varUpsilon)}^{2}\equiv\text{coker}\mathcal{D}\mathfrak{F}_{SO(3),(C,\varUpsilon)}
\end{cases}
\]
\end{defn}

\begin{lem}
\label{lem:Suppose-that-smoothness criterion}Suppose that $(C,\varUpsilon)$
is a solution to the $SO(3)$ vortex equations satisfying:

i) $\mathbf{H}_{(C,\varUpsilon)}^{0}\equiv0$.

ii) $\bar{\partial}_{C}:\varGamma(\check{E})\rightarrow\varGamma(K_{\check{\varSigma}}^{-1}\otimes\check{E})$
is surjective or alternatively $(C,\varUpsilon)$ is irreducible,
in the sense that $\varUpsilon$ is not identically zero and $C$
is an irreducible connection.

Then $\mathbf{H}_{(C,\varUpsilon)}^{2}\equiv0$ and near $[C,\varUpsilon]\in\mathcal{M}(\check{\varSigma},\check{E})$
the moduli space has the structure of a smooth manifold of real dimension
\begin{equation}
\dim_{\mathbb{R}}\mathcal{M}(\check{\varSigma},\check{E})=\dim\mathbf{H}_{[C,\varUpsilon]}^{1}=2(g-1)+2c_{1}(\det\check{E})+2(n-n_{0})-2\sum_{i=1}^{n}\frac{b_{i}^{-}+b_{i}^{+}}{a_{i}}\label{dimension moduli space}
\end{equation}
\end{lem}

\begin{proof}
In order to show that 
\[
\mathcal{D}\mathfrak{F}_{SO(3),(C,\varUpsilon)}:\varOmega^{1}(\check{\varSigma};\mathfrak{su}(\check{E}))\oplus\varGamma(\check{E})\rightarrow\varOmega^{0}(\check{\varSigma};\mathfrak{su}(\check{E}))\oplus\varGamma(K_{\check{\varSigma}}^{-1}\otimes\check{E})
\]
is surjective we follow a similar argument to the one given by Hitchin
in \cite[Section 5]{MR887284}. First of all, we write 
\[
\mathcal{D}\mathfrak{F}_{SO(3),(C,\varUpsilon)}=\left(\mathcal{D}\mathfrak{F}_{SO(3),(C,\varUpsilon)}^{1},\mathcal{D}\mathfrak{F}_{SO(3),(C,\varUpsilon)}^{2}\right)
\]
A simple calculation shows that if we define $\mathbf{J}_{(C,\varUpsilon)}(\dot{c},\dot{\varUpsilon})=(*\dot{c},i\dot{\varUpsilon})$,
then 
\[
\mathcal{D}\mathfrak{F}_{SO(3),(C,\varUpsilon)}^{1}(\dot{c},\dot{\varUpsilon})=-d_{(C,\varUpsilon)}^{0,*}\mathbf{J}(\dot{c},\dot{\varUpsilon})
\]
 This computation is implicitly done in Lemma (\ref{lem:almost complex structure}),
where we show that $\mathbf{J}$ defines an almost complex structure
on the Zariski tangent space. Therefore, the formal adjoint of $\mathcal{D}\mathfrak{F}_{SO(3),(C,\varUpsilon)}$
is (recall $\mathbf{J}^{*}=-\mathbf{J}$)
\[
\mathcal{D}\mathfrak{F}_{SO(3),(C,\varUpsilon)}^{*}=\left(\mathbf{J}\circ d_{(C,\varUpsilon)}^{0}\right)\oplus,\mathcal{D}\mathfrak{F}_{SO(3),(C,\varUpsilon)}^{2,*}:\varOmega^{0}(\check{\varSigma};\mathfrak{su}(\check{E}))\oplus\varGamma(K_{\check{\varSigma}}^{-1}\otimes\check{E})\rightarrow\varOmega^{1}(\check{\varSigma};\mathfrak{su}(\check{E}))\oplus\varGamma(\check{E})
\]
Now suppose that $(\tilde{\xi},\tilde{\varUpsilon})$ is such that
\[
\begin{cases}
\mathbf{J}\circ d_{(C,\varUpsilon)}^{0}\tilde{\xi}=0\\
\mathcal{D}\mathfrak{F}_{SO(3),(C,\varUpsilon)}^{2,*}\tilde{\varUpsilon}=0
\end{cases}
\]
Since $\mathbf{J}^{2}=-\text{Id}$, the first equation says that $\tilde{\xi}\in\ker d_{(C,\varUpsilon)}^{0}$,
which vanishes by assumption i). Therefore $\tilde{\xi}=0$. 

Likewise, the second equation says that for all $(\dot{c},\dot{\varUpsilon})\in\varOmega^{1}(\check{\varSigma};\mathfrak{su}(\check{E}))\oplus\varGamma(\check{E})$
, 
\[
<\bar{\partial}_{C}\dot{\varUpsilon}+\dot{c}''\otimes\varUpsilon,\tilde{\varUpsilon}>_{L^{2}}=0
\]
Taking $\dot{\varUpsilon}=\varUpsilon$ this means that $<\dot{c}''\otimes\varUpsilon,\tilde{\varUpsilon}>_{L^{2}}=0$
since $\bar{\partial}_{C}\varUpsilon=0$. Hence, we obtain the condition
\[
<\bar{\partial}_{C}\dot{\varUpsilon},\tilde{\varUpsilon}>_{L^{2}}=0\iff<\dot{\varUpsilon},\bar{\partial}_{C}^{*}\tilde{\varUpsilon}>_{L^{2}}=0\iff\tilde{\varUpsilon}\in\ker\bar{\partial}_{C}^{*}
\]
Therefore $\tilde{\varUpsilon}$ is zero if we assume that $\bar{\partial}_{C}$
is surjective. If we no longer assume that $\bar{\partial}_{C}$ is
surjective, notice that from equation (\ref{eq:QUESTION}) $<\dot{c}''\otimes\varUpsilon,\tilde{\varUpsilon}>_{L^{2}}=0$
is equivalent to $<\dot{c}'',(\tilde{\varUpsilon}\varUpsilon^{*})_{0}>_{L^{2}}=0$,
thus we find that
\[
\tilde{\varUpsilon}\varUpsilon^{*}=\frac{1}{2}\text{tr}(\tilde{\varUpsilon}\varUpsilon^{*})
\]
Locally, we can write $\varUpsilon=\left(\begin{array}{c}
\varPhi_{1}\\
\varPhi_{2}
\end{array}\right)$ and $\tilde{\varUpsilon}=\left(\begin{array}{c}
\phi_{1}\\
\phi_{2}
\end{array}\right)$ , the previous equation is equivalent to 
\[
\left(\begin{array}{cc}
\phi_{1}\bar{\varPhi}_{1} & \phi_{1}\bar{\varPhi}_{2}\\
\phi_{2}\bar{\Phi}_{1} & \phi_{2}\bar{\varPhi}_{2}
\end{array}\right)=\left(\begin{array}{cc}
\frac{1}{2}\left(\phi_{1}\bar{\varPhi}_{1}+\phi_{2}\bar{\varPhi}_{2}\right) & 0\\
0 & \frac{1}{2}\left(\phi_{1}\bar{\varPhi}_{1}+\phi_{2}\bar{\varPhi}_{2}\right)
\end{array}\right)
\]
 Therefore, we conclude that 
\[
\begin{cases}
\phi_{1}\bar{\varPhi}_{2}=0\\
\phi_{2}\bar{\Phi}_{1}=0
\end{cases}
\]
Now, if $\varUpsilon$ is not a section of some splitting holomorphic
sub-bundle of $\check{E}$, then we must conclude that $\phi_{1}$
and $\phi_{2}$ must vanish on the open sets where $\varPhi_{1},\varPhi_{2}$
are non-vanishing. But since $*\tilde{\varUpsilon}$ is a holomorphic
section, this can only happen if $\tilde{\varUpsilon}$ vanishes.

To find the expected dimension of the moduli space at the solution
$[C,\varUpsilon]$, we need to compute 
\[
\text{ind}(d_{(C,\varUpsilon)}^{0,*}\oplus\mathcal{D}\mathfrak{F}_{SO(3),(C,\varUpsilon)})
\]
This operator is homotopic to the operator 
\[
(d_{C}^{0,*}\oplus*d_{C})\oplus\bar{\partial}_{C}
\]
so the index is equal to 
\[
\text{ind}(d_{C}^{0,*}\oplus*d_{C})+\text{ind}\bar{\partial}_{C}
\]

From the Riemann-Roch formula (\ref{Riemann Roch U(2)}), we know
that 
\[
\text{ind}_{\mathbb{C}}\bar{\partial}_{C}=2(1-g)+c_{1}(\det\check{E})-\sum_{i=1}^{n}\frac{b_{i}^{-}+b_{i}^{+}}{a_{i}}
\]
To compute $\text{ind}(d_{C}^{0,*}\oplus*d_{C})$, we use the Dolbeault
model, and compute this index as 
\[
\chi(\check{\varSigma};\mathfrak{su}(\check{E})\otimes\mathbb{C})
\]
Just as we found Riemann-Roch for $U(2)$ orbifold bundles, suppose
that 
\[
\check{E}=\check{L}^{-}\oplus\check{L}^{+}
\]
Then 
\[
\mathfrak{su}(\check{E})\otimes\mathbb{C}\simeq\mathbb{C}\oplus(\check{L}^{-}\otimes\check{L}^{+*})\oplus(\check{L}^{-*}\otimes\check{L}^{+})
\]
Therefore, 
\[
\chi(\check{\varSigma};\mathfrak{su}(\check{E})\otimes\mathbb{C})=\chi(\check{\varSigma};\mathbb{C})+\chi(\check{\varSigma};\check{L}^{-}\otimes\check{L}^{+*})+\chi(\check{\varSigma};\check{L}^{-*}\otimes\check{L}^{+})
\]
Recall that the orbifold cohomology groups with complex coefficients
are isomorphic to the ordinary de Rham cohomology groups, in particular,
this means that 
\[
\chi(\check{\varSigma};\mathbb{C})=1-g
\]
To compute $\chi(\check{\varSigma};\check{L}^{-}\otimes\check{L}^{+*})+\chi(\check{\varSigma};\check{L}^{-*}\otimes\check{L}^{+})$
, we need to understand the isotropies of the line bundles $L=\check{L}^{-}\otimes\check{L}^{+*}$
and $L^{*}=\check{L}^{-*}\otimes\check{L}^{+}$.

Suppose that at a particular marked point, $b_{i}^{+}\neq b_{i}^{-}$.
Then $L$ and $L^{*}$ have non-trivial isotropies. If $L$ has isotropy
$b_{i,L}$ at the marked point $p_{i}$, then we can take the isotropy
$b_{i,L^{*}}$ of $L^{*}$ to be $b_{i,L^{*}}=a_{i}-b_{i,L}$. Therefore,
in the formula (\ref{Riemann Roch Line Bundles}) for Riemann-Roch
for orbifold line bundles, the terms $\frac{b_{i,L}}{a_{i}}$ and
$\frac{b_{i,L^{*}}}{a_{i}}$ will add up to $1$. 

On the other hand, if one looks at a marked point where $b_{i}^{+}=b_{i}^{-}$,
then, the isotropy of $L$ and $L^{*}$ become zero at this point.

In this way we obtain 
\[
\chi(\check{\varSigma};L)+\chi(\check{\varSigma};L^{*})=2(1-g)-(n-n_{0})
\]
adding the Euler characteristics we find 
\[
\chi(\check{\varSigma};\mathfrak{su}(\check{E})\otimes\mathbb{C})=3(1-g)-(n-n_{0})
\]
So the \emph{real }expected dimension of the moduli space is 
\begin{align*}
 & \dim_{\mathbb{R}}\mathcal{M}(\check{\varSigma},\check{E})\\
= & 6(g-1)+2(n-n_{0})+4(1-g)+2c_{1}(\det\check{E})-2\sum_{i=1}^{n}\frac{b_{i}^{-}+b_{i}^{+}}{a_{i}}\\
= & 2(g-1)+2c_{1}(\det\check{E})+2(n-n_{0})-2\sum_{i=1}^{n}\frac{b_{i}^{-}+b_{i}^{+}}{a_{i}}
\end{align*}
\end{proof}
\begin{rem}
a) In \cite{MR1124279}, Bradlow and Daskalopoulos restrict themselves
to the stable pairs $(\mathcal{E},\varUpsilon)$ where the bundle
$\mathcal{E}$ is semi-stable (here $\mathcal{E}$ denotes $E$ with
a particular holomorphic structure on it). Then they assume $\deg E$
is sufficiently large so that $\bar{\partial}_{C}$ becomes surjective
\cite[Proposition 1.8]{MR1250254}. In fact the proof of our previous
lemma should be compared with their proof of Proposion 2.1, which
is quite similar. Thaddeus also verifies a similar condition in \cite[(1.10)]{MR1273268}.

b) As a sanity check, notice that the real dimension for the moduli
space of stable pairs in Thaddeus paper is \cite[Equation (2.2)]{MR1273268}
$2g-4+2\deg E$, while in our case the expected dimension formula
seems to be $2g-2+2\deg E$. In other words, our dimension formula
is larger than the one Thaddeus gives, so why is that? Basically,
we still haven't taken the quotient by a residual $S^{1}$ action
on the moduli space of stable pairs, which will bring our dimension
down to $2g-3+2\deg E$. Furthermore, Thaddeus is working with a \emph{fixed
}value for the stability parameter $\tau$, which as we explain in
the next section roughly corresponds to fixing a level set for the
moment map function $\mu$ which is defined at the end of this section. 

c) In the next section, we will explain why further assumptions on
$c_{1}(\check{E})$ are needed in order to make the moduli space smooth
at the reducibles (fixed points of the circle action). \textbf{}
\end{rem}

As mentioned in the introduction, we can define a circle action on
the moduli space given by 
\begin{equation}
e^{i\theta}\cdot(C,\varUpsilon)=(C,e^{i\theta}\varUpsilon)\label{Circle action}
\end{equation}
Just as in the case of \cite[Proposition 7.1]{MR887284}, or \cite[Proposition 3.1]{MR1855754}
the fixed points \emph{modulo }gauge correspond to projectively flat
connections or solutions to the abelian vortex equations.
\begin{lem}
A fixed point $[C,\varUpsilon]$ to the circle action takes the form:

i) A projectively flat connection $[C,0]$.

ii) A reducible solution $[C_{L}\oplus(C_{L}^{*}\otimes C^{\det}),\alpha\oplus0]$
compatible with a splitting $\check{E}=\check{L}\oplus(\check{L}^{*}\otimes\det\check{E})$,
where $(C_{L},\alpha)$ satisfies the abelian vortex equation 
\[
\text{abelian vortex equation}\begin{cases}
*F_{C_{L}}-\frac{i}{2}|\alpha|^{2}=*\frac{1}{2}F_{C^{\det}}\\
\bar{\partial}_{C_{L}}\alpha=0
\end{cases}
\]
Moreover, $\alpha$ cannot vanish identically if either of the following
conditions is satisfied:

I) In the smooth case, $c_{1}(\det E)$ is of odd degree.

II) In the orbifold case $\check{\varSigma}=(\varSigma,p_{1},\cdots,p_{n})$
with the multiplicities $a_{i}$ of $p_{i}$ mutually coprime, we
have $\det\check{E}$ is an odd integer power of the fundamental orbifold
line bundle $\check{L}_{0}$ satisfying $c_{1}(\check{L}_{0})=\frac{1}{a_{1}\cdots a_{n}}$
.

\end{lem}

\begin{rem}
When we have a splitting $\check{E}=\check{L}\oplus(\check{L}^{*}\otimes\det\check{E})$,
it is assumed implicitly that $\check{L}$ has the correct isotropy
as a subbundle of $\check{E}$ as discussed before.
\end{rem}

\begin{proof}
Case i) is clear. 

ii) If $C$ is reducible, then there is a splitting $\check{E}=\check{L}\oplus(\check{L}^{*}\otimes\det\check{E})$
of the $U(2)$ bundle compatible with the reduction $C=C_{L}\oplus(C_{L}^{*}\otimes C^{\det})$,
where $C_{L}$ denotes a $U(1)$ connection on $\check{L}$. Likewise,
we can decompose $\varUpsilon$ into components $\varUpsilon=\alpha\oplus\beta$.
The curvature decomposes as 
\[
F_{C}=\left(\begin{array}{cc}
F_{C_{L}} & 0\\
0 & F_{C^{\det}}-F_{C_{L}}
\end{array}\right)
\]
so the traceless part is
\[
F_{C}^{0}=F_{C}-\frac{1}{2}\text{tr}(F_{C})1_{E}=\left(\begin{array}{cc}
F_{C_{L}}-\frac{1}{2}F_{C^{\det}} & 0\\
0 & \frac{1}{2}F_{C^{\det}}-\frac{1}{2}F_{C_{L}}
\end{array}\right)
\]
The other term of the curvature equation is 
\[
\left(\begin{array}{c}
\alpha\\
\beta
\end{array}\right)\left(\begin{array}{cc}
\bar{\alpha} & \bar{\beta}\end{array}\right)=\left(\begin{array}{cc}
|\alpha|^{2} & \alpha\bar{\beta}\\
\beta\bar{\alpha} & |\beta|^{2}
\end{array}\right)
\]
So 
\[
\varUpsilon\varUpsilon^{*}-\frac{1}{2}|\varUpsilon|^{2}I_{E}=\left(\begin{array}{cc}
\frac{1}{2}|\alpha|^{2}-\frac{1}{2}|\beta|^{2} & \alpha\bar{\beta}\\
\beta\bar{\alpha} & \frac{1}{2}|\beta|^{2}-\frac{1}{2}|\alpha|^{2}
\end{array}\right)
\]
Therefore, the $SO(3)$ vortex equations (\ref{SO(3) VORTEX MAP})
at reducible solutions can be expressed as 
\[
\begin{cases}
*F_{C_{L}}-*\frac{1}{2}F_{C^{\det}}-\frac{i}{2}|\alpha|^{2}+\frac{i}{2}|\beta|^{2}=0\\
\alpha\bar{\beta}=0\\
\bar{\partial}_{C_{L}}\alpha=0\\
\bar{\partial}_{C_{L}^{*}\otimes C^{\det}}\beta=0
\end{cases}
\]
The last two equations say that $\alpha,\beta$ are holomorphic sections,
so by the unique continuation principle, if they vanish on an open
set, they must vanish identically. Using the second equation we can
conclude that \emph{at least }one of the two sections must vanish.
Without loss of generality we can assume that $\beta$ vanishes (see
next remark). Therefore, $\alpha$ must solve the vortex equation
\begin{equation}
U(1)\text{ vortex equations }\begin{cases}
*F_{C_{L}}-\frac{i}{2}|\alpha|^{2}=*\frac{1}{2}F_{C^{\det}}\\
\bar{\partial}_{C_{L}}\alpha=0
\end{cases}\label{abelian vortex equation}
\end{equation}
Up to this point, it is not impossible for $\alpha$ to vanish identically.
If this happens, then we obtain $2F_{C_{\check{L}}}=F_{C^{\det}}$,
and so Chern-Weil theory says that 
\begin{equation}
c_{1}(\det\check{E})=2c_{1}(\check{L})\label{Condition Reducibles}
\end{equation}
Recall that in this case we are referring to the orbifold Chern classes,
which are \emph{rational }cohomology classes in general. 

a) In the smooth case, the condition (\ref{Condition Reducibles})
is not satisfied if we assume that $c_{1}(\det E)$ is of \emph{odd
degree. }

b) In the orbifold case, the classification of topological line bundles
is more complicated as explained before (see also \cite[Proposition 1.4]{MR1185787}).
However, if we assume that $a_{1},\cdots,a_{n}$ are coprime, then
we know that any line bundle has the form $\check{L}_{0}^{k}$ for
some integer $k$. Therefore, $\det\check{E}=\check{L}_{0}^{k}$ for
some $k$ and $\check{L}=\check{L}_{0}^{l}$ for some $l$. Since
$c_{1}$ continues to be a homomorphism, $c_{1}(\det\check{E})=kc_{1}(\check{L}_{0})$
and $c_{1}(\check{L})=lc_{1}(\check{L}_{0})$, so (\ref{Condition Reducibles})
is equivalent to 
\[
k=2l
\]
which again does not have any solution if we assume that $k$ is an
odd integer. 
\end{proof}
\begin{rem}
Notice that in the proof of the previous lemma we said that we could
assume $\beta$ to vanish. In reality this depends on the fact that
we are ultimately interested in studying the solutions \emph{up to
gauge.}

As mentioned before, a gauge transformation which swaps the factors
$\check{L},(\check{L}^{*}\otimes\det\check{E})$ is locally of the
form $\left(\begin{array}{cc}
0 & 1\\
1 & 0
\end{array}\right)$, which is \emph{not }of determinant one. However, the section $\left(\begin{array}{c}
\alpha\\
0
\end{array}\right)$ is gauge equivalent to the section $\left(\begin{array}{c}
0\\
-\alpha
\end{array}\right)$ if we use the matrix $\left(\begin{array}{cc}
0 & 1\\
-1 & 0
\end{array}\right)$, which is of determinant one. Since eventually we are interested
in the moduli space after the quotient by the residual $S^{1}$ action
is taken, then $\left(\begin{array}{c}
0\\
-\alpha
\end{array}\right)$ will be equivalent to $\left(\begin{array}{c}
0\\
\alpha
\end{array}\right)$, so it is possible to swap the factors after the full symmetry is
taken into account. Therefore, for our purposes no generality is lost
by assuming that $\beta$ is the section which vanishes.
\end{rem}

\textbf{}

\begin{lem}
If $(C,\varUpsilon)$ is a solution to the $SO(3)$ vortex equations
(\ref{SO(3) VORTEX MAP}) then $\mathbf{H}_{(C,\varUpsilon)}^{0}$
vanishes identically except when $C$ is a reducible connection and
$\varUpsilon$ vanishes identically.
\end{lem}

\begin{proof}
Recall that $d_{(C,\varUpsilon)}^{0}(\xi)=(-d_{C}\xi,\xi\varUpsilon)$
and that $\mathbf{H}_{(C,\varUpsilon)}^{0}=\ker d_{(C,\varUpsilon)}^{0}$.
It is clear that when $C$ is a reducible connection and $\varUpsilon$
vanishes identically, then $\dim\mathbf{H}_{(C,\varUpsilon)}^{0}\geq1$
{[}it has at least an $S^{1}$ stabilizer{]}.

If $C$ is an irreducible connection, then $\ker d_{C}=\{0\}$ so
$\mathbf{H}_{(C,\varUpsilon)}^{0}$ vanishes. 

If $C$ reduces to $C=C_{L}\oplus(C_{L}^{*}\otimes C^{\det})$, then
$C_{L}$ has at least an $S^{1}$ stabilizer which acts on $\varUpsilon=\alpha\oplus0$
by multiplication. Therefore, if $\alpha$ does not vanish identically,
the action of $S^{1}$ on $(C_{L},\alpha)$ will be trivial, and the
result follows.
\end{proof}
In the next sections the following function (moment map) will be very
important.
\begin{defn}
Define the \textbf{Higgs strength }function 
\begin{align*}
\mu: & \mathcal{A}^{\det}(\check{E})\times\varGamma(\check{E})\rightarrow\mathbb{R}\\
 & (C,\varUpsilon)\rightarrow\frac{1}{2}\|\varUpsilon\|_{L^{2}(\check{\varSigma})}^{2}
\end{align*}
\end{defn}

\begin{lem}
\label{Lem degree constraint}Suppose that $(C_{L},\alpha)$ is an
abelian vortex, i.e, it solves equation (\ref{abelian vortex equation}).
Then 
\[
\mu(C_{L},\alpha)=\pi(c_{1}(\check{E})-2c_{1}(\check{L}))
\]
In particular, an abelian vortex can only appear in a reduction of
the bundle $\check{E}$ if $c_{1}(\check{L})\leq\frac{1}{2}c_{1}(\check{E})$.
\end{lem}

\begin{proof}
The curvature equation $*F_{C_{L}}-\frac{i}{2}|\alpha|^{2}=*\frac{1}{2}F_{C^{\det}}$
is equivalent to 
\[
\frac{*iF_{C_{L}}}{2\pi}+\frac{1}{4\pi}|\alpha|^{2}=\frac{*iF_{C^{\det}}}{4\pi}
\]
and after integrating of $\check{\varSigma}$ we conclude that 
\[
\frac{1}{4\pi}\int|\alpha|^{2}=\frac{1}{2}c_{1}(\check{E})-c_{1}(\check{L})
\]
from which the result follows.
\end{proof}
Now we discuss the interpretation of $\mu$ as a moment map. For this
we need to introduce other structures on the moduli space. Recall
that in equation (\ref{Inner product}) we defined the real inner
product 
\[
\mathbf{g}_{(C,\varUpsilon)}((\dot{c}_{1},\dot{\varUpsilon}_{1}),(\dot{c}_{2},\dot{\varUpsilon}_{2}))=\int_{\check{\varSigma}}\left\langle \dot{c}_{1},\dot{c}_{2}\right\rangle +\text{Re}\left\langle \dot{\varUpsilon}_{1},\dot{\varUpsilon}_{2}\right\rangle 
\]
With respect to this inner product we have:
\begin{lem}
\label{lem:almost complex structure}The endomorphism of $\mathbf{H}_{(C,\varUpsilon)}^{1}=\ker(d_{(C,\varUpsilon)}^{0,*}\oplus\mathcal{D}\mathfrak{F}_{SO(3),(C,\varUpsilon)})$
given by 
\[
\mathbf{J}_{(C,\varUpsilon)}(\dot{c},\dot{\varUpsilon})=(*\dot{c},i\dot{\varUpsilon})
\]
defines an almost complex structure on $\mathbf{H}_{(C,\varUpsilon)}^{1}$
which is $\mathbf{g}$ -orthogonal, i.e, $\mathbf{g}(\mathbf{J}\bullet,\mathbf{J}\bullet)=\mathbf{g}(\bullet,\bullet)$. 
\end{lem}

\begin{proof}
Checking that $\mathbf{J}^{2}=-\mathbf{Id}$ is straightforward as
well as the compatibility with $\mathbf{g}$. The only interesting
thing to verify is that $\mathbf{H}_{(C,\varUpsilon)}^{1}$ is preserved
by this almost-complex structure. Recall that $\mathbf{H}_{(C,\varUpsilon)}^{1}$
consisted of the equations 
\[
\begin{cases}
-d_{C}^{*}\dot{c}+[\dot{\varUpsilon}\varUpsilon^{*}-\varUpsilon\dot{\varUpsilon}^{*}-i\text{Re}<i\varUpsilon,\dot{\varUpsilon}>I_{E}]=0\\
*d_{C}\dot{c}-i\left[\varUpsilon\dot{\varUpsilon}^{*}+\dot{\varUpsilon}\varUpsilon^{*}-\text{Re}\left\langle \varUpsilon,\dot{\varUpsilon}\right\rangle I_{E}\right]=0\\
\bar{\partial}_{C}\dot{\varUpsilon}+\dot{c}''\otimes\varUpsilon=0
\end{cases}
\]
Since $*\dot{c}''=i\dot{c}''$ , it is straightforward to check that
the last equation is preserved under $\mathbf{J}$. For the first
equation,
\begin{align*}
 & -d_{C}^{*}(*\dot{c})+[(i\dot{\varUpsilon})\varUpsilon^{*}-\varUpsilon(i\dot{\varUpsilon})^{*}-i\text{Re}<i\varUpsilon,i\dot{\varUpsilon}>I_{E}]\\
= & *d_{C}(**)\dot{c}+i[\dot{\varUpsilon}\varUpsilon^{*}+\varUpsilon\dot{\varUpsilon}{}^{*}-\text{Re}<i\varUpsilon,i\dot{\varUpsilon}>I_{E}]\\
= & -\left(*d_{C}\dot{c}-i\left[\varUpsilon\dot{\varUpsilon}^{*}+\dot{\varUpsilon}\varUpsilon^{*}-\text{Re}\left\langle \varUpsilon,\dot{\varUpsilon}\right\rangle I_{E}\right]\right)\\
= & 0
\end{align*}
 Conversely, 
\begin{align*}
 & *d_{C}(*\dot{c})-i\left[\varUpsilon(i\dot{\varUpsilon})^{*}+(i\dot{\varUpsilon})\varUpsilon^{*}-\text{Re}\left\langle \varUpsilon,i\dot{\varUpsilon}\right\rangle I_{E}\right]\\
= & -d_{C}^{*}\dot{c}-i\left[-i\varUpsilon\dot{\varUpsilon}+i\dot{\varUpsilon}\varUpsilon^{*}+\text{Re}\left\langle i\varUpsilon,\dot{\varUpsilon}\right\rangle I_{E}\right]\\
= & -d_{C}^{*}\dot{c}+\left[-\varUpsilon\dot{\varUpsilon}+\dot{\varUpsilon}\varUpsilon^{*}-i\text{Re}\left\langle i\varUpsilon,\dot{\varUpsilon}\right\rangle I_{E}\right]\\
= & 0
\end{align*}
 So the claim is verified.
\end{proof}
Therefore, we can define the 2-form 
\[
\boldsymbol{\varOmega}(\bullet,\bullet)=\mathbf{g}(\bullet,\mathbf{J}\bullet)
\]
 Since the formulas for $\mathbf{g},\mathbf{J},\boldsymbol{\varOmega}$
are base-point independent, it is clear that $\boldsymbol{\varOmega}$
will be a closed two-form, and $\mathbf{J}$ being integrable implies
that we have a Kahler form on the moduli space. 

Now consider the vector field 
\[
\boldsymbol{\partial}_{\theta,(C,\varUpsilon)}=(0,i\varUpsilon)
\]
This is clearly the vector field associated to the $S^{1}$ action
defined before in equation (\ref{Circle action}). 

Moreover, notice that 
\begin{align*}
 & (\imath_{\boldsymbol{\partial}_{\theta}}\boldsymbol{\varOmega})(\dot{c},\dot{\varUpsilon})\\
= & \boldsymbol{\varOmega}((0,i\varUpsilon),(\dot{c},\dot{\varUpsilon}))\\
= & \int\text{Re}\left\langle i\varUpsilon,i\dot{\varUpsilon}\right\rangle \\
= & \int\text{Re}\left\langle \varUpsilon,\dot{\varUpsilon}\right\rangle \\
= & d\mu_{(C,\varUpsilon)}(\dot{c},\dot{\varUpsilon})
\end{align*}
Where $\mu$ was the Higgs strength function. This is precisely the
condition that guarantees that $\mu$ is the moment map associated
to the $S^{1}$ action \cite[Chapter 2]{MR1929136}. 

However, notice that $(0,i\varUpsilon)$ does not belong to $\mathbf{H}_{(C,\varUpsilon)}^{1}$!
Indeed, if one looks at whether $(0,i\varUpsilon)$ satisfies the
first equation for $\mathbf{H}_{(C,\varUpsilon)}^{1}$, we see that
\begin{align*}
 & (i\varUpsilon)\varUpsilon^{*}-\varUpsilon(i\varUpsilon)^{*}-i\text{Re}<i\varUpsilon,i\varUpsilon>I_{E}\\
= & 2i\varUpsilon\varUpsilon^{*}-i|\varUpsilon|^{2}I_{E}
\end{align*}
and this clearly need not vanish. However, since $(0,i\varUpsilon)$
does satisfy the other two equations defining $\mathbf{H}_{(C,\varUpsilon)}^{1}$,
this just means that the representative for $(0,i\varUpsilon)$ in
$\mathbf{H}_{(C,\varUpsilon)}^{1}$ is of the form $(0,i\varUpsilon)+d_{(C,\varUpsilon)}^{0}\xi$,
for some $\xi\in\varOmega^{0}(\check{\varSigma},\mathfrak{su}(E))$. 

Therefore, in order to check that the moment map condition holds with
respect to the vector $(0,i\varUpsilon)+d_{(C,\varUpsilon)}^{0}\xi$,
it suffices to know that for any $(\dot{c},\dot{\varUpsilon})\in\mathbf{H}_{(C,\varUpsilon)}^{1}$,
we have 
\[
\boldsymbol{\varOmega}(d_{(C,\varUpsilon)}^{0}\xi,(\dot{c},\dot{\varUpsilon}))=0
\]
See also the discussion in \cite[Section iv), page 294]{MR1066174}.
In our case, this condition is satisfied, since
\begin{align*}
 & \boldsymbol{\varOmega}(d_{(C,\varUpsilon)}^{0}\xi,(\dot{c},\dot{\varUpsilon}))\\
= & \boldsymbol{g}(d_{(C,\varUpsilon)}^{0}\xi,\mathbf{J}(\dot{c},\dot{\varUpsilon}))\\
= & \left\langle \xi,d_{(C,\varUpsilon)}^{0,*}\mathbf{J}(\dot{c},\dot{\varUpsilon})\right\rangle \\
= & 0
\end{align*}
 The last equality is true since $\mathbf{J}(\dot{c},\dot{\varUpsilon})\in\mathbf{H}_{(C,\varUpsilon)}^{1}$,
and therefore it satisfies the Coulomb condition. Summarizing our
findings we have found:
\begin{thm}
\label{moment map theorem}The two-form $\boldsymbol{\varOmega}$
is a Kahler form on the smooth points of the $SO(3)$ vortex moduli
space, and $\mu$ is a moment map for the $S^{1}$ action defined
in (\ref{Circle action}).
\end{thm}

In the next few sections we will study some properties of the map
$\mu$ using this moment map interpretation. 

\section{\label{sec:-Vortices-and Stable Pairs}$SO(3)$ Vortices and Stable
Pairs}

In order to study the moment map $\mu$ more easily, we will now compare
in this section and the next the $SO(3)$ vortex equations to the
stable pairs equations \cite{MR1273268,MR1085139,MR1124279,MR1396775,MR1297851}.

A \textbf{stable pair} $(\bar{\partial}_{C},\varUpsilon)$ consists
for us of a holomorphic structure $\bar{\partial}_{C}$ on the orbifold
bundle $\check{E}$, together with a holomorphic section $\varUpsilon$
of $\check{E}$, that is, $\bar{\partial}_{C}\varUpsilon=0$. When
convenient, we will sometimes write the stable pair as $(\check{\mathcal{E}},\varUpsilon)$,
where $\check{\mathcal{E}}$ refers to the bundle $\check{E}$ endowed
with the holomorphic structure given by $\bar{\partial}_{C}$. 

The complex gauge group $\mathcal{G}_{\mathbb{C}}^{\det}(\check{E})$
acts on the space of stable pairs via 
\[
u\cdot(\bar{\partial}_{C},\varUpsilon)=(u\circ\bar{\partial}_{C}\circ u^{-1},u\varUpsilon)=(\bar{\partial}_{C}-(\bar{\partial}_{C}u)u^{-1},u\varUpsilon)
\]
Following \cite[Section 6.4]{MR1079726}, we can identify the tangent
space to the space of $\bar{\partial}$ operators with $\varOmega^{0,1}(\text{End}_{0}\check{E})$,
where as usual $\text{End}_{0}\check{E}$ represents the bundle of
trace free endomorphisms of $\check{E}$ (this is because we are working
with bundles having a fixed determinant). The linearization of the
derivative of the complex gauge group action is $\bar{\partial}_{C}$,
and so we have:
\begin{lem}
The deformation theorem for the stable pair equation is determined
by the cohomolology groups of the complex
\[
0\rightarrow\varOmega^{0}(\check{\varSigma},\text{End}_{0}\check{E})\rightarrow^{d_{1}}\varOmega^{0,1}(\check{\varSigma},\text{End}_{0}\check{E})\oplus\varOmega^{0}(\check{\varSigma},\check{E})\rightarrow^{d_{2}}\varOmega^{0,1}(\check{E})\rightarrow0
\]
where 
\[
\begin{cases}
d_{1}\xi=(-\bar{\partial}_{C}\xi,\xi\varUpsilon)\\
d_{2}(\dot{c}'',\dot{\varUpsilon})=\bar{\partial}_{C}\dot{\varUpsilon}+\dot{c}''\varUpsilon
\end{cases}
\]
\end{lem}

\begin{proof}
This is a simple computation, and essentially it is the orbifold version
of \cite[Proposition 2.1]{MR1124279}
\end{proof}
Now we review the relevant version of the Hitchin-Kobayashi correspondence
between solutions to the $SO(3)$ vortex equations and stable pairs.
This correspondence has found different instantiations, and the orbifold
version is a trivial extension of the argument given in \cite[Theorem 1.4.2]{MR2254074,MR2700715},
since integration by parts and the Kahler identities continue to work
on orbifolds. 

Observe first of all that if $(C,\varUpsilon)$ is an $SO(3)$ vortex,
then $C$ naturally defines a holomorphic structure $\bar{\partial}_{C}$,
and $(\bar{\partial}_{C},\varUpsilon)$ is a stable pair. Since $\mathcal{G}^{\det}(\check{E})$
is a subgroup of $\mathcal{G}_{\mathbb{C}}^{\det}(\check{E})$, this
means that there is a well-defined map
\begin{align}
\mathcal{M}^{vor}(\check{\varSigma},\check{E})\rightarrow\mathcal{M}^{st}(\check{\varSigma},\check{E})\label{eq: KH map}\\{}
[C,\varUpsilon]_{\mathbb{R}}\rightarrow[\bar{\partial}_{C},\varUpsilon]_{\mathbb{C}}\nonumber 
\end{align}
 where $[\cdot]_{\mathbb{R}}$ denotes the gauge equivalence class
with respect to the ``real'' gauge group $\mathcal{G}^{\det}(\check{E})$
, while $[\cdot]_{\mathbb{C}}$ denotes the gauge equivalence class
with respect to the complexified gauge group $\mathcal{G}_{\mathbb{C}}^{\det}(\check{E})$.
We need to understand the surjectivity of this map.

We will focus in the case that $C$ is irreducible and $\varUpsilon$
is nowhere vanishing, since the case of $(C,0)$ can be found in \cite{MR1863850},
while the case of a $U(1)$ vortex on an orbifold can be found in
\cite[Theorem 5]{MR1611061}.

We recall first the definition of stability for a pair $(\bar{\partial}_{C},\varUpsilon)$
\cite[Section 1.3]{MR2700715}, \cite[Chapter 1,  p.3]{MR2254074}.
Let $\check{\mathcal{E}}$ denote the bundle $\check{E}$ with the
holomorphic structure induced by $\bar{\partial}_{C}$. Denote by
$\mathcal{R}(\check{\mathcal{E}})$ the set of subsheaves of $\check{\mathcal{E}}$
with torsion free quotients.  For any sheaf $\check{F}$ define the
\textbf{slope $\mu_{\check{g}}(\check{F})$ }as 
\[
\mu_{\check{g}}(\check{F})=\frac{\int_{\check{\varSigma}}c_{1}(\check{F})}{\text{rk}(\check{F})}\in\mathbb{Q}
\]

\begin{defn}
The pair $(\check{\mathcal{E}},\varUpsilon)$ is \textbf{stable }if
$\mu_{\check{g}}(\check{\mathcal{F}})<\mu_{\check{g}}(\check{\mathcal{E}})$
whenever $\check{\mathcal{F}}\in\mathcal{R}(\check{\mathcal{E}})$
and $\varUpsilon\in H^{0}(\check{\mathcal{F}})$. In other words,
\[
c_{1}(\check{L})<\frac{1}{2}c_{1}(\check{E})
\]
whenever $\varUpsilon\in H^{0}(\check{L})$.
\end{defn}

The same arguments to those in \cite[Section 2.3]{MR2700715} show
that a pair $(\bar{\partial}_{C},\varUpsilon)$ which solves the\textbf{
projective vortex equation} 
\[
*F_{C}^{0}-i\left[\varUpsilon\varUpsilon^{*}-\frac{1}{2}|\varUpsilon|^{2}I_{E}\right]=0
\]
 must be stable, provided that $C$ is irreducible and $\varUpsilon$
non-vanishing. Moreover, if the pair $(\bar{\partial}_{C},\varUpsilon)$
is stable then it will solve the projective vortex equations (up to
gauge of course). Therefore, we have 
\begin{thm}
The irreducible solutions to the $SO(3)$ vortex equations $(C,\varUpsilon)$
{[}so $C$ is irreducible and $\varUpsilon$ non-vanishing{]}, are
in bijection with the stable pairs. In other words, the Kobayashi-Hitchin
map (\ref{eq: KH map}) is a bijection when restricted to the irreducible
$SO(3)$ vortices. 
\end{thm}

The reader may be more familiarized with a version of the stable pair
equations where a stability parameter explicitly appears \cite{MR1085139,MR1124279,MR1273268,MR1396775}.
This occurs if one uses the $U(2)$ vortex equations instead of the
$SO(3)$ vortex equations, so we now proceed to explain their relationship.
We can think of the $U(2)$ vortex equations as equations for a pair
$(C,\varUpsilon)$ which now satisfy 
\[
\begin{cases}
*F_{C}-i\varUpsilon\varUpsilon^{*}+i\tau\text{Id}_{E}=0\\
\bar{\partial}_{C}\varUpsilon=0
\end{cases}
\]
The main difference with equations (\ref{eq:SO(3) vortex equations})
is that we are no longer requiring that $C$ induce a fixed connection
$C^{\det}$ on $\det\check{E}$, which is why we now have an equation
for the entire curvature of $F_{C}$ and not just its trace-free part.
Here $\tau$ is a constant which is related to a stability parameter.
Clearly we can break the curvature equation into its traceless and
trace parts, so that the $U(2)$ equations read
\begin{equation}
\begin{cases}
2*F_{C^{\det}}-i|\varUpsilon|^{2}+2i\tau=0\\
*F_{C}^{0}-i\left[\varUpsilon\varUpsilon^{*}-\frac{1}{2}|\varUpsilon|^{2}I_{E}\right]=0\\
\bar{\partial}_{C}\varUpsilon=0
\end{cases}\label{eq:U(2) vortex equations}
\end{equation}
Notice that if we integrate the first equation, we obtain the relation
\[
\text{vol}(\check{\varSigma})\tau=\mu(\varUpsilon)+2\pi c_{1}(\det\check{E})
\]
Therefore, fixing a value of $\tau$ is equivalent to fixing a level
set for the moment map $\mu$. Clearly, any solution $(C,\varUpsilon)$
to the $U(2)$ vortex equations will yield a solution to the $SO(3)$
vortex equations, provided $C$ happened to induce our predetermined
connection $C^{\det}$ on $\det\check{E}$. On the other hand, if
we start with a solution $(C,\varUpsilon)$ to our $SO(3)$ vortex
equations, there is a priori no reason why the first equation in (\ref{eq:U(2) vortex equations})
has to be satisfied, so the moduli space of $SO(3)$ vortices is not
just the moduli space of $U(2)$ vortices which induce a particular
connection on $\det\check{E}$.

\section{\label{sec:A-Morse-Function}A Morse-Bott Function on the Moduli
Space}

As explained in Theorem (\ref{moment map theorem}), the Higgs strength
function can be interpreted as a moment map for the $S^{1}$ action
on the $SO(3)$ vortex moduli space. 

This means that whenever we choose the $U(2)$ bundle $\check{E}$
with the property that it guarantees the moduli space to be smooth
(at the reducibles), Frankel's theorem holds \cite{MR131883}, which
implies in our case the $\mu$ is a Morse-Bott function with the critical
set being equal to the fixed points of the $S^{1}$ action. In other
words, the critical sets of $\mu$ consist of the abelian vortices
or the moduli space of flat connections.

Moreover, the index of a particular critical set is the dimension
of the subspace on which the circle action acts with \emph{negative
weight.} As mentioned before, a version of Frankel's theorem for almost
hermitian manifolds was used recently by Feehan and Leness in order
to compute the Morse indices of the analogue of $\mu$ in the case
of the $SO(3)$ monopole moduli spaces (see Theorem 1 in \cite{Feehan-Leness[Virtual]}
and the discussion that follows). In the gauge theory context, probably
the first use of this idea appeared in Hitchin's paper \cite[Section 7]{MR887284}.
For the case of $U(2)$ vortices some calculations in this spirit
are done in \cite[Sections 3 and 4]{MR1396775}, and Thaddeus also
computes the dimensions of the subspaces with positive and negative
weights in \cite[Section 8]{MR1333296}.

First we understand the deformation theory at an abelian vortex. Recall
that the linearization of the $SO(3)$ vortex map was given in Lemma
\ref{Lem Linearization vortex map}
\[
\mathcal{D}\mathfrak{F}_{SO(3),(C,\varUpsilon)}(\dot{c},\dot{\varUpsilon})=\left(*d_{C}\dot{c}-i\left[\varUpsilon\dot{\varUpsilon}^{*}+\dot{\varUpsilon}\varUpsilon^{*}-\text{Re}\left\langle \varUpsilon,\dot{\varUpsilon}\right\rangle I_{E}\right],\bar{\partial}_{C}\dot{\varUpsilon}+\dot{c}''\otimes\varUpsilon\right)
\]
and likewise from equation (\ref{Coulomb condition})
\[
d_{(C,\varUpsilon)}^{0,*}(\dot{c},\dot{\varUpsilon})=-d_{C}^{*}\dot{c}+[\dot{\varUpsilon}\varUpsilon^{*}-\varUpsilon\dot{\varUpsilon}^{*}-i\text{Re}<i\varUpsilon,\dot{\varUpsilon}>I_{E}]
\]
At an abelian vortex we can write $C=C_{\check{L}}\oplus(C_{\check{L}}^{*}\otimes C^{\det})$
and $\varUpsilon=\alpha\oplus0$. Since $\check{E}=\check{L}\oplus(\check{L}^{*}\otimes\det\check{E})$,
we have $\mathfrak{g}_{\check{E}}=i\mathbb{R}\oplus(\check{L}^{2}\otimes(\det\check{E})^{-1})$.
If $\dot{\varUpsilon}=(\dot{\varUpsilon}_{t}\oplus\dot{\varUpsilon}_{n})$
and $\dot{c}=\dot{c}_{t}\oplus\dot{c}_{n}$ (here $t$ and $n$ stand
for tangential and normal respectively) then 
\[
\begin{cases}
\dot{\varUpsilon}\varUpsilon^{*}=\left(\begin{array}{c}
\dot{\varUpsilon}_{t}\\
\dot{\varUpsilon}_{n}
\end{array}\right)\left(\begin{array}{cc}
\alpha^{*} & 0\end{array}\right)=\left(\begin{array}{cc}
\dot{\varUpsilon}_{t}\alpha^{*} & 0\\
\dot{\varUpsilon}_{n}\alpha^{*} & 0
\end{array}\right)\\
\varUpsilon\dot{\varUpsilon}^{*}=\left(\begin{array}{c}
\alpha\\
0
\end{array}\right)\left(\begin{array}{cc}
\dot{\varUpsilon}_{t}^{*} & \dot{\varUpsilon}_{n}^{*}\end{array}\right)=\left(\begin{array}{cc}
\alpha\dot{\varUpsilon}_{t}^{*} & \alpha\dot{\varUpsilon}_{n}^{*}\\
0 & 0
\end{array}\right)\\
\varUpsilon\dot{\varUpsilon}^{*}+\dot{\varUpsilon}\varUpsilon^{*}-\text{Re}\left\langle \varUpsilon,\dot{\varUpsilon}\right\rangle I_{E}=\left(\begin{array}{cc}
\text{Re}\left\langle \alpha,\dot{\varUpsilon}_{t}\right\rangle  & \alpha\dot{\varUpsilon}_{n}^{*}\\
\dot{\varUpsilon}_{n}\alpha^{*} & -\text{Re}\left\langle \alpha,\dot{\varUpsilon}_{t}\right\rangle 
\end{array}\right)
\end{cases}
\]
We can interpret $\dot{c}$ as a matrix by writing $\dot{c}=\left(\begin{array}{cc}
\dot{c}_{t} & \dot{c}_{n}\\
-\dot{c}_{n}^{*} & -\dot{c}_{t}
\end{array}\right)$. Then if we write $\check{N}=\check{L}^{2}\otimes(\det\check{E})^{-1}$
so that $\dot{c}_{n}\in\varOmega^{1}(\check{\varSigma};\check{N})$,
and $C_{\check{N}}$ represents the induced connection $C_{\check{L}}^{*}\otimes C^{\det}$,
\[
*d_{C}\dot{c}-i\left[\varUpsilon\dot{\varUpsilon}^{*}+\dot{\varUpsilon}\varUpsilon^{*}-\text{Re}\left\langle \varUpsilon,\dot{\varUpsilon}\right\rangle I_{E}\right]=\left(\begin{array}{c}
*d\dot{c}_{t}-i\text{Re}\left\langle \alpha,\dot{\varUpsilon}_{t}\right\rangle \\
*d_{C_{\check{N}}}\dot{c}_{n}-i\alpha\dot{\varUpsilon}_{n}^{*}
\end{array}\right)
\]
Likewise, 
\[
\bar{\partial}_{C}\dot{\varUpsilon}+\dot{c}''\otimes\varUpsilon=\left(\begin{array}{c}
\bar{\partial}_{C_{\check{L}}}\dot{\varUpsilon}_{t}+\dot{c}_{t}''\otimes\alpha\\
\bar{\partial}_{C_{\check{L}}^{*}\otimes C^{\det}}\dot{\varUpsilon}_{n}-\dot{c}_{n}^{*\prime\prime}\alpha
\end{array}\right)
\]
It is easy to recognize the first row in each of these vectors as
corresponding to the linearization to the $U(1)$ vortex equations.
The remaining rows together with the Coulomb conditions leads us to
define the \textbf{normal operator} 
\[
\boxed{\mathscr{D}_{(C_{\check{L}},\alpha)}^{n}(\dot{c}_{n},\dot{\varUpsilon}_{n})=\left(\begin{array}{c}
-d_{C_{\check{N}}}^{*}\dot{c}_{n}-\alpha\dot{\varUpsilon}_{n}^{*}\\
*d_{C_{\check{N}}}\dot{c}_{n}-i\alpha\dot{\varUpsilon}_{n}^{*}\\
\bar{\partial}_{C_{\check{L}}^{*}\otimes C^{\det}}\dot{\varUpsilon}_{n}-\dot{c}_{n}^{*\prime\prime}\alpha
\end{array}\right)}
\]

Observe that the normal bundle to the abelian vortex moduli space
inside the $SO(3)$ vortex moduli space then corresponds to $\ker\mathscr{D}_{(C_{\check{L}},\alpha)}^{n}$.
For the case of the projectively flat connections, the story is a
bit different. Here the normal operator is easily seen to be 
\[
\boxed{\mathscr{D}_{(C,0)}^{n}(\dot{c},\dot{\varUpsilon})=\left(\begin{array}{c}
-d_{C}^{*}\dot{c}\\
\bar{\partial}_{C}\dot{\varUpsilon}
\end{array}\right)}
\]
A simple adaptation of (\ref{lem:Suppose-that-smoothness criterion})
shows:
\begin{lem}
Suppose that $c_{1}(\check{E})>2c_{1}(K_{\check{\varSigma}})$ and
that for the smooth case $\check{E}$ is of odd odd degree or when
all the $a_{1},\cdots,a_{n}$ are co-prime we are in the case where
$\det\check{E}$ is an odd power of $\check{L}_{0}$. Then the moduli
space $\mathcal{M}(\check{\varSigma},\check{E})$ is smooth at the
abelian vortices and the projectively flat connections.
\end{lem}

\begin{proof}
It is easier to explain why $\mathscr{D}_{(C,0)}^{n}$ is surjective.
Since $\ker d_{C}=\{0\}$ by assumption, we just need to worry about
the surjectivity of $\bar{\partial}_{C}$. Thus we need to guarantee
that $H^{1}(\check{E})$ vanishes. 

If $H^{1}(\check{E})\text{\ensuremath{\neq0}}$, then $H^{0}(K_{\check{\varSigma}}E^{*})\neq0$
by Serre duality, so following the argument in \cite[1.10]{MR1273268},
we conclude that there is an injection 
\[
0\rightarrow K_{\check{\varSigma}}^{-1}\otimes\det\check{E}\otimes L_{\check{D}}\rightarrow\check{E}
\]
 for some effective divisor $\check{D}$ (recall $L_{\check{D}}$
was the orbifold line bundle corresponding to $\check{D}$). 

By the famous theorem of Narasimhan and Seshadri \cite{MR575939,MR184252}
in the orbifold (or more generally parabolic) version, the bundle
$\check{E}$ equipped with the holomorphic structure provided by $\bar{\partial}_{C}$
is stable. Thus 
\[
c_{1}\left(K_{\check{\varSigma}}^{-1}\otimes\det\check{E}\otimes L_{\check{D}}\right)<\frac{1}{2}c_{1}(\check{E})
\]
At the same time, the left hand side is greater than or equal to $c_{1}(\check{E})-c_{1}(K_{\check{\varSigma}})$,
so we conclude that 
\[
\frac{1}{2}c_{1}(\check{E})<c_{1}(K_{\check{\varSigma}})
\]
and thus 
\[
c_{1}(\check{E})<2c_{1}(K_{\check{\varSigma}})
\]
The converse $c_{1}(\check{E})>2c_{1}(K_{\check{\varSigma}})$ implies
that $\bar{\partial}_{C}$ is surjective follows from the same argument
given by Thaddeus in \cite[1.10]{MR1273268}).

The argument given for the surjectivity of $\mathscr{D}_{(C_{L},\alpha)}^{n}$
is indistinguishable from the one we gave in Lemma (\ref{lem:Suppose-that-smoothness criterion})
. Alternatively, one can check that 
\[
H^{1}(\check{L}^{*}\otimes\det\check{E})\simeq H^{0}(K_{\check{\varSigma}}\otimes\check{L}\otimes(\det\check{E})^{*})
\]
and since $c_{1}(\check{L})<\frac{1}{2}c_{1}(\check{E})$ and we are
assuming $\frac{1}{2}c_{1}(\check{E})>c_{1}(K_{\check{\varSigma}})$
we have
\[
c_{1}(K_{\check{\varSigma}}\otimes\check{L}\otimes(\det\check{E})^{*})=c_{1}(K_{\check{\varSigma}})+c_{1}(\check{L})-c_{1}(\check{E})<c_{1}(K_{\check{\varSigma}})-\frac{1}{2}c_{1}(\check{E})<0
\]
and thus $H^{1}(\check{L}^{*}\otimes\det\check{E})$ vanishes, guaranteeing
the surjectivity of the $\bar{\partial}_{C_{\check{L}}^{*}\otimes C^{\det}}$
operator.
\end{proof}
From the stable pair perspective, the Zariski tangent space can be
identified with the hypercohomology 
\[
\mathbf{H}^{1}(\text{End}_{0}\check{E}\rightarrow^{\varUpsilon}\check{E})
\]
where $\text{End}_{0}\check{E}=\mathfrak{su}(2)\otimes\mathbb{C}$.
The reason for this is that the deformation theory at $(\bar{\partial}_{C},\varUpsilon)$
is obtained from the complex 
\[
0\rightarrow\varOmega^{0}(\check{\varSigma},\text{End}_{0}\check{E})\rightarrow^{d_{1}}\varOmega^{0,1}(\check{\varSigma},\text{End}_{0}\check{E})\oplus\varOmega^{0}(\check{\varSigma},\check{E})\rightarrow^{d_{2}}\varOmega^{0,1}(\check{\varSigma},\check{E})
\]
 where 
\[
\begin{cases}
d_{1}\xi=(-\bar{\partial}_{C}\xi,\xi\varUpsilon)\\
d_{2}(\dot{c}'',\dot{\varUpsilon})=(\bar{\partial}_{C}\dot{\varUpsilon}+\dot{c}''\varUpsilon)
\end{cases}
\]
When $\varUpsilon=\alpha\oplus0$, we know that the matrices \cite[Section 3.1]{MR1855754}
\[
G(\theta)=\left(\begin{array}{cc}
1 & 0\\
0 & e^{2i\theta}
\end{array}\right)
\]
 belong to the stabilizer of $(\check{E},\alpha)$. Thus $G(\theta)$
acts with trivial weight on the $\check{L}$ factor, and with positive
weight on the $\check{M}\equiv\check{L}^{*}\otimes\det\check{E}$
factor. Since 
\[
\mathfrak{su}(2)\simeq\mathbb{R}\oplus(\check{L}\otimes\check{M}^{*})
\]
After complexifying and writing an element $A\in\mathfrak{sl}(2,\mathbb{C})$
as 
\[
A=\left(\begin{array}{cc}
a_{11} & a_{12}\\
a_{21} & -a_{11}
\end{array}\right)
\]
with
\[
\begin{cases}
a_{11}\in\hom(\check{L},\check{L})\simeq\mathbb{C}\\
a_{12}\in\hom(\check{M},\check{L})\simeq\check{L}\otimes\check{M}^{*}\\
a_{21}\in\hom(\check{L},\check{M})\simeq\check{M}\otimes\check{L}^{*}
\end{cases}
\]
we get the decomposition 
\[
\text{End}_{0}\check{E}\simeq\underbrace{\mathbb{C}}_{\text{trivial weight}}\oplus(\underbrace{\check{L}\otimes\check{M}^{*}}_{\text{negative weight}})\oplus(\underbrace{\check{M}\otimes\check{L}^{*}}_{\text{positive weight}})
\]
Decompose the maps $d_{1},d_{2}$ as 
\[
\begin{cases}
d_{1}\xi=\left(\left(\begin{array}{cc}
-\bar{\partial}\xi_{11} & -\bar{\partial}_{\check{L}\otimes\check{M}^{*}}\xi_{12}\\
-\bar{\partial}_{\check{M}\otimes\check{L}^{*}}\xi_{21} & \bar{\partial}\xi_{11}
\end{array}\right),\left(\begin{array}{c}
\xi_{11}\alpha\\
\xi_{21}\alpha
\end{array}\right)\right)\\
d_{2}\left(\left(\begin{array}{cc}
a_{11} & a_{12}\\
a_{21} & -a_{11}
\end{array}\right),\left(\begin{array}{c}
\dot{\varUpsilon}_{\check{L}}\\
\dot{\varUpsilon}_{\check{M}}
\end{array}\right)\right)=\left(\begin{array}{c}
\bar{\partial}_{\check{L}}\dot{\varUpsilon}_{\check{L}}+a_{11}\alpha\\
\bar{\partial}_{\check{M}}\dot{\varUpsilon}_{M}+a_{21}\alpha
\end{array}\right)
\end{cases}
\]
Comparing the equations, we see that for a stable pair the \textbf{tangential
complex} is 
\begin{align*}
 & \varOmega^{0}(\check{\varSigma},\mathbb{C}) &  & \varOmega^{0,1}(\check{\varSigma},\mathbb{C})\oplus\varOmega^{0}(X,\check{L}) &  & \varOmega^{0,1}(\check{\varSigma},\check{L})\\
 & \xi_{11} & \rightarrow & (-\bar{\partial}\xi_{11},\xi_{11}\alpha)\\
 &  &  & (a_{11},\dot{\varUpsilon}_{\check{L}}) & \rightarrow & \bar{\partial}_{\check{L}}\dot{\varUpsilon}_{\check{L}}+a_{11}\alpha
\end{align*}
while the \textbf{normal complex }is the direct sum of the \textbf{negative
weight complex} 
\begin{equation}
\begin{array}{cccc}
\varOmega^{0}(\check{\varSigma},\check{L}\otimes\check{M}^{*}) &  & \varOmega^{0,1}(\check{\varSigma},\check{L}\otimes\check{M}^{*})\\
\xi_{12} & \rightarrow & -\bar{\partial}_{\check{L}\otimes\check{M}^{*}}\xi_{12}\\
 &  & a_{12}\rightarrow & 0
\end{array}\label{eq:negative weight complex}
\end{equation}
and the \textbf{positive weight complex }
\begin{align*}
 & \varOmega^{0}(\check{\varSigma},\check{M}\otimes\check{L}^{*}) &  & \varOmega^{0,1}(\check{\varSigma},\check{M}\otimes\check{L}^{*})\oplus\varOmega^{0}(X,\check{M}) &  & \varOmega^{0,1}(\check{\varSigma},\check{M})\\
 & \xi_{21} & \rightarrow & (-\bar{\partial}_{\check{M}\otimes\check{L}^{*}}\xi_{21},\xi_{21}\alpha)\\
 &  &  & (a_{21},\dot{\varUpsilon}_{\check{M}}) & \rightarrow & \bar{\partial}_{\check{M}}\dot{\varUpsilon}_{M}+a_{21}\alpha
\end{align*}
Since a slice for the gauge group action $\mathcal{G}_{\mathbb{C}}^{\det}(\check{E})$
can be taken to be $\ker\bar{\partial}_{C}^{*}$ (\cite[Section 4.1]{MR1288304},
\cite{MR1797591}), the normal bundle at $(\bar{\partial}_{C_{\check{L}}},\alpha)$
can be identified with $\ker\bar{\partial}_{C}^{*}$ together with
the kernel of the map $\dot{\varUpsilon}_{\check{M}}\rightarrow\bar{\partial}_{\check{M}}\dot{\varUpsilon}_{M}+a_{21}\alpha$. 

So to compare the normal bundles we complexify $\ker\mathscr{D}_{(C_{\check{L}},\alpha)}^{n}$.
This allows us to write $\dot{c}_{n}$ in terms of its $(1,0)$ and
$(0,1)$ types. For example, the second equation of $\mathscr{D}_{(C_{\check{L}},\alpha)}^{n}$
becomes
\begin{align*}
 & *d_{C_{\check{N}}}\dot{c}_{n}-i\alpha\dot{\varUpsilon}_{n}^{*}\\
= & -*d_{C_{\check{N}}}**\dot{c}_{n}-i\alpha\dot{\varUpsilon}_{n}^{*}\\
= & d_{C_{\check{N}}}^{*}*\dot{c}_{n}-i\alpha\dot{\varUpsilon}_{n}^{*}\\
= & d_{C_{\check{N}}}^{*}(-i\dot{c}_{n}'+i\dot{c}_{n}'')-i\alpha\dot{\varUpsilon}_{n}^{*}
\end{align*}
 while the first equation can be rewritten as 
\[
-d_{C_{\check{N}}}^{*}(\dot{c}_{n}'+\dot{c}_{n}'')-\alpha\dot{\varUpsilon}_{n}^{*}
\]
So the first two equations defining $\ker\mathscr{D}_{(C_{\check{L}},\alpha)}^{n}\otimes\mathbb{C}$
become equivalent to 
\[
\begin{cases}
d_{C_{\check{N}}}^{*}(\dot{c}_{n}'-\dot{c}_{n}'')+\alpha\dot{\varUpsilon}_{n}^{*}=0\\
d_{C_{\check{N}}}^{*}(\dot{c}_{n}'+\dot{c}_{n}'')+\alpha\dot{\varUpsilon}_{n}^{*}=0
\end{cases}
\]
which is the same as 
\[
\begin{cases}
d_{C_{\check{N}}}^{*}\dot{c}_{n}'+\alpha\dot{\varUpsilon}_{n}^{*}=0\\
d_{C_{\check{N}}}^{*}\dot{c}_{n}''=0\iff\dot{c}_{n}''\in\ker\bar{\partial}_{C_{\check{N}}}^{*}
\end{cases}
\]
Since $\ker\bar{\partial}_{C_{\check{N}}}^{*}$ was one of the equations
defining the normal bundle for the stable pair equation, this means
that the Zariski tangent space for the moduli space of $SO(3)$ vortices
can be regarded as a subspace of the Zariski tangent space to the
moduli space of stable pairs. Since both moduli spaces have the same
dimensions, there is an isomorphism between the Zariski tangent spaces,
so the moduli spaces are isomorphic as well (not just homeomorphic).

Once we know this, we can compute the dimension of the negative weight
complex (\ref{eq:negative weight complex}). A simple application
of the index formula shows that this is the same as 
\[
\dim_{\mathbb{R}}\mathbf{H}^{1}(\check{L}\otimes\check{M}^{*}\rightarrow0)=\dim_{\mathbb{R}}\mathbf{H}^{1}(\check{\varSigma},\check{L}\otimes\check{M}^{*})=-2\chi_{\mathbb{C}}(\check{\varSigma},\check{L}^{2}\otimes(\det\check{E})^{-1})
\]
where in the last step we used that $\mathbf{H}^{0}(\check{\varSigma},\check{L}\otimes\check{M}^{*})$
vanishes. This happens since
\begin{align*}
 & \deg_{B}(\check{L}\otimes\check{M}^{*})\\
\leq & c_{1}(\check{L}\otimes\check{M}^{*})\\
= & 2\deg\check{L}-\deg\check{E}\\
< & 0
\end{align*}
 where we used Lemma (\ref{Lem degree constraint}) (strictly speaking
we need to assume that $\check{E}$ was chosen so that there are no
abelian vortices of degree $\frac{1}{2}\deg\check{E}$).

To apply Riemann-Roch (\ref{Riemann Roch Line Bundles}) we need to
know what the isotropy of $\check{L}^{2}\otimes(\det\check{E})^{-1}$
is. Recall that if the $U(2)$ bundle $\check{E}$ has isotropy $\left(\begin{array}{cc}
\sigma_{i}^{b_{i}^{-}}\\
 & \sigma_{i}^{b_{i}^{+}}
\end{array}\right)$ , then the isotropy of the determinant line bundle is $\sigma_{i}^{b_{i}^{-}+b_{i}^{+}}$,
while the isotropy of $\check{L}^{\otimes2}$ will be $\sigma_{i}^{2b_{i}}$
where $b_{i}=\frac{\epsilon_{i}(b_{i}^{+}-b_{i}^{-})+(b_{i}^{+}+b_{i}^{-})}{2}$
, $\epsilon_{i}=1$ if $b_{i}=b_{i}^{+}$ and $\epsilon_{i}=-1$ if
$b_{i}=b_{i}^{-}$. Therefore the isotropy of $\check{L}^{2}\otimes(\det\check{E})^{-1}$
is 
\[
\epsilon_{i}(b_{i}^{+}-b_{i}^{-})+(b_{i}^{+}+b_{i}^{-})-(b_{i}^{-}+b_{i}^{+})=\epsilon_{i}(b_{i}^{+}-b_{i}^{-})
\]
if $\epsilon_{i}=1$ then this gives $b_{i}^{+}-b_{i}^{-}$ and this
is already between $0$ and $a_{i}$ so there is nothing to worry
about. 

If $\epsilon_{i}=-1$ then this gives $b_{i}^{-}-b_{i}^{+}$ which
is non-positive (smaller or equal to $0$ and definitely bigger than
$-a_{i}$), so we should in fact have to consider a shift by $a_{i}$,
i.e, the isotropy is 
\[
\begin{cases}
b_{i}^{+}-b_{i}^{-} & \epsilon_{i}=1\\
a_{i}+b_{i}^{-}-b_{i}^{+} & \epsilon_{i}=-1
\end{cases}
\]
Therefore, (minus) the Euler characteristic according to Rieman-Roch
is 
\[
g-1+c_{1}(\det\check{E})-2c_{1}(\check{L})+\sum_{i\mid\epsilon_{i}=1}\frac{b_{i}^{+}-b_{i}^{-}}{a_{i}}+n_{-}+\sum_{i\mid\epsilon_{i}=-1}\frac{b_{i}^{-}-b_{i}^{+}}{a_{i}}
\]
and thus the index is twice this amount, i.e, we found
\begin{thm}
\label{thm:computation indices}Suppose that a bundle $\check{E}$
is chosen on $\check{\varSigma}$ so $c_{1}(\check{E})>2c_{1}(K_{\check{\varSigma}})$
and that for the smooth case we are in the case $\check{E}$ with
odd degree or when all the $a_{1},\cdots,a_{n}$ are co-prime we are
in the case where $\det\check{E}$ is an odd power of $\check{L}_{0}$.

Consider the moduli space of abelian vortices associated to the reduction
$\check{E}=\check{L}\oplus\check{L}^{*}\otimes\det\check{E}$. Then
the Morse-Bott index of $\mu$ at this submanifold of abelian vortices
is given by the formula
\begin{equation}
\text{index}(\check{E},\check{L})=2\left[g-1+c_{1}(\det\check{E})-2c_{1}(\check{L})+\sum_{i\mid\epsilon_{i}=1}\frac{b_{i}^{+}-b_{i}^{-}}{a_{i}}+n_{-}+\sum_{i\mid\epsilon_{i}=-1}\frac{b_{i}^{-}-b_{i}^{+}}{a_{i}}\right]\label{index}
\end{equation}
where $\check{L}$ has isotropy $b_{i}$, and $\epsilon_{i}=1$ if
$b_{i}=b_{i}^{+}$, $\epsilon_{i}=-1$ if $b_{i}=b_{i}^{-}$. Here
$n_{-}=\#\{i\mid b_{i}=b_{i}^{-}<b_{i}^{+}\}$.
\end{thm}

\begin{rem}
From the previous formula we can observe the following:

In the smooth case the formula for the index reduces to 
\[
2(g-1+\underbrace{\deg E-2\deg L}_{\geq0})
\]
 Therefore the index of the abelian vortices will always be non-negative
whenever $g\geq1$. If we choose $\deg E$ odd, then the index is
strictly positive whenever $g\geq1$ so this means that the minima
of $\mu$ is not achieved at the abelian vortices. 

Moreover, observe that when $\deg L=0$, then the index gives $2g-2+2\deg E$,
which is the dimension formula for the moduli space of stable pairs
we found in (\ref{dimension moduli space}). This vortex moduli space
thus corresponds to the critical set which is the absolute maximum
for $\mu$, and since the vortex moduli space can be identified with
$\text{Sym}^{2\deg L}\varSigma$, this means that the maximum occurs
at a single abelian vortex (it cannot be more than one since the level
sets of $\mu$ are connected \cite[Theorem IV.3.1]{MR2091310}).

In the orbifold case when the $a_{i}$ are coprime, if $\det\check{E}=\check{L}_{0}^{k}$
and $\check{L}=\check{L}_{0}^{l}$ then the previous quantity is the
same as 
\[
2\left[g-1+\underbrace{\frac{k-2l}{a_{1}\cdots a_{n}}}_{>0}+\underbrace{\sum_{i\mid\epsilon_{i}=1}\frac{b_{i}^{+}-b_{i}^{-}}{a_{i}}}_{>0}+\underbrace{n_{-}+\sum_{i\mid\epsilon_{i}=-1}\frac{b_{i}^{-}-b_{i}^{+}}{a_{i}}}_{>0}\right]
\]
So the only chance this index has of being non-positive is if we take
$g=0$. This is consistent with the fact that on $S^{2}$ with marked
points the moduli space of parabolic stable bundles can be empty depending
on the choice of weights \cite{MR184252,MR575939}.
\end{rem}

\begin{example}
\label{example T2}To illustrate the index calculation, consider on
$T^{2}$ the $U(2)$ bundle $E$ with $\deg E=1$. The expected dimension
of the $SO(3)$ monopole moduli space is given in equation (\ref{dimension moduli space})
and happens to be $2$ in this case. Since $\deg E>4g-4$ the moduli
space of stable pairs is smooth \cite{MR1273268}.

The abelian moduli spaces can only appear when $\deg L\leq\frac{1}{2}\deg E=\frac{1}{2}$,
so $\deg L=0$ and hence there is a unique solution up to gauge (because
the abelian vortex can be interpreted as a holomorphic function and
must thus be constant). The index of this abelian monopole is $2$,
so it corresponds to the absolute maximum. Since the level sets of
$\mu$ are connected \cite[Theorem IV.3.1]{MR2091310}, there is only
a minimum, which must be a flat connection. After we take the quotient
by the $S^{1}$ action, this gives an interpretation of the moduli
space of stable pairs as a cobordism between the abelian vortex and
the flat connections.
\end{example}

\section{$SO(3)$ Monopoles on 3-Manifolds}

We now review the $SO(3)$ monopole equations on 3-manifolds, which
can be regarded as the dimensional reduction of the $SO(3)$ monopole
equations on 4-manifolds introduced by Pidstrigach and Tyurin and
studied extensively by Feehan and Leness \cite{MR1664908,MR1846123,MR1855754,MR1855754II,MR3349302,MR3897982}.
In particular, we will follow mostly their conventions (see \cite[Section 2.1.2]{MR1855754}).

In the same way that the Seiberg-Witten equations require a choice
of spin-c structure $\mathfrak{s}$ in order to be defined on a 3
or 4-manifold, the $SO(3)$ monopole equations require the choice
of a \emph{spin-u }structure $\mathfrak{t}$. The most expedient way
to define them on 3-manifolds $Y$ is as follows. 

Consider a spin-c structure $\mathfrak{s}=(S,\rho)$ which is represented
by a spinor bundle $S$ with corresponding Clifford multiplication
map $\rho$. Let $E$ denote a $U(2)$ bundle on $Y$. Then $(S,\rho)$
and $E$ determine the spin-u structure 
\[
\mathfrak{t}=(V,\rho_{V})=(S\otimes E,\rho\otimes1_{E})
\]
By abuse of notation we will write $\rho$ instead of $\rho_{V}$.
The above construction shows that spin-u structures always exist on
a 3-manifold, so the next question is the topological classification
of these. Notice first of all that if $L$ is an arbitrary complex
line bundle then 
\[
S\otimes E\simeq(S\otimes L)\otimes(E\otimes L^{-1})
\]
and since $\mathfrak{s}_{L}=(S\otimes L,\rho\otimes1_{L})$ represents
another spin-c structure, the decomposition of $V$ as $S\otimes E$
is not intrinsic. Since
\[
\begin{cases}
c_{1}(\mathfrak{s})+c_{1}(E)=c_{1}(\mathfrak{s}_{L})+c_{1}(E\otimes L^{-1})\\
w_{2}(\mathfrak{su}(E))=w_{2}(\mathfrak{su}(E\otimes L^{-1}))
\end{cases}
\]
 the three-dimensional version of the discussion in \cite[Section 2.1.3]{MR1855754}
shows that the spin-u structure $\mathfrak{t}$ determines the characteristic
classes
\[
\begin{cases}
c_{1}(\mathfrak{t})\equiv & c_{1}(\mathfrak{s})+c_{1}(E)\\
w_{2}(\mathfrak{t})\equiv & w_{2}(\mathfrak{su}(E))
\end{cases}
\]
On 3-manifolds $w_{2}(\mathfrak{su}(E))$ can always be lifted to
an integral cohomology class because of the Bockstein long exact sequence
and the fact that $H^{3}(Y;\mathbb{Z})\simeq\mathbb{Z}$ has no two-torsion,
so any pair $(c,w)\in H^{2}(Y;\mathbb{Z})\times H^{2}(Y;\mathbb{Z}_{2})$
determines a spin-u structure. 

If we think of the $SO(3)$ monopole equations as a system of equations
used to relate the monopole \cite{MR2388043} and instanton Floer
homologies on a specific 3-manifold \cite{MR1883043}, we should regard
$w_{2}(\mathfrak{t})$ as being the characteristic class provided
the $SO(3)$ bundle used to define the instanton Floer homology of
that 3-manifold. 

Since every $SO(3)$ bundle can be lifted to a $U(2)$ bundle on a
3-manifold, there is no harm in choosing a reference spin-c structure
to be a torsion spin-c structure, i.e, one for which $c_{1}(\mathfrak{s})=0$,
so that we can think of the spin-u structure $\mathfrak{t}$ as being
completely specified by a choice of a $U(2)$ bundle and nothing else.

Choosing the reference spin-c structure $\mathfrak{s}$ to be torsion
also has the added advantage that there is a \emph{canonical }lift
of the Levi-Civita connection on $TY$ to a spin-c connection on $S$
\cite[Section 5.1]{MR3827053}. With this reference spin-c connection,
we can define the configuration space where the $SO(3)$ monopole
equations are defined as 
\[
\mathcal{C}(Y,\mathfrak{t})=\mathcal{A}_{\det}(E)\times\varGamma(S\otimes E)
\]
where the notation $\mathcal{A}_{\det}(E)$ indicates that we are
considering the space of unitary connections $B$ on $E$ which induce
a fixed connection $B^{\det}$ on $\det E$. The \emph{unperturbed}
$SO(3)$\emph{ monopole equations }are then equations for a pair $(B,\varPsi)$
which satisfy the equations:
\begin{equation}
\mathit{SO(3)\;monopole\;equations\;}\begin{cases}
*F_{B}^{0}+2\rho^{-1}(\varPsi\varPsi^{*})_{00}=0\\
D_{B}\varPsi=0
\end{cases}\label{eq: SO(3) monopole equations on Y}
\end{equation}
Here $(\varPhi\varPhi^{*})_{00}$ represents the component of $\varPhi\varPhi^{*}\in\varGamma(S\otimes E\otimes S^{*}\otimes E^{*})\simeq\varGamma(\text{End}S\otimes\text{End}E)$
which is traceless on each of the factors of $\text{End}(S)\otimes\text{End}(E)$,
that is, $(\varPhi\varPhi^{*})_{00}$ is a section of $\text{End}_{0}(S)\otimes\text{End}_{0}(E)$.
We are using $\rho:TY\rightarrow\hom(S,S)$ to identify $TY$ isometrically
with the sub-bundle $\mathfrak{su}(S)$ of traceless, skew-adjoint
endomorphisms equipped with the inner product $\frac{1}{2}\text{tr}(a^{*}b)$.
As usual, $*$ denotes the Hodge star operator on $Y$. Finally, $D_{B}$
represents the twisted Dirac operator acting on sections of $S\otimes E$,
which strictly speaking should be written as $D_{B_{S}\otimes B}$,
where $B_{S}$ denotes the canonical spin-c connection we are using
on $S$, coming from the fact that we are using a torsion spin-c structure
for describing $\mathfrak{t}$.

One can first consider the solutions to the $SO(3)$ monopole equations
on $Y$ modulo the gauge group $\mathcal{G}_{\det}(E)$. We will denote
such equivalence classes as pairs $[B,\varPsi]$. In the space 
\[
\mathcal{B}(Y,\mathfrak{t})=\mathcal{C}(Y,\mathfrak{t})/\mathcal{G}_{\det}(E)
\]
 there is still a residual $S^{1}$ action given by 
\[
e^{i\theta}\cdot[B,\varPsi]\rightarrow[B,e^{i\theta}\varPsi]
\]
 
\begin{lem}
The fixed points of this circle action which solve (\ref{eq: SO(3) monopole equations on Y})
are of two types:

a) The pairs $[B,0]$ with vanishing spinor, which correspond to the
projectively flat connections.

b) The pairs of the form $[B_{L}\oplus(B^{\det}\otimes B_{L}^{*}),\varPsi=\psi\oplus0]$
corresponding to a reduction of the bundle $E$ into two line bundles
$E=L\oplus(\det E\otimes L^{*})$, which solve the perturbed\emph{
}Seiberg-Witten equations 
\begin{equation}
\mathit{perturbed\;Seiberg\;Witten\;equations}\begin{cases}
*F_{B_{L}}+\rho^{-1}(\psi\psi^{*})_{0}-\frac{1}{2}*F_{B^{\det}}=0\\
D_{B}\psi=0
\end{cases}\label{perturbed monopole equations}
\end{equation}
associated to the spin-c structure $S\otimes L$.
\end{lem}

\begin{proof}
This is basically a restatement of \cite[Lemma 3.12]{MR1855754},
with some simplifications in the case of a 3-manifold, given the choice
of our reference spin-c structure. The only thing we really need to
check is our choice of conventions for the constants. We have chosen
them so that they agree with \cite[Eq 4.4]{MR2388043}. At a reducible
connection $B=B_{L}\oplus(B^{\det}\otimes B_{L}^{*})$, the traceless
part of the curvature $F_{B}^{0}$ can be written as 
\[
\left(\begin{array}{cc}
F_{B_{L}}-\frac{1}{2}F_{B^{\det}} & 0\\
0 & \frac{1}{2}F_{B^{\det}}-F_{B_{L}}
\end{array}\right)
\]
so the curvature equation becomes 
\[
*F_{B_{L}}-\frac{1}{2}*F_{B^{\det}}+\rho^{-1}(\psi\psi^{*})_{0}=0
\]
We used the fact that for $\varPsi=\psi\otimes0$, we have \cite[Eq. 3.13]{MR1855754}
\[
(\varPsi\varPsi^{*})_{00}=(\psi\psi^{*})_{0}\otimes\left(\begin{array}{cc}
\frac{1}{2} & 0\\
0 & -\frac{1}{2}
\end{array}\right)
\]

Notice that $\det(\mathfrak{s}_{L})=\det(S\otimes L)\simeq L^{\otimes2}$,
so in the notation of Kronheimer and Mrowka, $B^{t}\equiv B_{L}^{\otimes2}$
is the induced connection on $\det(\mathfrak{s}_{L})$. In particular,
$F_{B^{t}}=2F_{B_{L}}$ and so the previous equation can be rewritten
as 
\[
\frac{1}{2}*F_{B^{t}}+\rho^{-1}(\psi\psi^{*})_{0}=\frac{1}{2}*F_{B^{\det}}
\]
This agrees with the conventions used in \cite[Eq 4.11]{MR2388043}.
\end{proof}
It is very important to notice that when the Seiberg-Witten monopoles
appear as fixed points of the $S^{1}$ action, the equations they
solve are ``automatically'' perturbed by the term $F_{B^{\det}}$.
These perturbations can have dramatic consequences, the first one
being the following:
\begin{lem}
Suppose that $(B_{L},\psi)$ solves the perturbed Seiberg-Witten equations
(\ref{perturbed monopole equations}), where $B_{L}\in\mathcal{A}(L)$
is a $U(1)$ connection, and $\psi\in\varGamma(S\otimes L)$. Then
$\psi$ vanishes identically if and only if $c_{1}(E)\in H^{2}(Y;\mathbb{Z})$
is an even class, i.e, $w_{2}(\mathfrak{su}(E))$ vanishes. In other
words, if $c_{1}(E)$ is an odd class, then there are no ``reducible''
solutions to the equations (\ref{perturbed monopole equations}).
\end{lem}

\begin{proof}
Observe that a solution of the form $(B_{L},0)$ must solve the equation
\[
F_{B^{\det}}=2F_{B_{L}}
\]
and by Chern-Weil theory we conclude that 
\[
c_{1}(E)=2c_{1}(L)
\]
from which the assertion in the lemma follows immediately.
\end{proof}
\begin{rem}
Another consequences of these perturbations is that in general the
symmetry between the solutions to the Seiberg-Witten equations for
$\mathfrak{s}$ and its conjugate spin-c structure $\bar{\mathfrak{s}}$
is lost. This is analogous to a comment Donaldson makes in \cite[Section 6]{MR1734402}.
\end{rem}

In fact, if define 
\[
\omega\equiv\frac{1}{4}F_{B^{\det}}
\]
 then $\omega$ is a harmonic representative of $-\frac{\pi i}{2}c_{1}(E)\in H^{2}(Y;i\mathbb{R})$
and using $F_{B^{t}}=2F_{B_{L}}$, the equations \ref{perturbed monopole equations}
can be rewritten as \cite[Eq 29.2]{MR2388043} 
\begin{equation}
\begin{cases}
\frac{1}{2}\rho(F_{B^{t}}-4\omega)-(\psi\psi^{*})_{0}=0\\
D_{B}\psi=0
\end{cases}\label{non exact perturbation}
\end{equation}
In particular, if $c_{1}(E)$ is an odd class, then the term $\omega$
corresponds to a \emph{non-balanced perturbation}, in the terminology
of Kronheimer and Mrowka \cite[Definition 29.1.1]{MR2388043}. A consequence
of working with a non-balanced perturbation is that the definition
of the monopole Floer homology groups $HM(Y,\mathfrak{s},\omega)$
require the choice of a local coefficient system in order to make
sense of the differential, since in general there are monotonicity
issues when defining these groups. Some choices of local coefficient
systems are described in \cite[Section 30.2]{MR2388043}.

It remains a daunting challenge to try to define a suitable version
of $SO(3)$ \emph{monopole Floer homology, }because of the presence
of reducibles (i.e, abelian $U(1)$ monopoles and flat connections). 

As a historical analogy, before the construction of monopole Floer
homology in full generality by Kronheimer and Mrowka \cite{MR2388043},
the simplest case where it could be defined was for three manifolds
with $b_{1}(Y)>0$ and for non-torsion spin-c structures, since in
this case reducibles could be avoided by using suitable perturbations.
The difficult case was that of torsion spin-c structures, which was
solved thanks to the blow-up model introduced by Kronheimer and Mrowka.

Likewise, we can think of defining $SO(3)$ monopole Floer homology
as trying to define monopole Floer homology in the torsion case, which
begs the question of whether there is an analogue of the non-torsion
case where reducibles can be avoided.

Such an analogue may be provided by a construction of an $U(2)$ \emph{monopole
Floer homology. }We explain the basic ideas behind this construction:
assuming all the technical hurdles can be solved, it can be regarded
as giving rise to a family of Floer homologies, indexed by subintervals
of the real line. So in a sense this would be the Floer-theoretic
analogue of the typical wall-crossing phenomenon for the various moduli
spaces which appear in Algebraic Geometry. 

In this case we no longer fix a reference connection on $\det E$,
so we are working with the enlarged configuration space 
\[
\mathcal{C}_{U(2)}(Y,\mathfrak{t})=\mathcal{A}(E)\times\varGamma(S\otimes E)
\]
Suppose that $\omega$ is a harmonic representative of $c_{1}(E)$
{[}before it was chosen to be a representative of $-\frac{\pi i}{2}c_{1}(E)${]}. 

The $U(2)$ \textbf{monopole equations }$SW_{U(2)}(Y,\mathfrak{t},\tau)$
are equations for a pair $(B,\varPsi)$ which satisfy 
\begin{equation}
U(2)\text{ monopole equations }SW_{U(2)}(Y,\mathfrak{t},\tau)\;\;\;\begin{cases}
*F_{B}+2\rho^{-1}(\varPsi\varPsi^{*})_{0}+i\tau(*\omega)I_{E}=0\\
D_{B}\varPsi=0
\end{cases}\label{eq:U(2) monopole equations}
\end{equation}

The main differences with the $SO(3)$ monopole equations (\ref{eq: SO(3) monopole equations on Y})
are the following:
\begin{enumerate}
\item As mentioned before, we are no longer fixing a reference connection
on $\det E$.
\item Since we dropped the reference connection, we now need an equation
which involves the entire curvature $F_{B}$ of $B$, instead of just
its traceless part. Subsequently, $(\varPsi\varPsi^{*})_{0}$ denotes
the component of $\varPsi\varPsi^{*}$ which is traceless only on
the ``spinorial'' component, that is, $(\varPsi\varPsi^{*})_{0}\in\varGamma(\text{End}_{0}S\otimes\text{End}E)$. 
\item Notice that the traceless part of the curvature equation in (\ref{eq:U(2) monopole equations})
gives us back the $SO(3)$ monopole equations (\ref{eq: SO(3) monopole equations on Y}).
\end{enumerate}
It is not difficult to identify the ``critical values'' of $\tau$,
that is, those values where the flat connections and Seiberg-Witten
solutions can appear.
\begin{lem}
\label{reductions U(2) monopole equations}a) Suppose that $(B,\varPsi\equiv0)$
solves the $U(2)$ monopole equations (\ref{eq:U(2) monopole equations}).
Then $\tau=\pi$ and $B$ is a flat $U(2)$ connection. If we assume
$c_{1}(E)$ is odd, then $B$ is irreducible, i.e, not compatible
with a decomposition $E=L_{1}\oplus L_{2}$, $B=B_{1}\oplus B_{2}$.

b) Suppose that $(B=B_{1}\oplus B_{2},\varPsi=\psi\oplus0)$ solves
the equation (\ref{eq:U(2) monopole equations}), where $E=L_{1}\oplus L_{2}$
is compatible with the reduction of $B$. Then $\tau=2\pi n$ for
$n$ an arbitrary integer, and $(B_{1},\psi)$ solves the perturbed
Seiberg-Witten equations (\ref{perturbed monopole equations}) for
the spin-c structure $S\otimes L_{1}$ with perturbation term $F_{B_{2}}$.
Moreover, $c_{1}(L_{1})$ and $c_{1}(L_{2})$ are both integer multiples
of $c_{1}(E)$, i.e, their Chern classes belong to the ray spanned
by $c_{1}(E)$ inside the integer lattice $H^{2}(Y;\mathbb{Z})$.
\end{lem}

\begin{proof}
a) This is a Chern-Weil argument once again. Taking traces we find
that $\text{tr}(F_{B})=-2i\tau\omega$ and therefore $\tau=\pi$.
If $B$ were compatible with a reduction then 
\[
F_{B}=\left(\begin{array}{cc}
F_{B_{1}} & 0\\
0 & F_{B_{2}}
\end{array}\right)=-i\tau\omega I_{E}
\]
 we would conclude that $F_{B_{1}}=F_{B_{2}}$. Therefore $L_{1}\simeq L_{2}$
and $c_{1}(E)$ must represent an even class, which cannot be true
by assumption.

b) In this case the equations become 
\[
\begin{cases}
*F_{B_{1}}+2\rho^{-1}(\psi\psi^{*})_{0}+i\tau(*\omega)=0\\
F_{B_{2}}=-i\tau\omega\\
D_{B_{1}}\psi=0
\end{cases}
\]
Substituting the second equation into the first one we recognize the
perturbed Seiberg Witten equation \ref{perturbed monopole equations}
\[
*F_{B_{1}}+2\rho^{-1}(\psi\psi^{*})_{0}-*F_{B_{2}}=0
\]
By Chern-Weil theory the second equation says that $-2\pi ic_{1}(L_{2})=-i\tau c_{1}(E)$,
so $\tau$ is of the form $\tau=2\pi n$, for $n\in\mathbb{Z}$ an
integer. Since 
\[
c_{1}(E)=c_{1}(L_{1}\oplus L_{2})=c_{1}(L_{1})+c_{1}(L_{2})=c_{1}(L_{1})+nc_{1}(E)
\]
 this forces $c_{1}(L_{1})$ to be
\[
c_{1}(L_{1})=(1-n)c_{1}(E)
\]
\end{proof}
\begin{rem}
a) Notice the second part of the previous lemma implies that the only
monopole Floer groups that the $U(2)$ monopole Floer equations can
``see'' are those whose spin-c structure satisfy $c_{1}(\mathfrak{s})=c_{1}(S\otimes L_{1})=2(1-n)c_{1}(E)$,
for $n$ an arbitrary integer. In particular, notice that all spin-c
structures which are torsion can appear in this picture. Also, in
the specific case that $Y$ is an integer homology $S^{1}\times S^{2}$,
and we choose $c_{1}(E)$ as a generator of $H^{2}(Y;\mathbb{Z})\simeq\mathbb{Z}$,
then all spin-c structures can arise in this case as well. 

b) Implicitly we are also saying that solutions to equation \ref{eq:U(2) monopole equations}
of the form $(B=B_{1}\oplus B_{2},\varPhi=\varPsi_{1}\oplus\varPsi_{2})$,
$E=L_{1}\oplus L_{2}$ , do not arise unless one of the spinors $\varPsi_{1}$
or $\varPsi_{2}$ vanish identically. An easy way to see that otherwise
by taking the traceless part of the curvature equation, we would obtain
a pair $(B=B_{1}\oplus B_{2},\varPhi=\varPsi_{1}\oplus\varPsi_{2})$
which satisfies the $SO(3)$ monopole equations \ref{eq: SO(3) monopole equations on Y}
where none of the spinors is identically zero. However, \cite[Lemma 5.22]{MR1664908}
shows that this cannot be the case. A quick look at the statement
of Feehan and Leness shows that they stated this already for a \emph{perturbed
}version of the $SO(3)$ monopole equations, which had to satisfy
specific properties so that this would not occur. Therefore, this
already shows one needs to be very careful about which properties
will continue to hold after adding (holonomy) perturbations.
\end{rem}

Having introduced the $U(2)$ monopole equations, it makes sense to
state the following conjecture:
\begin{conjecture}
Suppose that $Y$ is a closed oriented 3-manifold, and $E$ is an
$U(2)$ bundle which satisfies the admissibility condition that $w=c_{1}(E)\in H^{2}(Y;\mathbb{Z})$
is a non-torsion odd class. Choose a spin-c structure $\mathfrak{s}$
and consider the bundle $V=S\otimes E$ representing the spin-u structure
$\mathfrak{t}$, where $(S,\rho)$ is a spinor bundle representing
$\mathfrak{s}$.

Then for $\tau\neq\pi,2\pi n$ for $n\in\mathbb{Z}$, it is possible
to define a\textbf{ Seiberg-Witten $U(2)$ monopole Floer homology}
$HM_{U(2)}(Y,\mathfrak{t},\tau)$. Some expected features of these
groups are the following:

a) The chain complex of $HM_{U(2)}(Y,\mathfrak{t},\tau)$ will be
built out of a suitably perturbed version of the equations \ref{eq:U(2) monopole equations}.
These solutions can be identified with the critical points of a perturbed
$U(2)$ Chern-Simons-Dirac functional $CS_{\mathfrak{t},\tau}$, morally
a Morse function on the space $\mathcal{B}(Y,\mathfrak{t})=\left(\mathcal{A}(E)\times\varGamma(V)\right)/\mathcal{G}(E)$. 

b) The differential involves the count of gradient flow lines mod
$\mathbb{R}$- translations of $CS_{\mathfrak{t},\tau}$ as in every
other Floer homology. In general, there will be monotonicity issues
{[}a bound on the dimension of flow lines will not imply a uniform
energy bound on the corresponding moduli spaces{]}, so the groups
$HM_{U(2)}(Y,\mathfrak{t},\tau)$ will require a Novikov ring (local
coefficient system) so that the differential is well defined. 

c) The groups will be $\mathbb{Z}/2\mathbb{Z}$ graded in general. 

d) Write $\mathbb{R}\backslash\{\pi,2\pi n\mid n\in\mathbb{Z}\}=\bigcup_{i}I_{i}$
as a disjoint union of open intervals. If $\tau,\tau'$ belong to
the same ``chamber'' $I_{i}$, then $HM_{U(2)}(Y,\mathfrak{t},\tau)\simeq HM_{U(2)}(Y,\mathfrak{t},\tau')$,
i.e, the groups are isomorphic to each other as long as one does not
cross a ``wall''.
\end{conjecture}

Clearly behind this conjecture there is a great deal of work that
will require verification, but it could serve as a toy model before
a general construction of $SO(3)$ monopole Floer homology, which
would prove even more difficult in general. 

\section{$SO(3)$ Monopoles on $S^{1}\times\varSigma$: Framed Monopole Homology\label{sec:-Framed Monopole Homology}}

Now we want to understand the solutions to the $SO(3)$ monopole equations
on $Y=S^{1}\times\varSigma$ in terms of the solutions to the $SO(3)$
vortex equations on $\varSigma$. 

Before doing this, it is useful to understand how the spin-c structures
on $Y$ are related to the spin-c structures on $\varSigma$. More
details can be found in \cite[Section 3]{MR2141962}. Throughout our
discussion we will use a product metric on $S^{1}\times\varSigma$.

Suppose $\mathfrak{s}_{Y}$ is a spin-c structure on $Y$. Since $H^{2}(S^{1}\times\varSigma;\mathbb{Z})$
has no 2-torsion, we can specify $\mathfrak{s}_{Y}$ by its element
$c_{1}(\mathfrak{s}_{Y})\in H^{2}(S^{1}\times\varSigma;\mathbb{Z})$.
Let $S_{Y}$ denote a spinor bundle representing $\mathfrak{s}_{Y}$
and $L_{Y}=\det S_{Y}$ the corresponding line bundle.

If we think of $S^{1}\times\varSigma$ as $[0,1]\times\varSigma$
with the boundaries identified, then $L_{Y}$ is constructed as the
pullback under the projection $[0,1]\times\varSigma\rightarrow\varSigma$
of a line bundle $L_{\varSigma}$ by gluing along the boundaries with
an isomorphism $u\in\text{Map}(\varSigma,S^{1})$. Then 
\[
c_{1}(\mathfrak{s}_{Y})=c_{1}(L_{Y})=c_{1}(L_{\varSigma})+[S^{1}]\cup[u]
\]
where $[u]$ is the class of $u$ in $[\varSigma;S^{1}]\simeq H^{1}(\varSigma;\mathbb{Z})$.
The spin-c structure $\mathfrak{s}_{Y}$ induces a spin-c structure
$\mathfrak{s}_{\varSigma}$ with determinant line bundle $L_{\varSigma}$.
If $S_{\varSigma}$ is a spinor bundle for $\mathfrak{s}_{\varSigma}$
, we can assume that it is of the form 
\[
S_{\varSigma}=(\mathbb{C}\oplus K_{\varSigma}^{-1})\otimes(L_{\varSigma}\otimes K_{\varSigma})^{1/2}
\]
where $K_{\varSigma}=\varOmega^{1,0}(\varSigma)$ is the canonical
bundle of $\varSigma$ , and $(L_{\varSigma}\otimes K_{\varSigma})^{1/2}$
denotes a line bundle whose square is the line bundle $L_{\varSigma}\otimes K_{\varSigma}$.
In particular, if $L_{\varSigma}\simeq\mathbb{C}$ is the trivial
line bundle, then we obtain the canonical spin-c structure on $\varSigma$
\[
S_{can}\simeq K_{\varSigma}^{1/2}\oplus K_{\varSigma}^{-1/2}
\]
Notice that an implicit choice of \emph{spin} structure on $\varSigma$
was made so that we could find a square root of the canonical bundle,
so we will assume a fixed choice once and for all. Moreover, if $[u]=0$,
then $c_{1}(L_{Y})=0$ which means that we can regard the torsion
spin-c structure on $Y$ as being of the form 
\[
S_{tor}=K_{\varSigma}^{1/2}\oplus K_{\varSigma}^{-1/2}
\]
Now suppose that $E$ is a $U(2)$ bundle on $\varSigma$. We can
then consider the spin-u structure
\[
V=(K_{\varSigma}^{1/2}\oplus K_{\varSigma}^{-1/2})\otimes E=(\mathbb{C}\oplus K_{\varSigma}^{-1})\otimes(K_{\varSigma}^{1/2}\otimes E)
\]

Therefore, it is more elegant to reabsorb $K_{\varSigma}^{1/2}$ into
the bundle $E$ we are considering and work with the canonical spin-u
structure \footnote{Notice that $\mathbb{C}\oplus K_{\varSigma}^{-1}$ is not a torsion
spin-c structure in general, but this does not make a big difference
for our arguments, since $\text{ad}(K_{\varSigma}^{1/2}\otimes E)\simeq\text{ad}(E)$.
Moreover, the Chern connection is being implicitly used on $K_{\varSigma}^{-1}$
, and the trivial connection on the $\mathbb{C}$ factor.}
\[
V_{can}=(\mathbb{C}\oplus K_{\varSigma}^{-1})\otimes E
\]

It suffices to define the Clifford map $\rho$ on $S=\mathbb{C}\oplus K_{\varSigma}^{-1}$
. First of all, notice that
\[
\rho(d\theta)^{2}=-|d\theta|_{g}^{2}1_{S}=-1_{S}
\]
so $\rho(d\theta)$ is an endomorphism of $S$ which squares to $-1$.
We will thus identify $\rho(d\theta)$ with the matrix  
\[
\rho(d\theta)=\left(\begin{array}{cc}
i & 0\\
0 & -i
\end{array}\right)
\]

Decompose a one form $\eta\in T^{*}(S^{1}\times\varSigma)$ as 
\[
\eta=td\theta+\varpi
\]
where $\varpi$ is identified with the pullback of a $(0,1)$ form
on $\varSigma$. Notice that our conventions are such that 
\[
|\eta|_{g_{Y}}=t^{2}+2|\varpi|_{g_{\varSigma}}^{2}
\]
and that we are regarding $\varOmega^{0,1}(\varSigma)$ only as a
\emph{real }vector space, so that we have the isomorphism $\varOmega^{0,1}(\varSigma)\simeq\varOmega^{1}(\varSigma)$.
\textbf{}

Finally, we define for $(\alpha,\beta)\in\mathbb{C}\oplus\pi^{*}K_{\varSigma}^{-1}$
the Clifford map
\[
\rho(\eta)\left(\begin{array}{c}
\alpha\\
\beta
\end{array}\right)=\left(\begin{array}{c}
it\alpha-\sqrt{2}\left\langle \beta,\varpi\right\rangle \\
\sqrt{2}\alpha\varpi-it\beta
\end{array}\right)
\]
Notice that $K_{\varSigma}^{-1}$ is being interpreted as a \emph{complex
}vector space, and that $\left\langle \beta,\varpi\right\rangle $
refers to the hermitian inner product on $\varOmega^{0,1}(\varSigma)$,
complex linear in the first entry. Slightly abusing notations we can
write this as 
\[
\rho(\eta)=\left(\begin{array}{cc}
i\left\langle d\theta,\eta\right\rangle  & -\sqrt{2}\varpi^{*}\\
\sqrt{2}\varpi & -i\left\langle d\theta,\eta\right\rangle 
\end{array}\right)
\]
It is also useful to have local expressions for these formulas. Consider
a coframe $e^{1},e^{2}$ of $T^{*}\varSigma$ and define 
\[
\begin{cases}
\epsilon=\frac{1}{\sqrt{2}}(e^{1}+ie^{2})\\
\bar{\epsilon}=\frac{1}{\sqrt{2}}(e^{1}-ie^{2})
\end{cases}
\]
Under an almost complex structure $J$ of $\varSigma$, we have that
$J(e^{1})$ is mapped to $e^{2}$, thus
\[
J(\epsilon)=\frac{1}{\sqrt{2}}(e^{2}-ie^{1})=-\frac{i}{\sqrt{2}}(e^{1}+ie^{2})=-i\epsilon
\]
In other words, $\epsilon$ is of type $(1,0)$. Likewise, $\bar{\epsilon}$
is of type $(0,1)$ since $J(\bar{\epsilon})=i\bar{\epsilon}$. It
is an easy check that 
\begin{align*}
\rho(\epsilon)(\alpha,\beta)=(-\sqrt{2}\left\langle \beta,\bar{\epsilon}\right\rangle ,0) & \iff\left(\begin{array}{cc}
0 & -\sqrt{2}\epsilon\\
0 & 0
\end{array}\right)\\
\rho(\bar{\epsilon})(\alpha,\beta)=(0,\sqrt{2}\bar{\epsilon}\alpha) & \iff\left(\begin{array}{cc}
0 & 0\\
\sqrt{2}\bar{\epsilon} & 0
\end{array}\right)
\end{align*}
Following the discussion in \cite[Section 3.2]{MR3875977}, we can
write a connection $B$ on $S^{1}\times\varSigma$ in the form 
\[
B=C(\theta)+c(\theta)d\theta
\]
 where $C(\theta)$ is a one parameter periodic family of $U(2)$
connections on $\varSigma$ and $c(\theta)\in\varOmega^{0}(\varSigma,\mathfrak{su}(2))$
is also another periodic family. The curvature becomes 
\begin{equation}
F_{B}=F_{C(\theta)}+d\theta\wedge\left(\frac{\partial C(\theta)}{\partial\theta}+d_{C(\theta)}c(\theta)\right)\label{decomposition curvature}
\end{equation}
 Moreover, the Dirac operator can be written as {[}compare with \cite[Lemma 3.5]{MR3875977},
\cite[Eq 5.12]{MR1611061}, \cite[Section 3.1]{MR2814036}{]}

\[
D_{B}=\left(\begin{array}{cc}
i\left(\frac{\partial}{\partial\theta}+c(\theta)\right) & \sqrt{2}\bar{\partial}_{C(\theta)}^{*}\\
\sqrt{2}\bar{\partial}_{C(\theta)} & -i\left(\frac{\partial}{\partial\theta}+c(\theta)\right)
\end{array}\right)
\]
Therefore, the Dirac equation for $(B,(\alpha,\beta))$ reads 
\[
\begin{cases}
i\left(\frac{\partial}{\partial\theta}+c(\theta)\right)\alpha+\sqrt{2}\bar{\partial}_{C(\theta)}^{*}\beta=0\\
\sqrt{2}\bar{\partial}_{C(\theta)}\alpha-i\left(\frac{\partial}{\partial\theta}+c(\theta)\right)\beta=0
\end{cases}
\]
where $\alpha\in\varOmega^{0}(E)$ and $\beta\in\varOmega^{0,1}(E)$.

In order to understand the curvature equation, decompose the curvature
in terms of a co-frame as 
\[
F_{B}^{0}=F_{\theta,e^{1}}d\theta\wedge e^{1}+F_{\theta,e^{2}}d\theta\wedge e^{2}+F_{e^{1},e^{2}}e^{1}\wedge e^{2}
\]
\[
*F_{B}^{0}=F_{\theta,e^{1}}e^{2}-F_{\theta,e^{2}}e^{1}+F_{e^{1},e^{2}}d\theta
\]
Since $\frac{\epsilon+\bar{\epsilon}}{\sqrt{2}}=e^{1}$, $\frac{i(\bar{\epsilon}-\epsilon)}{\sqrt{2}}=e^{2}$
then $\rho(e^{1})=\left(\begin{array}{cc}
0 & -\epsilon\\
\bar{\epsilon} & 0
\end{array}\right)$, $\rho(e^{2})=\left(\begin{array}{cc}
0 & i\epsilon\\
i\bar{\epsilon} & 0
\end{array}\right)$
\[
\rho(*F_{B}^{0})=\left(\begin{array}{cc}
iF_{e^{1},e^{2}} & \left(iF_{\theta,e^{1}}+F_{\theta,e^{2}}\right)\epsilon\\
\left(iF_{\theta,e^{1}}-F_{\theta,e^{2}}\right)\bar{\epsilon} & -iF_{e^{1},e^{2}}
\end{array}\right)
\]
Thus the curvature equation reads (compare with \cite[p.10]{MR2700715})
\[
\begin{cases}
iF_{e^{1},e^{2}}+(\alpha\otimes\alpha^{*})_{0}-(\varLambda\beta\otimes\beta^{*})_{0}=0\\
\left(iF_{\theta,e^{1}}-F_{\theta,e^{2}}\right)\bar{\epsilon}+2(\beta\otimes\alpha^{*})_{0}=0
\end{cases}
\]
The expression $(\varLambda\beta\otimes\beta^{*})_{0}$ denotes that
a contraction with the symplectic form on $\varSigma$ is implicit.
\begin{thm}
\label{thm-solutions on S1xsigma} Suppose $(B,(\alpha,\beta))$ satisfies
the $SO(3)$ vortex equations on $S^{1}\times\varSigma$, for the
spin-u structure $V_{can}=(\mathbb{C}\oplus K_{\varSigma}^{-1})\otimes E$.
Here $\alpha\in\varOmega^{0}(E)$ and $\beta\in\varOmega^{0,1}(E)$.
Write $B=C(\theta)+c(\theta)d\theta$. Then
\begin{align}
i\left(\frac{\partial}{\partial\theta}+c(\theta)\right)\alpha+\sqrt{2}\bar{\partial}_{C(\theta)}^{*}\beta=0\label{equations on S1xSigma}\\
\sqrt{2}\bar{\partial}_{C(\theta)}\alpha-i\left(\frac{\partial}{\partial\theta}+c(\theta)\right)\beta=0\nonumber \\
iF_{e^{1},e^{2}}+(\alpha\otimes\alpha^{*})_{0}-(\varLambda\beta\otimes\beta^{*})_{0}=0\nonumber \\
\left(iF_{\theta,e^{1}}-F_{\theta,e^{2}}\right)\bar{\epsilon}+2(\beta\otimes\alpha^{*})_{0}=0\nonumber 
\end{align}
In fact, all solutions are of the form $(B,(\alpha,0))$ or $(B,(0,\beta))$. 

In the first case, we get a solution of the form 
\begin{align*}
\bar{\partial}_{C}\alpha=0\\
F_{C}-i(\alpha\otimes\alpha^{*})_{0}=0
\end{align*}
 which can be identified with an $SO(3)$ vortex on the bundle $E$. 

In the second case, we get a solution of the form 
\begin{align*}
\bar{\partial}_{C}^{*}\beta=0\\
F_{C}+i(\varLambda\beta\otimes\beta^{*})_{0}=0
\end{align*}
 which can be identified via Serre duality with an $SO(3)$ vortex
on the bundle $K_{\varSigma}\otimes E^{-1}$. 

In particular, if we assume that $\deg E>2\deg K_{\varSigma}$, then
this second type of moduli space of $SO(3)$ vortices is empty. 

\end{thm}

\begin{rem}
The previous theorem states that on $S^{1}\times\varSigma$ the solutions
to the $SO(3)$ monopole equations decompose into the union of two
moduli spaces of $SO(3)$ vortices on the Riemann surface. This is
the analogue of what happens for the solutions to the $SO(3)$ monopole
equations on a Kahler surface, where the moduli space also decomposes
into the union of two moduli spaces of stable pairs on the Kahler
surface (see \cite[Theorem 6.3.10]{MR2254074} and \cite[Section 1.2]{MR2700715}).
\end{rem}

\begin{proof}
We follow the proof in \cite[Proposition 3.1]{MR2141962}, \cite[Proposition 3.6]{MR3875977},
which applies with cosmetic changes to our situation.

Suppose that $(B,(\alpha,\beta))$ is irreducible in that $B$ is
irreducible and $(\alpha,\beta)$ non-vanishing. This forces the the
stabilizer of this $SO(3)$ monopole to be $\text{Id}$ under the
determinant one gauge group. Moreover, one can use a gauge transformation
to assume the connection is in ``temporal'' gauge, which means that
$B=C(\theta)$ instead of the more general case $B=C(\theta)+c(\theta)d\theta$.
So we can assume that $c(\theta)=0$ in the equations (\ref{equations on S1xSigma})

Since $i\bar{\partial}_{C(\theta)}^{*}=\varLambda\partial_{C(\theta)}$
on $(0,1)$ forms, the first equation (\ref{equations on S1xSigma})
can be written as 
\[
-\frac{\partial}{\partial\theta}\alpha+\sqrt{2}\varLambda\partial_{C(\theta)}\beta=0
\]
Differentiate this with respect to $\theta$ to find
\[
0=-\frac{\partial^{2}}{\partial\theta^{2}}\alpha+\sqrt{2}(\frac{\partial}{\partial\theta}\varLambda\partial_{C(\theta)})\beta+\sqrt{2}\varLambda\partial_{C(\theta)}\frac{\partial\beta}{\partial\theta}
\]
Because of the second equation in (\ref{equations on S1xSigma}),
$\frac{\partial}{\partial\theta}\beta=-\sqrt{2}i\bar{\partial}_{C(\theta)}\alpha$,
so we obtain
\[
0=-\frac{\partial^{2}}{\partial\theta^{2}}\alpha+\sqrt{2}(\frac{\partial}{\partial\theta}\varLambda\partial_{C(\theta)})\beta-2i\varLambda\partial_{C(\theta)}\underbrace{\bar{\partial}_{C(\theta)}\alpha}_{(0,1)\text{ form}}
\]
Using the Kahler identity $i\bar{\partial}_{C(\theta)}^{*}=\varLambda\partial_{C(\theta)}$
one more time we end up with 
\[
0=-\frac{\partial^{2}}{\partial\theta^{2}}\alpha+\sqrt{2}(\frac{\partial}{\partial\theta}\varLambda\partial_{C(\theta)})\beta+2\bar{\partial}_{C(\theta)}^{*}\bar{\partial}_{C(\theta)}\alpha
\]
Take inner product with $\alpha$, and use integration by parts and
the third equation in (\ref{equations on S1xSigma}) to conclude 
\[
0=-\int_{\check{\varSigma}}\left\langle \frac{\partial^{2}\alpha}{\partial\theta^{2}},\alpha\right\rangle +\sqrt{2}\int_{\check{\varSigma}}\left\langle (\alpha\otimes\beta^{*})_{0}\beta,\alpha\right\rangle +2\int_{\check{\varSigma}}|\bar{\partial}_{C(\theta)}\alpha|^{2}
\]
The integrand of the term in the middle is
\[
\left\langle |\beta|^{2}\alpha-\frac{1}{2}\left\langle \beta,\alpha\right\rangle _{\tilde{E}}\beta,\alpha\right\rangle =|\beta|^{2}|\alpha|^{2}-\frac{1}{2}|<\beta,\alpha>_{\tilde{E}}|^{2}
\]
Integrate from $0$ to $1$ to obtain (recall that we are thinking
of $S^{1}\times\varSigma$ as $[0,1]\times\varSigma$ with the ends
identified)
\[
|\frac{\partial}{\partial\theta}\alpha|^{2}+\sqrt{2}\left(|\beta|^{2}|\alpha|^{2}-\frac{1}{2}|<\beta,\alpha>_{\tilde{E}}|^{2}\right)+2|\bar{\partial}_{C(\theta)}\alpha|^{2}=0
\]
So $\alpha$ is $\theta$ independent and by Cauchy Schwarz since
$|<\beta,\alpha>_{\tilde{E}}|^{2}\leq|\beta|^{2}|\alpha|^{2}$ we
conclude that either $\alpha$ or $\beta$ vanishes identically. Because
of the curvature decomposition (\ref{decomposition curvature}) we
also find that $\frac{\partial C}{\partial\theta}=0$ thus all the
data pulls back from the surface $\varSigma$, and the claim of the
theorem is now straightforward.

For the last claim, suppose that $\deg E>2K_{\varSigma}$. Then $\deg(K_{\varSigma}\otimes E^{-1})<0$
so if $(C,\beta)$ denotes an $SO(3)$ vortex on the bundle $K_{\varSigma}\otimes E^{-1}$
, we have $c_{1}(L)<\frac{1}{2}\deg(K_{\varSigma}\otimes E^{-1})<0$
for any holomorphic sub-line bundle of $E$ for which $\beta\in H^{0}(L)$
and this is clearly impossible. Thus the moduli space of $SO(3)$
vortices must be empty.
\end{proof}
Now we specialize to the case that $\varSigma=T^{2}$. The canonical
bundle of an elliptic curve is trivial, and choose the $U(2)$ bundle
over $T^{2}$ with $\deg E=1$. Now we discuss the solutions to the
$SO(3)$ monopole equations $(B,\varPsi)$ on $S^{1}\times T^{2}$.

$\bullet$ For the irreducible $SO(3)$ monopoles, where $B$ is an
irreducible $SO(3)$ connection, and $\varPsi$ non-vanishing, by
the previous theorem (\ref{thm-solutions on S1xsigma}), $\varPsi=\alpha\oplus0$
or $\varPsi=0\oplus\beta$, and $B$ pulls back from a connection
$C$ on $T^{2}$. Moreover, $(C,\alpha)$ is an $SO(3)$ vortex on
$E$ and there are no $SO(3)$ vortices of the second kind since $\deg E=1>0=2\deg K_{T^{2}}$.
As described in Example (\ref{example T2}), the dimension of the
moduli space of $SO(3)$ vortices on $E$ is two, which becomes one
dimensional after taking the quotient by the $S^{1}$ action.

$\bullet$ The moduli space of flat connections on $T^{3}$. As described
in \cite[Section 4.1]{MR2860345}, there are two flat connections
on $T^{3}$. One corresponds to the flat connection (representation)
on $T^{2}$, while the other is obtained from this one by applying
a twist to the representation on $T^{2}$.

$\bullet$ The Seiberg-Witten solutions which appear associated to
our choice of spin-u structure must satisfy the perturbed Seiberg
Witten equations (\ref{perturbed monopole equations}):
\[
\begin{cases}
*F_{B_{L}}+\rho^{-1}(\psi\psi^{*})_{0}-\frac{1}{2}*F_{B^{\det}}=0\\
D_{B}\psi=0
\end{cases}
\]
Here $\psi$ corresponds to a section of $\underbrace{K_{\varSigma}^{1/2}\otimes L}_{\simeq L}\subset\underbrace{K_{\varSigma}^{1/2}\otimes E}_{\simeq E}$
for a splitting of $E=L\oplus(L^{*}\otimes\det E)$. Thus $\psi$
satisfies the abelian vortex equations 
\[
\begin{cases}
*_{T^{2}}F_{C}-\frac{i}{2}|\psi|^{2}=\frac{1}{2}*_{T^{2}}F_{C^{\det E}}\\
\bar{\partial}_{C}\psi=0
\end{cases}
\]
 The usual Chern-Weil argument says that 
\[
\deg L+\frac{1}{4\pi}\int_{T^{2}}|\psi|^{2}=\frac{1}{2}\deg E=\frac{1}{2}
\]
Since $\deg L\geq0$, this forces $\deg L=0$ and thus $\psi$ is
a holomorphic \emph{function }on $T^{2}$. Thus it is constant and
up to gauge there is a unique solution. In this way we find that there
is a unique solution to these perturbed Seiberg-Witten equations on
$T^{3}$.

In particular, for this choice of $U(2)$ bundle on $T^{3}$, the
moduli space of $SO(3)$ monopoles on $T^{3}$ consists of a one-dimensional
moduli space (an interval) with one endpoint corresponding to the
Seiberg-Witten monopole, the other endpoint corresponding to the flat
connection which pulls back from the flat connection on $T^{2}$,
and an additional flat connection not connected to this one dimensional
moduli space. Based on this picture, we define:
\begin{defn}
Suppose that $Y$ is a closed oriented 3-manifold. The\textbf{ }\textbf{\emph{framed
monopole Floer}}\emph{ }groups of $Y$, denoted $HM^{\#}(Y)$, is
defined as the monopole Floer homology $HM(Y\#T^{3},\omega,\varGamma)$.
Here $\omega$ is a non-balanced perturbation which is a harmonic
representative of $-\frac{\pi i}{2}c_{1}(E)\in H^{2}(T^{2};i\mathbb{R})\subset H^{2}(T^{3};i\mathbb{R})$
for the $U(2)$ bundle $E$ over $T^{2}$ of degree one. Since $\omega$
is a non-balanced perturbation, the definition of the monopole Floer
homology requires the use of a local coefficient system $\varGamma$. 
\end{defn}

\begin{rem}
Some choices of local systems are described in \cite[Chapter VIII]{MR2388043}
as well as in the author's paper \cite{Echeverria[Furuta]}.
\end{rem}

As explained in the previous section, the fact that we are using a
non-balanced perturbation means that there are no reducible solutions
to the Seiberg-Witten equations on the 3 manifold $Y\#T^{3}$. In
fact, there are only finitely many irreducible Seiberg-Witten monopoles
which make up the chain complex $CM(Y\#T^{3},\omega)$ which computes
$HM(Y\#T^{3},\omega,\varGamma)$. Notice that 
\[
HM(Y\#T^{3},\omega,\varGamma)=\bigoplus_{\mathfrak{s}\#\mathfrak{s}'}HM(Y\#T^{3},\omega,\varGamma,\mathfrak{s}\#\mathfrak{s}')
\]
where $\mathfrak{s}\#\mathfrak{s}'$ denotes a spin-c structure on
$Y\#T^{3}$. Observe that our notation for the framed monopole Floer
groups $HM^{\#}(Y)$ suggest that they depend on $Y$, not on $Y\#T^{3}$.
In order to make this more precise we need to use the Künneth formula
from \cite[Section 6]{MR4194309}.

In the notation of Kutluhan, Lee, Taubes, $M_{1},M_{2}$ are two three
manifolds and $M_{\#}$ denotes its connected sum. The Floer complex
for $M_{1}$ is assumed to come from a \emph{nonbalanced perturbation.
}In particular, the complex for $M_{\#}$ uses a non-balanced perturbation
as well. The connected sum theorem relates $CM_{*}(M_{\#})$ with
$S_{U_{\sqcup}}(\hat{C}_{*}(M_{\sqcup}))$. The latter refers to $\hat{C}_{\sqcup}\otimes\mathbb{Z}[y]$,
where $y$ acts as a degree 1 chain map \cite[Section 4.1]{MR4194309}.
If we write this as $\hat{C}_{\sqcup}\oplus y\hat{C}_{\sqcup}$ ,
then the differential takes the block form 
\[
D_{\sqcup}=\left(\begin{array}{cc}
\hat{\partial}_{\sqcup} & 0\\
\hat{U}_{\sqcup} & -\hat{\partial}_{\sqcup}
\end{array}\right)
\]
Proposition 6.7 in \cite[Section 6]{MR4194309} then states that there
is a chain homotopy equivalence between the complex $C(M_{\#},\mathfrak{s}_{\#},\omega_{\#},\varGamma_{\#})$
and the complex $S_{U_{\sqcup}}(\hat{C}_{*}(M_{\sqcup},\mathfrak{s}_{\sqcup},\omega_{\sqcup};\varGamma_{\sqcup})$. 

In our case, $M_{1}=T^{3}$, $M_{2}=Y$. The differential $\hat{\partial}_{\sqcup}=1\otimes\partial_{\natural}^{\#}(M_{2})+\partial_{o}^{o}(M_{1})\otimes1)$
reduces to $1\otimes\partial_{\natural}^{\#}(Y)$. Likewise, $\hat{U}_{\sqcup}=1\otimes\hat{U}_{p_{2}}-\hat{U}_{p_{1}}\otimes1$
reduces to $1\otimes\hat{U}_{p_{2}}$. Thus, it is not difficult to
see that in our case the complex $S_{U_{\sqcup}}(\hat{C}_{*}(M_{\sqcup},\mathfrak{s}_{\sqcup},\omega_{\sqcup};\varGamma_{\sqcup})$
is isomorphic to $S_{U_{Y}}(\hat{C}_{*}(Y,\mathfrak{s};\varGamma))$,
but this is precisely the complex which computes the monopole Floer
group $\widetilde{HM}(Y,\varGamma)$ \cite[Section 8]{MR2764887},
the only difference is that we now end up with a twisted version of
this group. In other words, we have
\begin{thm}
\label{theo iso framed group}The framed monopole Floer group $HM^{\#}(Y)$
is isomorphic to the group $\widetilde{HM}(Y,\varGamma)$. Moreover,
the Euler characteristic of $HM^{\#}(Y)$ (hence $\widetilde{HM}(Y,\varGamma)$)
is the same as the one of $HI^{\#}(Y)$.
\end{thm}

\begin{rem}
The equality of the Euler characteristics was explained near the end
of the introduction to this work.
\end{rem}

\section{$SO(3)$ Monopoles on\label{sec:-Seifert} Seifert Manifolds}

Now we proceed to analyze the $SO(3)$ monopole equations on a Seifert
manifold $Y$. We will follow \cite{MR1651420} and \cite{MR1611061}.
Unfortunately, our notation and some of our conventions are not quite
isomorphic to those used in either reference, so we have to redo some
of their computations. 

For the identies we require we need to have very good control on all
the constants appearing in our formulas, so at the cost of redudance
we will reprove some of the results in each paper.

First of all, for an $SO(3)$ connection $^{\circ}\nabla$ on the
cotangent bundle $T^{*}Y$ which is not necessarily the Levi-Civita
connection, we say that a connection 
\[
\nabla^{S}:\varGamma(S)\rightarrow\varGamma(T^{*}Y\otimes S)
\]
 is a \textbf{spinorial connection} with respect to $^{\circ}\nabla$
if all vector fields $v$ on $Y$, sections $\varPhi$ of $S$ and
one forms $\theta$ we have 
\begin{equation}
\nabla_{v}^{S}(\rho(\theta)\varPhi)=\rho(^{\circ}\nabla_{v}\theta)\varPhi+\rho(\theta)(\nabla_{v}^{S}\varPhi)\label{eq:Clifford map}
\end{equation}
If $(e^{i})$ is a coframe and the connection matrix of $^{\circ}\nabla$
is given by \footnote{Our convention is chosen in such a way that if $^{\circ}\nabla^{TY}$
is the induced connection on the tangent bundle then $^{\circ}\nabla^{TY}e_{i}=\sum_{j=1}^{3}\omega_{ji}e_{j}$
where $e_{i}$ an orthonormal frame}

\[
^{\circ}\nabla e^{i}=\omega e^{i}=\sum_{j=1}^{3}\omega_{ji}\otimes e^{j}
\]
where $\omega\in\varOmega^{1}(Y)\otimes\mathfrak{so}(3)$ then then
the connection matrix of $\nabla^{S}$ is 
\begin{equation}
-\frac{1}{4}\sum_{i,j=1}^{3}\omega_{ij}\otimes\rho(e^{i}\wedge e^{j})+bId_{S}\label{coefficients spin connection}
\end{equation}
where $b$ is any imaginary valued one form.

With this baggage out of the way we can start defining some special
geometric structures on $Y=S(L)$. Recall that $S(L)$ is a principal
$S^{1}$ bundle. As such, we can find a connection $i\eta$ for $S(L)$
and we assume that it is of constant curvature. We give $Y$ the metric
\[
g_{Y}=\eta\otimes\eta+\pi^{*}(g_{\check{\varSigma}})
\]
This is a locally homogeneous Riemann metric on $Y=S(L)$. The global
1 form induces a reduction in the structure group of $TY$ to $SO(2)$;
the kernel of $\eta$ is naturally identified with the pull back of
the orbifold tangent bundle of $\varSigma$, so that we have orthogonal
splittings 
\[
\begin{cases}
TY\simeq\mathbb{R}\frac{\partial}{\partial\varphi}\oplus\pi^{*}(T\check{\varSigma})\\
T^{*}Y\simeq\mathbb{R}\eta\oplus\pi^{*}(T^{*}\check{\varSigma})
\end{cases}
\]
Here $\frac{\partial}{\partial\varphi}$ is the vector field dual
to $\eta$ with respect to $g_{Y}$. Notice that $\frac{\partial}{\partial\varphi}$
has unit length with respect to this metric and in fact it is a Killing
vector field for the metric $g_{Y}$, that is, the Lie derivative
of the metric $g_{Y}$ with respect to $\frac{\partial}{\partial\varphi}$
vanishes.

Moreover, for any vector field $v_{\check{\varSigma}}$ dual to the
pull back of a 1 form from $\check{\varSigma}$, we have that $\left[v_{\check{\varSigma}},\frac{\partial}{\partial\varphi}\right]=0$.
If $\nabla_{\check{\varSigma}}^{LC}$ is the Levi-Civita connection
on $T^{*}\check{\varSigma}$, we give $T^{*}Y$ the so called \textbf{adiabatic
connection
\begin{equation}
\nabla^{\infty}=d\oplus\pi^{*}(\nabla_{\check{\varSigma}}^{LC})\label{eq:adiabatic connection}
\end{equation}
}In particular, 
\[
\begin{cases}
\nabla^{\infty}\eta=0\\
\nabla^{\infty}\pi^{*}(\check{\theta})=\pi^{*}(\nabla_{\check{\varSigma}}^{LC}\check{\theta})
\end{cases}
\]
where $\check{\theta}$ is any 1 form on $\check{\varSigma}$. Since
$d\eta$ is constant, let $c_{\eta}$ be the constant determined by
\begin{equation}
d\eta=2c_{\eta}\pi^{*}(\mu_{\check{\varSigma}})\label{eq:constant Seifert}
\end{equation}
Using Chern-Weil theory one deduces that 
\begin{equation}
c_{\eta}=-\frac{\pi\int_{\check{\varSigma}}c_{1}(L)}{\text{vol}(\check{\varSigma})}=-\frac{\pi\deg(\check{L})}{\text{vol}(\check{\varSigma})}\label{eq:constant Seifert 2}
\end{equation}
The following estimate will allows us to compare two Dirac operators
associated to the adiabatic and Levi-Civita connections.
\begin{lem}
(Comparing Dirac Operators\label{lem:(Comparing-Dirac-Operators)},
\cite[Lemma 5.5]{MR1611061}) Let $(S,\rho)$ be a spin-c structure
over $Y=S(L)$ and $\nabla^{S},\nabla_{LC}^{S}$ a pair of connections
which are spinorial with respect to $\nabla^{\infty}$ and $\nabla^{LC}$
respectively, which induce the same connection on the determinant
line bundle of $S$. Then, for any $\theta\in\varOmega^{1}(Y,\mathbb{R})$
, we have 
\[
\nabla_{LC,\theta^{\#}}^{S}=\nabla_{\theta^{\#}}^{S}+\frac{c_{\eta}}{2}\left(\rho(\theta)-2\left\langle \theta,\eta\right\rangle \rho(\eta)\right)
\]
where $\theta^{\#}$ denotes the vector field which is $g_{Y}$ dual
to $\theta$. If $D_{LC}$ and $D^{\infty}$ are the corresponding
Dirac operators then 
\[
D_{LC}=D^{\infty}-\frac{1}{2}c_{\eta}
\]
In particular, this implies that the Dirac operator $D^{\infty}$
is self adjoint, since $D_{LC}$ and multiplication by a real scalar
valued function are self adjoint. 
\end{lem}

\begin{proof}
Recall the Cartan's structure equations \cite[Section 4.3]{MR1735502}.
Let $\{e^{0},e^{1},e^{2}\}$ be a (local) orthonormal frame of $T^{*}Y$.
Define the connection matrix for the Levi Civita connection $\nabla^{LC}$
by 
\[
\nabla^{LC}e^{i}=\sum_{j=0}^{2}\omega_{ji}^{LC}\otimes e_{j}
\]
then
\[
de^{i}=\sum_{j=0}^{2}e^{j}\wedge\omega_{ij}^{LC}=-\sum_{j=0}^{2}\omega_{ij}^{LC}\wedge e^{j}
\]
These can be conveniently rewritten as
\[
d\left[\begin{array}{c}
e^{0}\\
e^{1}\\
e^{2}
\end{array}\right]=-\left[\begin{array}{ccc}
0 & \omega_{01}^{LC} & \omega_{02}^{LC}\\
\omega_{10}^{LC} & 0 & \omega_{12}^{LC}\\
\omega_{20}^{LC} & \omega_{21}^{LC} & 0
\end{array}\right]\wedge\left[\begin{array}{c}
e^{0}\\
e^{1}\\
e^{2}
\end{array}\right]
\]
Choose $e^{0}=\eta$, and $e^{1},e^{2}$ the pull back from an orthonormal
frame on $\check{\varSigma}$. We obtain the system of equations 
\begin{align}
\begin{array}{ccccc}
d\eta & = &  & -\omega_{01}^{LC}\wedge e^{1} & -\omega_{02}^{LC}\wedge e^{2}\\
de^{1} & = & -\omega_{10}^{LC}\wedge e^{0} &  & -\omega_{12}^{LC}\wedge e^{2}\\
de^{2} & = & -\omega_{20}^{LC}\wedge e^{0} & -\omega_{21}^{LC}\wedge e^{1}
\end{array}\label{eq:structure 1}
\end{align}
Given that
\[
d\eta=2c_{\eta}\mu_{\check{\varSigma}}=2c_{\eta}e^{1}\wedge e^{2}
\]
our first equation in \ref{eq:structure 1}becomes 
\[
2c_{\eta}e^{1}\wedge e^{2}=e^{1}\wedge\omega_{01}^{LC}-\omega_{02}^{LC}\wedge e^{2}
\]
Therefore we have
\begin{align*}
\omega_{01}^{LC}=c_{\eta}e^{2}\\
\omega_{02}^{LC}=-c_{\eta}e^{1}
\end{align*}
Also, if we let $\omega_{\check{\varSigma}}^{LC}$ be the connection
form for the Levi-Civita connection, since pull back commutes with
exterior differentiation then the structure equations for $\omega_{\check{\varSigma}}^{LC}$
are
\begin{equation}
d\left[\begin{array}{c}
e^{1}\\
e^{2}
\end{array}\right]=-\left[\begin{array}{cc}
0 & \omega_{\check{\varSigma},12}^{LC}\\
\omega_{\check{\varSigma},21}^{LC} & 0
\end{array}\right]\wedge\left[\begin{array}{c}
e^{1}\\
e^{2}
\end{array}\right]\label{structure 2}
\end{equation}
Comparing \ref{structure 2} with \ref{eq:structure 1} and using
the antisymmetry of the $\omega_{ij}$ we obtain the equations
\begin{align*}
c_{\eta}e^{2}\wedge e^{0}-\omega_{12}^{LC}\wedge e^{2}=-\omega_{\check{\varSigma},12}^{LC}\wedge e^{2}\implies(-c_{\eta}e^{0}-\omega_{12}^{LC}+\omega_{\check{\varSigma},12}^{LC})\wedge e^{2}=0\\
-c_{\eta}e^{1}\wedge e^{0}-\omega_{21}^{LC}\wedge e^{1}=-\omega_{\check{\varSigma},21}^{LC}\wedge e^{1}\implies e^{1}\wedge(-c_{\eta}e^{0}+\omega_{21}^{LC}-\omega_{\check{\varSigma},21}^{LC})=0
\end{align*}
Clearly we obtain
\[
\omega_{21}^{LC}=c_{\eta}e^{0}+\omega_{\check{\varSigma},21}^{LC}
\]
Therefore the connection matrix for the Levi-Civita connection $\omega^{LC}$
becomes \footnote{As mentioned in \cite{MR1611061}, we are writing the connection matrix
for $TY$ which differs by a sign from the connection matrix for $T^{*}Y$
, hence explaining the difference sign with their paper}

\begin{align}
\left(\begin{array}{ccc}
0 & \omega_{01}^{LC} & \omega_{02}^{LC}\\
\omega_{10}^{LC} & 0 & \omega_{12}^{LC}\\
\omega_{20}^{LC} & \omega_{21}^{LC} & 0
\end{array}\right)=\left(\begin{array}{ccc}
0 & c_{\eta}e^{2} & -c_{\eta}e^{1}\\
-c_{\eta}e^{2} & 0 & -c_{\eta}e^{0}-\omega_{\check{\varSigma},21}^{LC}\\
c_{\eta}e^{1} & c_{\eta}e^{0}+\omega_{\check{\varSigma},21}^{LC} & 0
\end{array}\right)\label{eq:connection matrix LC}
\end{align}
The connection matrix for $\nabla^{\infty}$ is easier to find since
it is a reducible connection and trivial in the first summand; it
is simply 
\begin{equation}
\left(\begin{array}{ccc}
0 & 0 & 0\\
0 & 0 & -\omega_{\check{\varSigma},21}^{LC}\\
0 & \omega_{\check{\varSigma},21}^{LC} & 0
\end{array}\right)\label{connection matrix reducible}
\end{equation}
Therefore the difference between connection matrices is
\[
\nabla^{LC}-\nabla^{\infty}\iff\left(\begin{array}{ccc}
0 & c_{\eta}e^{2} & -c_{\eta}e^{1}\\
-c_{\eta}e^{2} & 0 & -c_{\eta}e^{0}\\
c_{\eta}e^{1} & c_{\eta}e^{0} & 0
\end{array}\right)=(\tilde{\omega}_{ij})
\]
Using formula \ref{coefficients spin connection} the corresponding
spin connections differ by
\begin{align*}
\nabla_{LC}^{S}-\nabla^{S} & = & -\frac{1}{4}\sum_{i,j=0}^{2}\tilde{\omega}_{ij}\otimes\rho(e^{i})\rho(e^{j})\\
 & = & -\frac{1}{2}\sum_{i<j}\tilde{\omega}_{ij}\otimes\rho(e^{i})\rho(e^{j})\\
 & = & -\frac{1}{2}\left(c_{\eta}e^{2}\otimes\rho(e^{0})\rho(e^{1})-c_{\eta}e^{1}\otimes\rho(e^{0})\rho(e^{2})-c_{\eta}e^{0}\otimes\rho(e^{1})\rho(e^{2})\right)\\
 & = & \frac{c_{\eta}}{2}\left(e^{0}\otimes\rho(e^{1})\rho(e^{2})+e^{1}\otimes\rho(e^{0})\rho(e^{2})-e^{2}\otimes\rho(e^{0})\rho(e^{1})\right)\\
 & = & \frac{c_{\eta}}{2}\left(-e^{0}\otimes\rho(e^{0})+e^{1}\otimes\rho(e^{1})+e^{2}\otimes\rho(e^{2})\right)\rho(\mu_{Y})\\
 & = & \frac{c_{\eta}}{2}\left(-e^{0}\otimes\rho(e^{0})+e^{1}\otimes\rho(e^{1})+e^{2}\otimes\rho(e^{2})\right)
\end{align*}
where we used $1_{S}=\rho(\mu_{Y})=\rho(e_{0})\rho(e_{1})\rho(e_{2})$.
To check the first formula in the lemma $\nabla_{LC,\theta^{\#}}^{S}=\nabla_{\theta^{\#}}^{S}+\frac{c_{\eta}}{2}\left(\rho(\theta)-2\left\langle \theta,\eta\right\rangle \rho(\eta)\right)$
notice that it is linear in $\theta^{\#}$ and so we just need to
verify it on the coframe for each covector $e^{0},e^{1},e^{2}$.

We will just check it for $\theta=e^{0}$ and $\theta=e^{1}$ since
$\theta=e^{2}$ is entirely analogous to the second case. In the case
$\theta^{\#}=e_{0}$ because of the previous calculation
\begin{align*}
\nabla_{LC,e_{0}}^{S}-\nabla_{e_{0}}^{S} & =\frac{c_{\eta}}{2}\left(-e^{0}\otimes\rho(e^{0})+e^{1}\otimes\rho(e^{1})+e^{2}\otimes\rho(e^{2})\right)e_{0}\\
 & =-\frac{c_{\eta}}{2}\rho(e^{0})
\end{align*}
On the other hand
\begin{align*}
\frac{c_{\eta}}{2}\left(\rho(e^{0})-2\left\langle e^{0},\eta\right\rangle \rho(\eta)\right) & =\frac{c_{\eta}}{2}\left(\rho(e^{0})-2\rho(\eta)\right)\\
 & =-\frac{c_{\eta}}{2}\rho(e^{0})
\end{align*}
and so both formulas agree. When $\theta=e^{1}$ our long calculation
implies that 
\begin{align*}
\nabla_{LC,e_{1}}^{S}-\nabla_{e_{1}}^{S} & =\frac{c_{\eta}}{2}\left(-e^{0}\otimes\rho(e^{0})+e^{1}\otimes\rho(e^{1})+e^{2}\otimes\rho(e^{2})\right)e_{1}\\
 & =\frac{c_{\eta}}{2}\rho(e^{1})
\end{align*}
On the other hand
\[
\begin{array}{ccc}
\frac{c_{\eta}}{2}\left(\rho(e^{1})-2\left\langle e^{1},\eta\right\rangle \rho(\eta)\right) & = & \frac{c_{\eta}}{2}\rho(e^{1})\end{array}
\]
Which verifies the formula. The statement about the Dirac operators
$D_{LC}=D^{\infty}-\frac{1}{2}c_{\eta}$ also follows because
\begin{align*}
D_{LC}-D^{\infty} & =\frac{c_{\eta}}{2}\left(-\rho(e^{0})\rho(e^{0})+\rho(e^{1})\rho(e^{1})+\rho(e^{2})\rho(e^{2})\right)\\
 & =\frac{c_{\eta}}{2}\left(1-1-1\right)\\
 & =-\frac{c_{\eta}}{2}
\end{align*}
\end{proof}
As a corollary, we also have
\begin{cor}
(\cite[Corollary 5.7]{MR1611061}) The anticommutator of the Dirac
operator with Clifford multiplication by a one form $\theta$ is given
by the formula
\begin{equation}
\{D^{\infty},\rho(\theta)\}=-\rho((*d+d*)\theta)-2\nabla_{\theta^{\#}}^{S}+2c_{\eta}<\theta,\eta>\rho(\eta)\label{eq:eq:anticommutator dirac one form}
\end{equation}
\end{cor}

\begin{proof}
This follows from the analogous formula for the Levi-Civita operator
\cite[Proposition 3.45]{MR2273508}
\[
\{D_{LC},\rho(\theta)\}=\rho((*d+d*)\theta)-2\nabla_{LC,\theta^{\#}}^{S}
\]

Since $D_{LC}=D^{\infty}-\frac{c_{\eta}}{2}$ and $\nabla_{LC,\theta^{\#}}^{S}=\nabla_{\theta^{\#}}^{S}+\frac{c_{\eta}}{2}\left(\rho(\theta)-2\left\langle \theta,\eta\right\rangle \rho(\eta)\right)$
we find that 
\[
\{D^{\infty},\rho(\theta)\}-c_{\eta}\rho(\theta)=\rho((*d+d*)\theta)-2\nabla_{\theta^{\#}}^{S}-c_{\eta}(\rho(\theta)-2\left\langle \theta,\eta\right\rangle \rho(\eta))
\]
from which the result is now obvious.
\end{proof}
For our purposes we are interested in the case where we twist our
spinor bundle with a rank two hermitian bundle $E$, endowed with
a $U(2)$ connection $B$. Then we have a twisted connection 
\[
\nabla^{S,B}:\varGamma(S\otimes E)\rightarrow\varGamma(T^{*}Y\otimes S\otimes E)
\]
 defined in the usual way 
\[
\nabla^{S,B}(s\otimes s_{E})=(\nabla^{S}s)\otimes s_{E}+s\otimes(\nabla^{B}s_{E})
\]
Here $s$ is a section of $S$, while $s_{E}$ a section of $E$.
The twisted Dirac operator is then 
\[
D_{S,B}=\sum_{i=0}^{2}\rho(e^{i})\nabla_{e_{i}}^{S,B}=\rho(\eta)\nabla_{\frac{\partial}{\partial\varphi}}^{S,B}+\rho(e^{1})\nabla_{e_{1}}^{S,B}+\rho(e^{2})\nabla_{e_{2}}^{S,B}
\]
Observe first of all that since $D_{S}$ {[}called $D^{\infty}$ before{]}
is self-adjoint thanks to Lemma (\ref{lem:(Comparing-Dirac-Operators)}),
we know that $D_{S,B}$ is a self-adjoint operator as well. Hence
\[
D_{S,B}^{*}D_{S,B}=D_{S,B}^{2}
\]
and our objective is to find an useful decomposition for the square
of this operator, in other words, a Weitzenbock-type formula. Tautologically
we define 
\[
D_{2,S,B}\equiv D_{S,B}-\rho(\eta)\nabla_{\frac{\partial}{\partial\varphi}}^{S,B}
\]
 Then
\[
D_{S,B}^{2}=\left(\rho(\eta)\cdot\nabla_{\frac{\partial}{\partial\varphi}}^{S,B}\right)^{2}+D_{2,S,B}^{2}+\left\{ \rho(\eta)\nabla_{\frac{\partial}{\partial\varphi}}^{S,B},D_{2,S,B}\right\} 
\]
decomposes $D_{S,B}^{2}$ as a sum of three operators, the first two
which are evidently non-negative. The third term is an anti-commutator,
which we need to analyze explicitly. 
\begin{lem}
The term $\left\{ \rho(\eta)\nabla_{\frac{\partial}{\partial\varphi}}^{S,B},D_{2,S,B}\right\} $
is equal to 
\[
\left\{ \rho(\eta)\nabla_{\frac{\partial}{\partial\varphi}}^{S,B},D_{2,S,B}\right\} =\rho(\eta)\rho(e^{1})F_{S,B}(\partial_{\varphi},e_{1})+\rho(\eta)\rho(e^{2})F_{S,B}(\partial_{\varphi},e_{2})
\]
\end{lem}

\begin{rem*}
This computation should be compared with the formula given in \cite[Lemma 5.8]{MR1611061}.
There $E$ would be a rank complex line bundle, and the formula would
read in terms of an orthonormal frame
\[
\left\{ \rho(\eta)\nabla_{\frac{\partial}{\partial\varphi}}^{S,B},D_{2,S,B}\right\} =\rho(\eta)\left(\rho(e^{1})(db)_{\eta,e_{1}}+\rho(e^{2})(db)_{\eta,e_{2}}\right)1_{S}
\]
where $b$ is the imaginary valued one form which appears in \ref{coefficients spin connection}
and we decomposed it as 
\[
db=(db)_{\eta,e_{1}}\eta\wedge e^{1}+(db)_{\eta,e_{2}}\eta\wedge e^{2}+(db)_{e_{1},e_{2}}e^{1}\wedge e^{2}
\]
\end{rem*}
\begin{proof}
Clearly, 
\[
\left\{ \rho(\eta)\nabla_{\partial_{\varphi}}^{S},D_{2,SB}\right\} =\left\{ \rho(\eta)\nabla_{\partial_{\varphi}}^{S},\rho(e^{1})\nabla_{e_{1}}^{S,B}\right\} +\left\{ \rho(\eta)\nabla_{\partial_{\varphi}}^{S},\rho(e^{2})\nabla_{e_{2}}^{S,B}\right\} 
\]
where $\partial_{\varphi}=\frac{\partial}{\partial\varphi}$ , and
it suffices to find the first term since the other one is analogous.

For a spinor $\varPsi$ of $S\otimes E$ we have
\[
\left\{ \rho(\eta)\nabla_{\partial_{\varphi}}^{S,B},\rho(e^{1})\nabla_{e_{1}}^{S,B}\right\} \varPsi=\rho(\eta)\nabla_{\partial_{\varphi}}^{S,B}\left(\rho(e^{1})\nabla_{e_{1}}^{S,B}\varPsi\right)+\rho(e^{1})\nabla_{e_{1}}^{S,B}\left(\rho(\eta)\nabla_{\partial_{\varphi}}^{S,B}\varPsi\right)
\]
Now we use the fact that the connection is spinorial, i.e, $\nabla_{v}^{S,B}(\rho(\theta)\varPsi)=\rho(\nabla_{v}^{\infty}\theta)\varPsi+\rho(\theta)(\nabla_{v}^{S,B}\varPsi)$.
In our case the anticommutator becomes
\begin{align*}
 & \rho(\eta)\rho(\nabla_{\partial_{\varphi}}^{\infty}e^{1})\left(\nabla_{e_{1}}^{S,B}\varPsi\right)+\rho(\eta)\rho(e^{1})\nabla_{\partial_{\varphi}}^{S,B}\left(\nabla_{e_{1}}^{S,B}\varPsi\right)\\
+ & \rho(e^{1})\rho\left(\nabla_{e_{1}}^{\infty}\rho(\eta)\right)\left(\nabla_{\partial_{\varphi}}^{S,B}\varPsi\right)+\rho(e^{1})\rho(\eta)\nabla_{e_{1}}^{S,B}\left(\nabla_{\partial_{\varphi}}^{S,B}\varPsi\right)
\end{align*}
Now, for our adiabatic connection we have that $\nabla^{\infty}\eta=0$
and $\nabla^{\infty}\pi^{*}(\check{\theta})=\pi^{*}(\nabla_{\varSigma}^{LC}\check{\theta})$
so the first term in each row of the previous expression disappears.
Also, using the Clifford relations $\rho(e^{1})\rho(\eta)=-\rho(\eta)\rho(e^{1})$
the computation for the anticommutator becomes 
\[
\left\{ \rho(\eta)\nabla_{\partial_{\varphi}}^{S,B},\rho(e^{1})\nabla_{e_{1}}^{S,B}\right\} \varPsi=\rho(\eta)\rho(e^{1})[\nabla_{\partial_{\varphi}}^{S,B},\nabla_{e_{1}}^{S_{c}}]\varPsi
\]
To compute the commutator recall that because $\nabla^{S,B}$ is a
connection on a vector bundle it has a curvature $F_{S,B}$ given
by \cite[Section 3.3.3]{MR2363924}
\[
[\nabla_{v}^{S,B},\nabla_{w}^{S,B}]-\nabla_{[v,w]}^{S,B}=F_{S,B}(v,w)
\]
Since $\partial_{\varphi}$ commutes with any vector field which is
dual to a form that pulls back from $\varSigma$ we have that $[\partial_{\varphi},e_{1}]=0$
and so 
\[
\begin{array}{c}
[\nabla_{\partial_{\varphi}}^{S,B},\nabla_{e_{1}}^{S,B}]=F_{S,B}(\partial_{\varphi},e_{1})\end{array}
\]
 Since a similar result holds for $\partial_{\varphi}$ and $e_{2}$
we have just found that
\[
\left\{ \rho(\eta)\nabla_{\partial_{\varphi}}^{S,B},D_{2,SB}\right\} =\rho(\eta)\rho(e^{1})F_{S,B}(\partial_{\varphi},e_{1})+\rho(\eta)\rho(e^{2})F_{S,B}(\partial_{\varphi},e_{2})
\]

\end{proof}
\begin{rem}
Notice that as \cite{MR1611061} point out before Lemma 5.8, the term
$\left\{ \rho(\eta)\nabla_{\partial_{\varphi}}^{S,B},D_{2,SB}\right\} $
ends up being a zeroth-order operator, not a first order operator. 

Now we write the $SO(3)$ monopole equations on a Seifert manifold.
Locally, we can decompose the curvature in the same way as in the
previous section, namely
\begin{align*}
F_{B}^{0}=F_{\eta,e^{1}}\eta\wedge e^{1}+F_{\eta,e^{2}}\eta\wedge e^{2}+F_{e^{1},e^{2}}e^{1}\wedge e^{2}\\
*F_{B}^{0}=F_{\eta,e^{1}}e^{2}-F_{\eta,e^{2}}e^{1}+F_{e^{1},e^{2}}d\eta
\end{align*}
and we obtain:
\end{rem}

\begin{thm}
\label{thm solutions on Seifert manifolds}Suppose $(B,\varPsi)$
satisfies the $SO(3)$ vortex equations on a Seifert manifold $Y$,
for the spin-u structure $V_{can}=(\mathbb{C}\oplus\pi^{*}(K_{\check{\varSigma}}^{-1}))\otimes\pi^{*}(\check{E})$.
Write $\varPsi=\alpha\oplus\beta$ with $\alpha\in\varOmega^{0}(\pi^{*}(\check{E}))$
and $\beta\in\varOmega^{0,1}(\pi^{*}(K_{\check{\varSigma}}^{-1}\otimes\check{E}))$. 

Then with respect to the $SO(3)$ monopole equations
\begin{align}
iF_{e^{1},e^{2}}+(\alpha\otimes\alpha^{*})_{0}-(\varLambda\beta\otimes\beta^{*})_{0}=0\nonumber \\
\left(iF_{\eta,e^{1}}-F_{\eta,e^{2}}\right)\bar{\epsilon}+2(\beta\otimes\alpha^{*})_{0}=0\nonumber \\
\rho(\eta)\nabla_{\frac{\partial}{\partial\varphi}}^{S,B}\varPsi+D_{2,S}\varPsi=0\label{eq:Seifert}
\end{align}
all solutions are of the form $(B,(\alpha,0))$ or $(B,(0,\beta))$. 

In the first case, we get a solution of the form 
\begin{align*}
\bar{\partial}_{C}\alpha=0\\
F_{C}-i(\alpha\otimes\alpha^{*})_{0}=0
\end{align*}
 which can be identified with a $SO(3)$ vortex on a bundle $\check{E}'$
which satisfies that $\det\check{E}'\simeq\det\check{E}$. 

In the second case, we get a solution of the form 
\begin{align*}
\bar{\partial}_{C}^{*}\beta=0\\
F_{C}+i(\varLambda\beta\otimes\beta^{*})_{0}=0
\end{align*}
 which can be identified via Serre duality with an $SO(3)$ vortex
on the bundle $K_{\check{\varSigma}}\otimes\check{E}'^{-1}$, where
again $\tilde{E}$ is a $U(2)$ bundle satisfying $\det\check{E}'\simeq\det E$. 

In particular, if we assume that $c_{1}(\check{E})>2c_{1}(K_{\check{\varSigma}})$,
then this second type of moduli space of stable pairs is empty.
\end{thm}

\begin{rem}
a) Notice that implicitly we chose a connection $C^{\det}$ on $\det\check{E}$,
and we are considering all $U(2)$ connections on $\pi^{*}(\check{E})$
which induce the predetermined connection $\pi^{*}(C^{\det})$ on
$\pi^{*}(\det\check{E})$. Using \cite[Propositon 5.3]{MR1611061},
there is a bijection between bundles such connections on orbifold
line bundles over $\check{\varSigma}$, and the usual connections
on line bundles over $Y$, whose curvature pulls up from $\check{\varSigma}$
and whose fiberwise holonomy is trivial. 

Thus, this is why the bundle $\pi^{*}(\check{E})$ ``remembers''
the choice of connection we made downstairs (i.e, on the orbifold
$\check{\varSigma}$). Since there are several $U(2)$ orbifold bundles
with the same $\det\check{E}$ (up to isomorphism), we need to consider
all the distinct choices, which differ in their isotropy data. But
at least there are only finitely many isomorphism classes to consider.

b) One could also be interested in understanding how the abelian vortices
that appear in these moduli spaces are related to those appearing
in MOY \cite{MR1611061}. As Theorem 1 of MOY states, the Seiberg-Witten
solutions that appear there can be identified with two copies of the
space of effective orbifold divisors over $\varSigma$ with orbifold
degree less than $-\frac{\chi(\varSigma)}{2}=\frac{c_{1}(K_{\varSigma})}{2}$.
For simplicity suppose that the $a_{i}$ are coprime and that $\check{\varSigma}$
admits orbi-spin bundles \cite[Definition 5.13]{MR1611061}. 

This means that there is a square-root $K_{\check{\varSigma}}^{1/2}$
of $K_{\check{\varSigma}}$ in the sense that $2c_{1}(K_{\check{\varSigma}}^{1/2})=c_{1}(K_{\check{\varSigma}})$.
In our setup this would require all the $a_{i}$ to be odd integers
and moreover $K_{\check{\varSigma}}^{1/2}$ is unique (see the proof
of Corollary 5.17 in MOY). Thus, if we take for example $\det\check{E}\simeq K_{\check{\varSigma}}\otimes\check{L}_{0}$
this is a valid choice in the sense that $K_{\check{\varSigma}}$
is an even power of $\check{L}_{0}$ thus $\det\check{E}$ is an odd
power of $\check{L}_{0}$. However, this choice for $\det\check{E}$
is not within the vanishing condition stated in our theorem thus we
obtain:

1) Moduli spaces of $SO(3)$ vortices on $\check{E}'$: now abelian
vortices can arise provided $c_{1}(\check{L})<\frac{1}{2}c_{1}(K_{\check{\varSigma}}\otimes\check{L}_{0})=\frac{1}{2}c_{1}(K_{\check{\varSigma}})+\frac{1}{2a_{1}\cdots a_{n}}$
.

2) Moduli spaces of $SO(3)$ vortices on $K_{\check{\varSigma}}\otimes\check{E}'^{-1}$:
now abelian vortices can arise provided $c_{1}(\check{L})<\frac{1}{2}c_{1}(K_{\check{\varSigma}})-\frac{1}{2a_{1}\cdots a_{n}}$
. 

This seems the closest one can get to the constraint that appears
on MOY. However, because of this choice for $\det\check{E}$ it is
not clear that the moduli space of $SO(3)$ vortices is smooth at
the reducibles so it seems better to assume $c_{1}(\check{E})>2c_{1}(K_{\check{\varSigma}})$
instead.
\end{rem}

\begin{proof}
We follow the strategy of \cite[Theorem 4]{MR1611061}. If we start
with the equation $D_{S,B}\varPsi=0$, and apply $D_{S,B}$ to it,
we obtain
\[
0=\left(\rho(\eta)\cdot\nabla_{\frac{\partial}{\partial\varphi}}^{S,B}\right)^{2}\varPsi+D_{2,S,B}^{2}\varPsi+\rho(\eta)\rho(e^{1})F_{S,B}(\partial_{\varphi},e_{1})\varPsi+\rho(\eta)\rho(e^{2})F_{S,B}(\partial_{\varphi},e_{2})\varPsi
\]
Now take inner product with $\varPsi$ and integrate over $Y$ in
order to obtain 
\begin{equation}
0=\left|\nabla_{\partial_{\varphi}}^{S}\varPhi\right|^{2}+|D_{2}\varPhi|^{2}+\frac{1}{2}\left\langle \rho(\eta)\rho(e^{1})F_{\eta,e^{1}}\varPsi+\rho(\eta)\rho(e^{2})F_{\eta,e^{2}}\varPsi,\varPsi\right\rangle \label{eq:vanishing Seifert}
\end{equation}
where $F_{\eta,e^{1}}=F_{S,B}(\partial_{\varphi},e_{1})$ and $F_{\eta,e^{2}}=F_{S,B}(\partial_{\varphi},e_{2})$.

Under a slight abuse of notation, recall from last section that 
\begin{align*}
\rho(\eta)=\left(\begin{array}{cc}
i & 0\\
0 & -i
\end{array}\right) &  & \rho(e^{1})= & \left(\begin{array}{cc}
0 & -\epsilon\\
\bar{\epsilon} & 0
\end{array}\right) & \rho(e^{2})=\left(\begin{array}{cc}
0 & i\epsilon\\
i\bar{\epsilon} & 0
\end{array}\right)
\end{align*}
Thus 
\begin{align*}
\rho(\eta)\rho(e^{1})=\left(\begin{array}{cc}
0 & -i\epsilon\\
-i\bar{\epsilon} & 0
\end{array}\right) &  & \rho(\eta)\rho(e^{2})=\left(\begin{array}{cc}
0 & -\epsilon\\
\bar{\epsilon} & 0
\end{array}\right)
\end{align*}
Recall that if we were to write $\varPsi$ locally as $\varPsi=\sum_{i,j=1}^{2}c_{ij}\varPsi_{i}^{S}\otimes\varPsi_{j}^{E}$,
where $\varPsi_{1}^{S},\varPsi_{2}^{S}$ is an orthonormal basis for
$S$ and $\varPsi_{1}^{E},\varPsi_{2}^{E}$ an orthonormal basis for
$E$, then for any one form $\rho(\theta)\varPsi=\sum_{i,j=1}^{2}c_{ij}\left(\rho(\theta)\varPsi_{i}^{S}\right)\otimes\varPsi_{j}^{E}$
, while for any section $\xi$ of $\mathfrak{su}(E)$, $\xi\varPsi=\sum_{i,j=1}^{2}c_{ij}\varPsi_{i}^{S}\otimes(\xi\varPsi_{j}^{E})$.
In other words, the Clifford multiplication action on spinors $\varPsi$
is uncoupled from the action of $\mathfrak{su}(E)$ on spinors. 

Therefore, we can write 
\begin{align*}
 & \rho(\eta)\rho(e^{1})F_{\eta,e^{1}}\varPsi &  & \rho(\eta)\rho(e^{2})F_{\eta,e^{2}}\varPsi\\
= & \left(\begin{array}{cc}
0 & -i\epsilon\\
-i\bar{\epsilon} & 0
\end{array}\right)F_{\eta,e^{1}}\left(\begin{array}{c}
\alpha\\
\beta
\end{array}\right) & = & \left(\begin{array}{cc}
0 & -\epsilon\\
\bar{\epsilon} & 0
\end{array}\right)F_{\eta,e^{2}}\left(\begin{array}{c}
\alpha\\
\beta
\end{array}\right)\\
= & \left(\begin{array}{cc}
0 & -i\epsilon\\
-i\bar{\epsilon} & 0
\end{array}\right)\left(\begin{array}{c}
F_{\eta,e^{1}}\alpha\\
F_{\eta,e^{1}}\beta
\end{array}\right) & = & \left(\begin{array}{cc}
0 & -\epsilon\\
\bar{\epsilon} & 0
\end{array}\right)\left(\begin{array}{c}
F_{\eta,e^{2}}\alpha\\
F_{\eta,e^{2}}\beta
\end{array}\right)\\
= & \left(\begin{array}{c}
-i\epsilon F_{\eta,e^{1}}\beta\\
-i\bar{\epsilon}F_{\eta,e^{1}}\alpha
\end{array}\right) & = & \left(\begin{array}{c}
-\epsilon F_{\eta,e^{2}}\beta\\
\bar{\epsilon}F_{\eta,e^{2}}\alpha
\end{array}\right)
\end{align*}
And so we must understand 
\begin{align*}
 & \left\langle \left(\begin{array}{c}
-\epsilon(iF_{\eta,.e^{1}}+F_{\eta,e^{2}})\beta\\
-\bar{\epsilon}(iF_{\eta,e^{1}}-F_{\eta,e^{2}})\alpha
\end{array}\right),\left(\begin{array}{c}
\alpha\\
\beta
\end{array}\right)\right\rangle \\
= & 2\left\langle \left(\begin{array}{c}
(\alpha\otimes\beta^{*})_{0}\beta\\
(\beta\otimes\alpha^{*})_{0}\alpha
\end{array}\right),\left(\begin{array}{c}
\alpha\\
\beta
\end{array}\right)\right\rangle \\
= & 2\left\langle \left(\begin{array}{c}
|\beta|^{2}\alpha-\frac{1}{2}<\beta,\alpha>_{E}\beta\\
|\alpha|^{2}\beta-\frac{1}{2}<\alpha,\beta>_{E}\alpha
\end{array}\right),\left(\begin{array}{c}
\alpha\\
\beta
\end{array}\right)\right\rangle \\
= & 2\left(|\beta|^{2}|\alpha|^{2}-\frac{1}{2}|<\beta,\alpha>_{E}|^{2}+|\beta|^{2}|\alpha|^{2}-\frac{1}{2}|<\alpha,\beta>_{E}|^{2}\right)
\end{align*}
 Just as in the case of $S^{1}\times\varSigma$, by Cauchy-Schwarz
this is non-negative so going back to the equality (\ref{eq:vanishing Seifert})
we conclude that 
\[
\begin{cases}
\nabla_{\partial_{\varphi}}^{S}\alpha\equiv0\\
\nabla_{\partial_{\varphi}}^{S}\beta\equiv0\\
D_{2,S,B}\varPsi\equiv0\\
\alpha\equiv0\text{ or }\beta\equiv0
\end{cases}
\]
Once we know this, the argument is identical to the one given by \cite[Theorem 4]{MR1611061}. 

Namely, the first two equations say that $\alpha,\beta$ actually
pullback from the orbifold Riemann surface. Likewise, the connection
$B$ will pullback from an orbifold bundle $\check{E}'$ over $\check{\varSigma}$
which satisfies $\det\check{E}'\simeq\det\check{E}$ . Finally, the
operator $D_{2}$ can then be identified with a twisted $\bar{\partial}\oplus\bar{\partial}^{*}$
operator \cite[p. 709]{MR1611061}, so this is why $D_{2,S,B}\varPsi\equiv0$
together with the equation for $F_{e^{1},e^{2}}$ gives a solution
means that we have a solution to the $SO(3)$ vortex equations. 

The vanishing result uses and the identification with the stable pairs
on $\check{\varSigma}$ is exactly the same argument as the one we
used for $S^{1}\times\varSigma$.
\end{proof}
\begin{example}
The Poincaré Homology Sphere

As is well know the Poincaré homology sphere can be identified with
the Brieskorn homology sphere $\varSigma(2,3,5)$, and as a Seifert
manifold, the corresponding orbifold is $S^{2}(2,3,5)$. 

In this case the generator for the topological Picard group is an
orbifold line bundle with $c_{1}(\check{L}_{0})=\frac{1}{30}$. The
isotropy invariants consists of three integers $b_{1},b_{2},b_{3}$
satisfying 
\begin{align*}
1\leq b_{1}<2 &  & 1\leq b_{2}<3 &  & 1\leq b_{3}<5
\end{align*}
These numbers are found by the requirement that 
\[
\frac{1}{\prod\alpha_{i}}-\sum_{i=1}^{n}\frac{\beta_{i}}{\alpha_{i}}=\frac{1}{30}-\frac{15b_{1}+10b_{2}+6b_{3}}{30}
\]
be an integer (the background degree of $\check{L}_{0}$) . It is
easy to check that the only choice which works is 
\[
b_{1}=b_{2}=b_{3}=1
\]
Thus we will write $\check{L}_{0}(1,1,1)$ to emphasize the isotropy
data. 

Recall that the Seifert invariants for the canonical bundle were given
in equation (\ref{eq:Seifert Canonical}). Thus, our non-vanishing
condition reads in this case 
\[
c_{1}(\check{E})>2c_{1}(K_{\check{\varSigma}})=2\left(0-2+3-\frac{1}{2}-\frac{1}{3}-\frac{1}{5}\right)=2\left(\frac{30-15-10-6}{30}\right)=-\frac{1}{15}
\]
In particular, it suffices to take $\det\check{E}\simeq\check{L}_{0}$.
Now suppose that $\check{E}'$ is a $U(2)$ bundle with $\det\check{E}'\simeq\check{L}_{0}$.
Write the isotropy invariants of $\check{E}'$ as $((b_{1}^{-},b_{1}^{+}),(b_{2}^{-},b_{2}^{+}),(b_{3}^{-},b_{3}^{+}))$.
The conditions these integers must satisfy are 
\[
0\leq b_{i}^{-}\leq b_{i}^{+}<a_{i}
\]
and also 
\[
b_{i}^{-}+b_{i}^{+}\equiv1\mod a_{i}
\]
since they must give the isotropy data of $\check{L}_{0}$.

We find by trial and error that the only options which work are

\begin{align*}
((0,1),(0,1),(0,1)) &  & ((0,1),(0,1),(2,4)) &  & ((0,1),(0,1),(3,3))\\
((0,1),(2,2),(0,1)) &  & ((0,1),(2,2),(2,4)) &  & ((0,1),(2,4),(3,3))
\end{align*}
In order to analyze these bundles in a systematic way we recall the
Facts (\ref{FACTS}), as well as some other important things to keep
in mind when analyzing the different moduli spaces:
\end{example}

\begin{enumerate}
\item Because $g=0$, $n=3$, if $n_{0}=\#\{i\mid b_{i}^{-}=b_{i}^{+}\}\geq1$,
then $\check{E}'$ admits no projectively flat connections. Moreover,
the space of projectively flat connections is connected and for three
marked points of expected dimension zero (this can be seen from the
proof of our Lemma (\ref{lem:Suppose-that-smoothness criterion})). 
\item If the moduli space $\mathcal{M}(\check{\varSigma},\check{E}')$ admits
no moduli space of abelian vortices, then there are no irreducible
vortices inside this moduli space.
\item The abelian vortices which can arise must satisfy $c_{1}(\check{L})<\frac{1}{2}c_{1}(\check{L}_{0})=\frac{1}{2}\det(\check{E})$.
Since $\check{L}=\check{L}_{0}^{l}$ for some integer $l$ and $l\geq0$
because otherwise $H^{0}(\check{L})$ vanishes, we must have $\check{L}$
is the trivial line bundle (isotropy $(0,0,0)$ and $c_{1}(\check{L})=0$).
Thus in this case any isotropy data which does not contain a $0$
in each of the three pairs $(b_{i}^{-},b_{i}^{+})$ will have an empty
moduli space of irreducible $SO(3)$ vortices.
\item Moreover, the dimension of the moduli space of irreducible $SO(3)$
vortices can be read from (\ref{dimension moduli space})
\[
2\left((g-1)+c_{1}(\det\check{E})+(n-n_{0})-\sum_{i=1}^{n}\frac{b_{i}^{-}+b_{i}^{+}}{a_{i}}\right)=2\left(-1+\frac{1}{30}+3-n_{0}-\sum_{i=1}^{n}\frac{b_{i}^{-}+b_{i}^{+}}{a_{i}}\right)
\]
\item The index for an abelian monopole will be according to equation (\ref{index})
\begin{align*}
 & 2\left(g-1+c_{1}(\det\check{E})-2c_{1}(\check{L})+\sum_{i\mid\epsilon_{i}=1}\frac{b_{i}^{+}-b_{i}^{-}}{a_{i}}+n_{-}+\sum_{i\mid\epsilon_{i}=-1}\frac{b_{i}^{-}-b_{i}^{+}}{a_{i}}\right)\\
= & 2\left[-1+\frac{1}{30}+\sum_{i\mid\epsilon_{i}=1}\frac{b_{i}^{+}-b_{i}^{-}}{a_{i}}+n_{-}+\sum_{i\mid\epsilon_{i}=-1}\frac{b_{i}^{-}-b_{i}^{+}}{a_{i}}\right]
\end{align*}
\end{enumerate}
With this information, we can make the following table {[}the dimensions
refer to those before we take the quotient by the circle action{]}:

\ 

\begin{tabular}{|c|c|c|c|}
\hline 
Isotropy & $\mathcal{M}^{*}(\check{\varSigma},\check{E}')$  & \# Proj. Flat & \# Abelian Vortices\tabularnewline
\hline 
\hline 
$((0,1),(0,1),(0,1))$ & non empty, of dim $2$ & one  & one vortex of index $2$\tabularnewline
\hline 
$((0,1),(0,1),(2,4))$ & empty, exp. dim$=$$0$ & one (isolated) & no vortices (wrong isotropy)\tabularnewline
\hline 
$((0,1),(0,1),(3,3))$ & empty, exp. dim$=$$-2$ & empty since $n_{0}\geq1$ & no vortices (wrong isotropy)\tabularnewline
\hline 
$((0,1),(2,2),(0,1))$ & empty, exp. dim$=$$-2$ & empty since $n_{0}\geq1$ & no vortices (wrong isotropy)\tabularnewline
\hline 
$((0,1),(2,2),(2,4))$ & empty, exp. dim$=$$-4$ & empty since $n_{0}\geq1$ & no vortices (wrong isotropy)\tabularnewline
\hline 
$((0,1),(2,4),(3,3))$ & empty, exp. dim$=$$-4$ & empty since $n_{0}\geq1$ & no vortices (wrong isotropy)\tabularnewline
\hline 
\end{tabular}

\ 

Therefore, we conclude that the solutions to the $SO(3)$ monopole
equations on the Poincaré homology sphere consist of: a) the trivial
$SU(2)$ connection, b) two irreducible flat $SU(2)$ connections,
c) one Seiberg-Witten monopole, d) a one dimensional moduli space
of $SO(3)$ monopoles which serves as a ``cobordism'' between the
Seiberg-Witten monopole and one of the irreducible flat connections. 

A more realistic example to consider is the following:
\begin{example}
The Brieskorn Homology Sphere $\varSigma(2,3,7)$

In this case we are looking at $S^{2}(2,3,7)$. The generator $\check{L}_{0}$
now satisfies $c_{1}(\check{L}_{0})=\frac{1}{42}$ and its isotropy
must satisfy that 
\[
\frac{1}{42}-\frac{21b_{1}+14b_{2}+6b_{3}}{42}
\]
is an integer. By trial and error we find that the only choice is
\[
b_{1}=1,b_{2}=2,b_{3}=6
\]
Write $\check{L}_{0}(1,2,6)$. Now the vanishing condition reads 
\[
c_{1}(\check{E})>2c_{1}(K_{\check{\varSigma}})=2\left(0-2+3-\frac{1}{2}-\frac{1}{3}-\frac{1}{7}\right)=\frac{1}{21}
\]
So we need to look at the moduli spaces of $SO(3)$ vortices with
$\det\check{E}\simeq\check{L}_{0}^{3}$. Notice that the isotropy
is now $\check{L}_{0}^{3}(1,0,4)$. In this case the expected dimension
would need to be 
\[
2\left(2+\frac{1}{14}-\frac{b_{1}^{-}+b_{1}^{+}}{2}-\frac{b_{2}^{-}+b_{2}^{+}}{3}-\frac{b_{3}^{-}+b_{3}^{+}}{7}\right)
\]
Notice that in this case $b_{1}^{-}+b_{1}^{+}=1$ and $b_{2}^{-}+b_{2}^{+}=3$
so this forces the expected dimension to be 
\[
2\left(\frac{1}{2}+\frac{1}{14}-\frac{b_{3}^{-}+b_{3}^{+}}{7}\right)=1+\frac{1-2b_{3}^{-}-2b_{3}^{+}}{7}
\]
and since one of the $b_{3}^{\pm}$ must be non-trivial, this forces
the dimension after taking the quotient by the $S^{1}$ action to
be negative, thus we end up with empty moduli space of $SO(3)$ vortices,
which is not interesting.

Thus, we need to increase $c_{1}(\det\check{E})$ in order to make
the moduli spaces of $SO(3)$ vortices to have positive expected dimension.
The next choice is thus $\det\check{E}\simeq\check{L}_{0}^{5}$. Notice
that the isotropy is now $\check{L}_{0}^{5}(1,1,2)$. In this case
the formula for the expected dimension reads
\[
2\left(2+\frac{5}{42}-\frac{b_{1}^{-}+b_{1}^{+}}{2}-\frac{b_{2}^{-}+b_{2}^{+}}{3}-\frac{b_{3}^{-}+b_{3}^{+}}{7}\right)
\]
Again, $b_{1}^{-}+b_{1}^{+}=1$, but now we can have $b_{2}^{-}+b_{2}^{+}=1$
and $b_{3}^{-}+b_{3}^{+}=2$ in which case the dimension would be
\[
2\left(2+\frac{5}{42}-\frac{1}{2}-\frac{1}{3}-\frac{2}{7}\right)=2
\]
so that means we will find positive dimensional moduli spaces of $SO(3)$
vortices!

The isotropy $((b_{1}^{-},b_{1}^{+}),(b_{2}^{-},b_{2}^{+}),(b_{3}^{-},b_{3}^{+}))$
of our $U(2)$ bundles must satisfy
\[
\begin{cases}
0\leq b_{1}^{-}\leq b_{1}^{+}\leq1 & b_{1}^{-}+b_{1}^{+}-1\equiv0\mod2\\
0\leq b_{2}^{-}\leq b_{2}^{+}\leq1 & b_{2}^{-}+b_{2}^{+}-1\equiv0\mod3\\
0\leq b_{3}^{-}\leq b_{3}^{+}\leq2 & b_{3}^{-}+b_{3}^{+}-2\equiv0\mod7
\end{cases}
\]
Again after trial and error one finds the isotropies 
\begin{align*}
((0,1),(0,1),(0,2)) &  & ((0,1),(0,1),(1,1)) &  & ((0,1),(0,1),(3,6)) &  &  & ((0,1),(0,1),(4,5))\\
((0,1),(2,2),(0,2)) &  & ((0,1),(2,2),(1,1)) &  & ((0,1),(2,2),(3,6)) &  &  & ((0,1),(2,2),(4,5))
\end{align*}

In addition to the facts used for the case of the Poincaré homology
sphere, notice that:
\end{example}

\begin{enumerate}
\item The moduli spaces of abelian vortices which can arise are associated
to $\check{L}_{0}(1,2,6)$, $\check{L}_{0}^{2}(0,1,5)$ and the trivial
one $\check{L}_{triv}$. The background degrees are 
\[
\begin{cases}
\deg_{B}\check{L}_{triv}=0\\
\deg_{B}\check{L}_{0}(1,2,6)=\frac{1}{42}-\frac{1}{2}-\frac{2}{3}-\frac{6}{7}=-2\\
\deg_{B}\check{L}_{0}^{2}(0,1,5)=\frac{2}{42}-\frac{0}{2}-\frac{1}{3}-\frac{5}{7}=-1
\end{cases}
\]
so the space of abelian vortices to consider is really just $\check{L}_{triv}$,
since recall from MOY that these moduli spaces of abelian vortices
are isomorphic to $\text{Sym}^{\deg_{B}\check{L}}(\varSigma)$, thus
they are empty if the background degree is negative. 
\item From Facts (\ref{FACTS}), the space of irreducible projectively flat
connections is empty if and only if there exists a vector $(\epsilon_{i})$,
with $\epsilon_{i}=\pm1$, such that $n_{+}+\deg_{B}(\det\check{E})\equiv1\mod2$
and $n_{+}-\sum_{i=1}^{n}\frac{\epsilon_{i}(b_{i}^{+}-b_{i}^{-})}{a_{i}}<1-g$.
Here $n_{\pm}=\#\{i\mid\epsilon_{i}=\pm1\}$.
\item Moreover, the dimension of the moduli space of irreducible $SO(3)$
vortices can be read from (\ref{dimension moduli space})
\[
2\left((g-1)+c_{1}(\det\check{E})+(n-n_{0})-\sum_{i=1}^{n}\frac{b_{i}^{-}+b_{i}^{+}}{a_{i}}\right)=2\left(-1+\frac{5}{42}+3-n_{0}-\sum_{i=1}^{n}\frac{b_{i}^{-}+b_{i}^{+}}{a_{i}}\right)
\]
\item The index for an abelian monopole will be according to equation (\ref{index})
\begin{align*}
 & 2\left(g-1+c_{1}(\det\check{E})-2c_{1}(\check{L})+\sum_{i\mid\epsilon_{i}=1}\frac{b_{i}^{+}-b_{i}^{-}}{a_{i}}+n_{-}+\sum_{i\mid\epsilon_{i}=-1}\frac{b_{i}^{-}-b_{i}^{+}}{a_{i}}\right)\\
= & 2\left[-1+\frac{1}{42}+\sum_{i\mid\epsilon_{i}=1}\frac{b_{i}^{+}-b_{i}^{-}}{a_{i}}+n_{-}+\sum_{i\mid\epsilon_{i}=-1}\frac{b_{i}^{-}-b_{i}^{+}}{a_{i}}\right]
\end{align*}
With this information, we can make the following table {[}the dimensions
refer to those before we take the quotient by the circle action{]}:
\end{enumerate}
\ 

\begin{tabular}{|c|c|c|c|}
\hline 
Isotropy & $\mathcal{M}^{*}(\check{\varSigma},\check{E}')$ & \# Proj. Flat & \# Abelian Vortices\tabularnewline
\hline 
\hline 
$((0,1),(0,1),(0,2))$ & non empty, of dim $2$ & one & one vortex of index $2$\tabularnewline
\hline 
$((0,1),(0,1),(1,1))$ & empty, exp. dim$=0$ & empty since $n_{0}\geq1$ & no vortices (wrong isotropy)\tabularnewline
\hline 
$((0,1),(0,1),(3,6))$ & empty, exp. dim$=0$ & one & no vortices (wrong isotropy)\tabularnewline
\hline 
$((0,1),(0,1),(4,5))$ & empty, exp. dim$=$$0$ & empty (see below) & no vortices (wrong isotropy)\tabularnewline
\hline 
$((0,1),(2,2),(0,2))$ & empty, exp. dim$=$$-2$ & empty since $n_{0}\geq1$ & no vortices (wrong isotropy)\tabularnewline
\hline 
$((0,1),(2,2),(1,1))$ & empty, exp. dim$=$$-4$ & empty since $n_{0}\geq1$ & no vortices (wrong isotropy)\tabularnewline
\hline 
$((0,1),(2,2),(3,6))$ & empty, exp. dim$=$$-4$ & empty since $n_{0}\geq1$ & no vortices (wrong isotropy)\tabularnewline
\hline 
$((0,1),(2,2),(4,5))$ & empty, exp. dim$=$$-4$ & empty since $n_{0}\geq1$ & no vortices (wrong isotropy)\tabularnewline
\hline 
\end{tabular}

\ 

The reason why the bundle with isotropy $((0,1),(0,1),(4,5))$ is
empty is that we can take the isotropy vector $\epsilon=(1,-1,-1)$,
so $n_{+}=1$ and $\frac{1}{2}-\frac{1}{3}-\frac{1}{7}>0$ which verifies
the criterion mentioned before for the emptiness of the space of irreducible
projectively flat connections.

\ 

Thus one more time our picture is similar to the case of the Poincaré
homology sphere. On $\varSigma(2,3,7)$ we have the trivial connection,
two irreducible flat connections, one Seiberg-Witten monopole, and
a ``cobordism'' of $SO(3)$ monopoles between this Seiberg-Witten
monopole and one of the irreducible flat connections.

\textbf{}

\textbf{}

\bibliographystyle{plain}
\bibliography{amsj,../ARXIV/references}

\begin{quote}
Department of Mathematics, Rutgers University.

\textit{\small{}E-mail address}{\small{}:} \textsf{\footnotesize{}mariano.echeverria@rutgers.edu}{\footnotesize\par}
\end{quote}

\end{document}